\documentclass[11pt]{amsart}

\usepackage[latin1]
{inputenc}
\usepackage[T1]{fontenc}
\usepackage{float}
\usepackage[dvips,pdftex]{graphicx}
\usepackage{amssymb}
\usepackage{amsfonts}
\usepackage{setspace}
\usepackage{amsmath}
\usepackage{dsfont}
\usepackage{amsthm}
\usepackage{youngtab}
\usepackage{tikz}

\newtheorem{satz}{Theorem}
\newtheorem{epr}[satz]{Proposition}
\newtheorem{defi}[satz]{Definition}
\newtheorem{elem}[satz]{Lemma}
\newtheorem{ecor}[satz]{Corollary}
\newtheorem{rem}[satz]{Remark}
\begin{document}

\pagestyle{empty}

\title[Hall-Littlewood polynomials] {Generalization of the Macdonald formula for Hall-Littlewood polynomials} 

\author{Inka Klostermann}

\thanks{This research was supported by the DFG Priority program Darstellungstheorie 1388.}

\address{Inka Klostermann:\newline Mathematisches Institut, Universit\"at zu K\"oln,\newline Weyertal 86-90, D-50931 K\"oln,Germany}

\email{ikloster@math.uni-koeln.de}

\begin{abstract}
We study the Gaussent-Littelmann formula for Hall-Littlewood polynomials and we develop combinatorial tools to describe the formula in a purely combinatorial way for type $A_n$, $B_n$ and $C_n$. This description is in terms of Young tableaux and arises from identifying one-skeleton galleries that appear in the Gaussent-Littelmann formula with Young tableaux. Furthermore, we show by using these tools that the Gaussent-Littelmann formula and the well-known Macdonald formula for Hall-Littlewood polynomials for type $A_n$ are the same.
\end{abstract}

\maketitle

\section{Introduction}
The symmetric Hall-Littlewood polynomials $P_{\lambda}(x,q)$ have an intrinsic meaning in combinatorial representation theory generalizing other important families of symmetric functions i.e. the monomial symmetric functions and the Schur functions. Originally P. Hall defined the Hall-Littlewood polynomials for type $A_n$ as a family of symmetric functions associated to certain elements in the Hall algebra. Later, Littlewood defined them explicitly in terms of the Weyl group $W$ and a coweight lattice $X \check{}$ for type $A_n$ \cite{Li}. This formula led to defining Hall-Littlewood polynomials of arbitrary type by replacing $W$ and $X\check{}$ in Littlewood's definition by a Weyl group and a coweight lattice of arbitrary type. These polynomials coincide with the so-called Macdonald spherical functions \cite{Mac2}, thus both names appear in the literature denoting the same objects.\\
There are various explicit combinatorial formulas for the Hall-Littlewood polynomials proven by Gaussent-Littelmann,  Macdonald, Lenart, Schwer, Haiman-Haglund-Loehr \cite{GL1} \cite{Mac1},\cite{L1}, \cite{L2}, \cite{S}, \cite{HHL} to name only a few. The first and probably most famous combinatorial formula, the Macdonald formula, is exclusively for type $A_n$. This formula is in terms of Young tableaux of type $A_n$. Most recently, Gaussent-Littelmann developed a formula for Hall-Littlewood polynomials for arbitrary type as sum over positively folded one-skeleton galleries in the standard apartment of the affine building. This formula has a geometric background which relates it closely to the Schwer formula which is a sum over positively folded alcove galleries in the standard apartment of the affine building. Let us explain the geometric background and their connection more precisely:\\
Express a given Hall-Littlewood polynomial $P_{\lambda}(x,q)$ of arbitrary type in the monomial basis $\{{m_{\mu}}(x)\}_{\mu \in X \check{}_+}$:
$$ P_{\lambda}(x,q)=\sum_{\mu \in X\check{}_+}q^{-\left\langle \lambda+\mu, \rho \right \rangle}L_{\lambda,\mu}(q)m_{\mu}(x)$$
with $L_{\lambda, \mu}(q) \in \mathbb{Z}[q]$.\\
The Satake isomorphism yields that the Laurent polynomial $L_{\lambda, \mu}(q)$ can be calculated by counting points in a certain intersection of orbits in an affine Grassmannian depending on the coweights $\lambda$ and $\mu$ over a finite field $\mathbb{F}_q$. Both, Schwer and Gaussent-Littelmann use this approach by describing the elements in this intersection with galleries in the standard apartment of the affine building, namely Gaussent-Littelmann use one-skeleton galleries whereas Schwer uses alcove galleries. In geometric terms using different types of galleries results from choosing different Bott-Samelson type varieties. Gaussent-Littelmann refer to this connection between the formulas as "geometric compression". One major advantage of using one-skeleton galleries instead of alcove galleries is that there is a one-to-one correspondence between the positively folded one-skeleton galleries of type $\lambda$ and target $\mu$ for some dominant coweights $\lambda$ and $\mu$ and the semistandard Young tableaux of shape $\lambda$ and content $\mu$, for classical types. This correspondence leads to the question whether it is possible to calculate the contribution to the Gaussent-Littelmann formula of a positively folded one-skeleton gallery $\delta$ directly from the associated semistandard Young tableau $T_{\delta}$. In the first part of this article we give a positive answer to this question by developing the so-called combinatorial Gaussent-Littelmann formula. The key ingredient for the proof of this formula is a recurrence for a certain set of positively folded galleries of chambers in the standard apartment of the residue building that appears in the Gaussent-Littelmann formula. \\ 
It turns out that the Macdonald formula and the combinatorial Gaussent-Littelmann formula coincide for type $A_n$. In fact, the Macdonald formula is a closed formula for the recursively defined combinatorial Gaussent-Littelmann formula. The aim of the second part of this article is to explain and prove this equality. Apparently, the first indicator for the equality of the two formulas is that they are both sums over semistandard Young tableaux. Further, in the combinatorial Gaussent-Littelmann formula the contribution of a semistandard Young tableau is a product of contributions of the columns of the tableau. These contributions only depend on the column itself and, if existing, on the column to the right. Reformulating the Macdonald formula reveals this property in the formula, too, although it is not at all obvious at first glance. We prove the equality of the two formulas by showing that the contribution of every column is the same.\\
Since the Macdonald formula is valid only for type $A_n$ the formula of Gaussent-Littelmann generalizes the Macdonald formula and provides it with a geometric background.\\
This article is organized as follows:\\
In the second section we start by revisiting all basics regarding building theory that we need in order to compute the Gaussent-Littelmann formula. In the third section we state this formula. The combinatorial version of it for type $A_n$, $B_n$ and $C_n$ is developed in the fourth section including a detailed description of the respective Young tableaux. In the fifth section we present the Macdonald formula and the sixth section is devoted to proving the equality of both formulas for type $A_n$.
\section{Basics}
In this section we recall some basic notation, definitions and facts. Let $(X, \phi, X{\check{}}, \phi{\check{}})$ be a reduced root datum where $\left\langle.,.\right\rangle$ denotes the non-degenerate pairing between $X$ and $X\check{}$. Let $W$ be the Weyl group and $l(\cdot)$ denote the length function on $W$. Further, let $\Delta= \{\alpha_1, \dots, \alpha_n\}$ be a fixed choice of simple roots and $\phi^+$ the set of positive roots with respect to $\Delta$, $\rho$ is, as usual, half the sum of all positive roots. Let 
$$ X{\check{}}_+=\{ \lambda \in X\check{} \; \vert \; \left\langle \lambda, \alpha \right\rangle \geq 0 \; \text{for all} \; \alpha \in \phi^+  \}$$ be the set of dominant coweights.
\subsection{Hall-Littlewood polynomials}
Consider the group algebra $R[X\check{}]$ with coefficients in some ring $R$. Let $\{ \epsilon_1,\dots,\epsilon_n \}$ be a $\mathbb{Z}$-basis of unit vectors for $X{\check{}}$. By sending $\epsilon_i$ to $x_i$ for every $i$ we can identify this group algebra with the polynomial ring in $n$ variables over the ring $R$. In the following we identify a given coweight $\mu= \mu_1\epsilon_1+\dots+\mu_n\epsilon_n\in X{\check{}}$ with the monomial $x^{\mu}:= x_1^{\mu_1}*\dots*x_n^{\mu_n}$. The Weyl group $W$ acts naturally on this algebra thus we can consider $R[X \check{}]^W$ which is the algebra of invariants under this action. This algebra is also known as the algebra of symmetric polynomials due to the above identification. There are several classical bases known for the algebra of symmetric polynomials. The one we need is the monomial basis $\{m_{\lambda}(x)\}_{\lambda \in X\check{}_+}$ consisting of the monomial symmetric functions $m_{\lambda}(x)$ for $\lambda \in X\check{}_+$. They are defined as follows:
$$ m_{\lambda}(x):= \sum_{w \in W/W_{\lambda}}x^{w(\lambda)}, $$
where $W_{\lambda}$ is the stabilizer of $\lambda$ in $W$.\\
Now let $R= \mathbb{Z}[q,q^{-1}]$ be the ring of Laurent polynomials with coefficients in $\mathbb{Z}$. The Hall-Littlewood polynomials $\{P_{\lambda}(x,q)\}_{\lambda \in X\check{}_+}$ form a basis for $\mathbb{Z}[q,q^{-1}][X]^{W}$. They are defined as follows:
$$ P_{\lambda}(x,q)= \frac{1}{W_{\lambda}(q^{-1})}\sum_{w \in W}w(x^{\lambda}\prod_{\alpha \in \phi^+}\frac{1-q^{-1}x^{-\alpha\check{}}}{1-x^{-\alpha\check{}}}),$$
where $W_{\lambda}(q^{-1})= \sum_{w \in W_{\lambda}}q^{-l(w)}$.\\
Expanding the Hall-Littlewood polynomials $\{P_{\lambda}(x,q)\}_{\lambda \in X\check{}_+}$ in terms of the monomial basis $\{m_{\mu}(x)\}_{\mu \in X\check{}_+}$ leads to Laurent polynomials $L_{\lambda, \mu}(q)$:
$$ P_{\lambda}(x,q)= \sum_{\mu \in X\check{}_+}q^{-{\left \langle \lambda+\mu, \rho \right \rangle}}L_{\lambda,\mu}(q)m_{\mu}(x). $$
The Gaussent-Littelmann formula describes exactly these Laurent polynomials $L_{\lambda, \mu}(q)$.
\subsection{Buildings}
In this section we recall that part of the theory of buildings which is essential for understanding the combinatorics in this article (see \cite{GL1}). For a detailed introduction to buildings we refer to Ronan's book \cite{R}.\\ \\
The Gaussent-Littelmann formula for Hall-Littlewood polynomials is in terms of combinatorial one-skeleton galleries. These objects are contained in a fixed standard apartment of the affine building.\\ \\
Let $\mathbb{A}:= X\check{}\otimes_{\mathbb{Z}}\mathbb{R}$ be the real span of the coweight lattice. For every pair $(\alpha,n)$ with $\alpha \in \phi^+$ and $n \in \mathbb{Z}$ we define the affine hyperplane 
$$H_{(\alpha,n)}:= \{ x \in \mathbb{A} \; \vert \; \left\langle x, \alpha \right\rangle = n \}.$$
Let $H^a=\{ H_{(\alpha,n)} \; \vert \; \alpha \in \phi^+\; , \; n \in \mathbb{Z}\}$ be the set of all affine hyperplanes.\\
The standard \textbf{apartment of the affine building} associated to the given root datum is the vector space $\mathbb{A}$ together with the hyperplane arrangement $H^a$.\\ \\Recall that the Weyl group $W$ is the subgroup of $GL(\mathbb{A})$ generated by all reflections $$s_{\alpha}(x)=x - \left\langle \alpha, x \right \rangle \alpha\check{}$$ at the hyperplanes $H_{(\alpha,0)}$ for $\alpha \in \phi$ and $x \in \mathbb{A}$. The reflections $s_i=s_{\alpha_i}$ for $i \in \{1, \dots, n\}$ at the hyperplanes $H_{(\alpha_i,0)}$ are called simple reflections. The set of simple reflections already generates $W$. We define the affine Weyl group $W^a$ to be the subgroup of the affine transformations of $\mathbb{A}$ generated by all reflections at the hyperplanes $H^a$. We denote the reflection at the hyperplane $H_{(\alpha,n)}$ by $s_{(\alpha,n)}$.\\ \\Let $$H_{(\alpha,n)}^+=\{ x \in \mathbb{A} \; \vert  \; \left \langle \alpha, x \right \rangle \geq n \}$$ be the positive closed half-space and $$H_{(\alpha,n)}^-=\{ x \in \mathbb{A} \; \vert  \; \left \langle \alpha, x \right \rangle \leq n \}$$ be the negative closed half-space corresponding to $(\alpha,n)$.\\
A \textbf{face} $F$ in $\mathbb{A}$ is a subset of $\mathbb{A}$ of the following form:
$$ F =  \bigcap_{(\alpha, n) \in \phi^+\times\mathbb{Z}} H^{e_{(\alpha,n)}}_{(\alpha,n)},$$
where $e_{(\alpha,n)} \in \{+,-,\emptyset\}$ and $H^{\emptyset}_{(\alpha,n)}=H_{(\alpha,n)}$. By the corresponding {\bf open face} $F^{°}$ we mean the subset of $F$ obtained by replacing the closed affine half-spaces in the definition of $F$ by the corresponding open affine half-spaces. We call the affine span of a face $F$ the \textbf{support of $F$}, the \textbf{dimension of a face $F$} is the dimension of its support. We call a face of dimension zero a \textbf{vertex} and a face of dimension one an \textbf{edge}. The faces given by non-empty intersections of half-spaces are the faces of maximal dimension and are called \textbf{alcoves}. For a vertex $V$ let $\phi_V$ denote the subrootsystem of $\phi$ consisting of all roots $\alpha$ such that $H_{(\alpha,n)}$ contains $V$ for some $n\in \mathbb{Z}$. We call a vertex $V$ with $\phi_V=\phi$ a \textbf{special vertex}. The special vertices are precisely the coweights $X\check{}$. We denote the subgroup of $W$ consisting of all reflections $s_{\alpha}$ such that $\alpha \in \phi_{V}$ by $W_V^v$. For an arbitrary face $F$ in $\mathbb{A}$ let $W_F$ be the stabilizer of $F$ in $W$ and $W_F^a$ the stabilizer of $F$ in $W^a$. Note that $W^a_V$ and $W_V^v$ are isomorphic. An important face, the \textbf{fundamental alcove} $\Delta_f$ is defined as follows:
$$\Delta_f = \{ x \in \mathbb{A} \; \vert \; 0 \leq \left\langle \alpha, x \right\rangle \leq 1 \; \forall \alpha \in \phi^+ \}.$$
The fundamental alcove $\Delta_f$ is a fundamental domain for the action of $W^a$ on $\mathbb{A}$.\\ \\
The \textbf{type of a face} $F$ is defined as follows: Let $S^a$ be the subset of $(\phi^+ \times \mathbb{Z})$ with $\Delta_f \cap H_{(\alpha, n)}$ is a face of codimension one for $(\alpha,n)\in (\phi^+ \times \mathbb{Z})$. Let $F$ be a face of $\mathbb{A}$ contained in $\Delta_f$. The type of $F$ is defined as
$$ t(F)= \{(\alpha, n) \in S^a \; \vert \; F \subset H_{(\alpha,n)}\}.$$
Let now $F$ be an arbitrary face of $\mathbb{A}$. Then there is a unique face $F_f$ of $\mathbb{A}$ contained in $\Delta_f$ with the property that there exists an element $w \in W^a$ such that $w(F_f)=F$. We define the type of $F$ to be $t(F)=t(F_f)$.\\ \\
Let $\Omega$ be a subset of $\mathbb{A}$. We say that a hyperplane $H \in H^a$ \textbf{separates} $\Omega$ and a face $F$ of $\mathbb{A}$ if $\Omega$ is contained in $H^+$ or $H^-$ and $F^{°}$ is contained in the opposite open half space.\\ \\
The closures of the irreducible components of $\mathbb{A}\backslash \bigcup_{\alpha \in \phi}H_{(\alpha,0)}$ are called \textbf{chambers}. The chamber that contains $\Delta_F$ is called the \textbf{dominant chamber} and is denoted by $C^+$. This chamber is a fundamental domain for the action of $W$ on $\mathbb{A}$. By applying the longest Weyl group element $w_0 \in W$ to $C^+$ we obtain the so-called \textbf{anti-dominant chamber} $C^-$. This chamber is the unique chamber of $\mathbb{A}$ such that every hyperplane $H_{(\alpha_i,0)}$ separates $\alpha_i\check{}$ from $C^-$ for all simple roots $\alpha_i$. A \textbf{sector} $S$ at the vertex $V$ in $\mathbb{A}$ is a chamber $C$ translated by the vertex $V$ and the sector $-S$ at $V$ is the chamber $-C=\{-x  \; \vert \; x \in C\}$ translated by $V$. We can define an equivalence relation on the set of sectors as follows: Any two sectors are in the same equivalence class if their intersection is again a sector. We denote the equivalence class of a sector $S$ by $\overline{S}$. The set of all equivalence classes is in bijection with the Weyl group $W$ via the map that sends a sector $S$ to $w \in W$ with $\overline{w(C^+)}=\overline{S}$.\\ \\
The last object we need to introduce is the \textbf{ standard apartment of the residue building} at a vertex $V$ in $\mathbb{A}$. The standard apartment of the residue building at $V$ is the vector space $\mathbb{A}$ together with the subset of $H^a$ consisting of the affine hyperplanes that contain $V$. We refer to the standard apartment of the residue building as $\mathbb{A}_V$. The set of faces of $\mathbb{A}_V$ consists of all faces $F$ of $\mathbb{A}$ that contain $V$. We denote the corresponding face in $\mathbb{A}$ by $F_V$. Let $S$ be a sector at $V$ in $\mathbb{A}$. We associate to $S$ a face $S_V$ of $\mathbb{A}_V$ as follows: Let $\Delta \in \mathbb{A}$ be the unique alcove in the apartment of the affine building with $S^{°}\cap \Delta^{°} \neq \emptyset$. Then $S_V$ is defined to be $\Delta_V$.\\ \\
We can identify the faces of $\mathbb{A}_V$ as follows: Consider the set $R$ of objects in $\mathbb{A}$ of the form
$$  \underset{H_{(\alpha,n)}\ni V}{\bigcap_{(\alpha, n) \in \phi^+\times\mathbb{Z}}}H^{e_{(\alpha,n)}}_{(\alpha,n)},$$
where $e_{(\alpha,n)} \in \{+,-,\emptyset\}$ and $H^{\emptyset}_{(\alpha,n)}=H_{(\alpha,n)}$. We identify every face $F$ of $\mathbb{A}_V$ with the object in $R$ that contains $F$ and is the smallest with this property. Because of this identification we call a face $F_V$ in the standard apartment of the residue building at $V$ coming from an alcove $F$ in the apartment of the affine building a chamber.\\
Consider the sector $C^-+V$, the translated anti-dominant chamber in $\mathbb{A}$. We denote the corresponding chamber in the standard apartment of the residue building at $V$ by $C_V^-$. The chamber $C_V^-$ is a fundamental domain for the action of the subgroup $W_V^a$ of $W^a$ on $\mathbb{A}_V$. Let $\{ \beta_1, \dots, \beta_m\}$ be the set of simple roots for $\phi_V$ such that for every $(\beta_i,n_i)$ with $H_{(\beta_i,n_i)}$ contains $V$ the hyperplane $H_{(\beta_i,n_i)}$ separates $\beta_i\check{}+V$ from $C_V^-$. The reflections $s_{(\beta_i,n_i)}$ for $i \in \{1, \dots, m\}$ are the simple reflections of the Weyl group $W_V^a$. Since there will be no room for confusion we also denote the simple reflections of $W_V^a$ by $s_i=s_{(\beta_i,n_i)}$ for $i \in \{1, \dots, m\}$. With respect to this set of simple reflections $C_V^-$ is the anti-dominant chamber of $\mathbb{A}_V$. Let $F$ be a face of $\mathbb{A}_V$ in $C_V^-$. The\textbf{ type of $F$} is defined as $t(F)= \{ i \; \vert \; F \in H_{(\beta_i,n_i)}\}.$  Now let $F$ be an arbitrary face of $\mathbb{A}_V$, then there exists a unique face $F_f \in C_V^-$ and a $w \in W_V^a$ such that $F= w(F_f)$. We define the type of $F$ to be $t(F)=t(F_f)$. Note that we start by defining the type of the faces of the anti-dominant chamber $C_V^-$ and not of the dominant as in the definition of the type of a face in the standard apartment of the affine building.

\subsection{One-skeleton galleries}
In general, a gallery is a sequence of faces in the standard apartment of a building where a face is contained or contains the subsequent face.
\begin{defi}
A one-skeleton gallery in $\mathbb{A}$ is a sequence $\delta = ( V_0 \subset E_0 \supset V_1 \subset \cdots \supset V_{r+1})$ where
\begin{itemize}
     \item[(i)] $V_i$ for $i \in \{0, \dots, r+1\}$ is a vertex in $\mathbb{A}$ and
     \item[(ii)]$E_i$ for $i \in \{0, \dots, r\}$ is an edge in $\mathbb{A}$.
\end{itemize}
\end{defi}
We can concatenate two one-skeleton galleries $\delta = ( V_0 \subset E_0 \supset V_1 \subset \cdots \supset V_{r+1})$  and $\delta' = ( V_0' \subset E_0' \supset V_1' \subset \cdots \supset V_{r+1}')$ if $V_{r+1}=V_0'$ in $\mathbb{A}$ by 
$$\delta*\delta'= ( V_0 \subset E_0 \supset V_1 \subset \cdots \supset V_{r+1}=V_0' \subset E_0' \supset V_1' \subset \cdots \supset V_{r+1}').$$
We also use this notation if the last vertex of the first gallery does not coincide with the first vertex of the second one. In this situation $\delta*\delta'$ means the concatenation of $\delta$ with the translated gallery $\delta'+(V_{r+1}-V_0)$.\\
The formula of Gaussent-Littelmann is in terms of \textit{combinatorial} one-skeleton galleries.
\begin{defi}
A combinatorial one-skeleton gallery in $\mathbb{A}$ is a one-skeleton gallery $\delta = ( V_0 \subset E_0 \supset V_1 \subset \cdots \supset V_{r+1})$ where $V_0$ and $V_{r+1}$ are special vertices.
\end{defi}
Let $\omega$ be a fundamental coweight. We define a combinatorial one-skeleton gallery $\delta_{\omega}=(\mathfrak{o} \subset E_0 \supset \dots \supset \omega)$ associated to $\omega$ as follows:\\
Consider $\mathbb{R}_{\geq 0}\omega$ the real span of $\omega$ in $\mathbb{A}$. Let $E_0$ be the unique edge in the intersection of the fundamental alcove $\Delta_f$ with $\mathbb{R}_{\geq 0}\omega$. Let $V_1$ be the vertex contained in $E_0$ different from $V_0$. If $V_1$ is $\omega$ (as it is if and only if $\omega$ is minuscule) then $\delta_{\omega}=(\mathfrak{o}\subset E_0 \supset \omega)$. Otherwise the subsequent edge $E_1$ is the unique edge in $\mathbb{R}_{\geq 0}\omega$ that contains $V_1$ and is different from $E_0$ and so on until the vertex is $\omega$. We call these one-skeleton galleries fundamental although they do not have to be contained in the fundamental alcove and we call $E_{\omega}:=E_0$ the fundamental edge and $V_{\omega}:=V_1$ the fundamental vertex in $\omega$-direction. Unless $\omega$ is not minuscule $V_{\omega}$ does not coincide with the vertex $\omega$. Further, we define the minuscule one-skeleton gallery $\delta_{E_{\omega}}=(\mathfrak{o}\subset E_{\omega} \supset V_{\omega})$ and its Weyl group conjugates $\delta_{w(E_{\omega})}=w(\delta_{E_{\omega}})=(\mathfrak{o}\subset w(E_{\omega}) \supset w(V_{\omega}))$ for $w\in W$. Every edge in the apartment $\mathbb{A}$ is a (displaced) Weyl group conjugate of a fundamental edge. Thus, every one-skeleton gallery $\delta$ is a concatenation of Weyl group conjugates of minuscule one-skeleton galleries i.e. $\delta= \delta_{w_0(E_{\omega_{i_0}})}*\cdots*\delta_{w_r(E_{\omega_{i_r}})}$ where $\omega_{i_j}$ are fundamental coweights and $w_i \in W$ for every $i$. We refer to this presentation of a one-skeleton gallery as its \textbf{minuscule presentation}.\\
We now define a combinatorial one-skeleton gallery associated to a dominant coweight: Let $\omega_1, \dots, \omega_n$ be an enumeration of the fundamental coweights and let $\lambda= \sum_{i}\lambda_i\omega_i \in X\check{}_+$ be a dominant coweight. The associated one-skeleton gallery $\delta_{\lambda}$ is defined as follows:
$$ \delta_{\lambda}= \underbrace{\delta_{\omega_1}*\dots *\delta_{\omega_1}}_{\lambda_1 \; \text{times}}*\dots*\underbrace{\delta_{\omega_n}*\dots*\delta_{\omega_n}}_{\lambda_n \; \text{times}}.$$  
\begin{defi}
The type of a one-skeleton gallery $\delta = (V_0 \subset E_0 \supset \dots \subset E_r \subset V_{r+1})$ in $\mathbb{A}$ is defined as
$$ t(\delta)=(S^a(V_0) \subset  S^a(E_0) \supset \dots \supset  S^a(V_{r+1})).$$
We say a combinatorial one-skeleton gallery $\delta$ in $\mathbb{A}$ has type $\lambda$ for some $\lambda \in X\check{}_+$ if $t(\delta)=t(\delta_{\lambda})$. 
\end{defi}
Note that there exist combinatorial one-skeleton galleries such that the type is not a dominant coweight $\lambda$. For example for the combinatorial one-skeleton gallery $\delta=(\mathfrak{o} \subset E_0 \supset \frac{1}{2}\omega_1 \subset E_1 \supset \omega_2)$ with $E_0= \{ t \omega_1 \; \vert  \; t \in \left[ 0, \frac{1}{2}\right]\}$ and $E_1=\{ \frac{1}{2}\omega_1+ t(\omega_2-\frac{1}{2}\omega_1) \; \vert \; t \in \left[0,1\right]\}$ in the standard apartment of the affine building of type $B_2$ there does not exist a dominant coweight $\lambda$ such that $\lambda$ is the type of $\delta$:
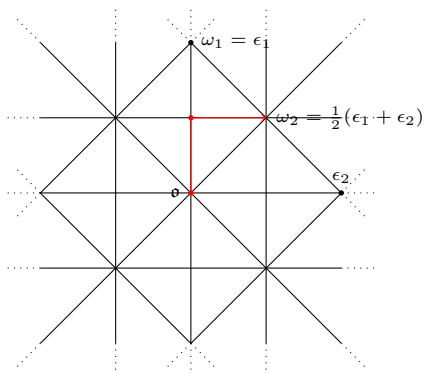
\begin{center}
   
   \begin{tikzpicture}[x=0.4cm,y=0.4cm]
   
    \draw (0,0)-- (0:2cm);   
   \draw (0,0)-- (45:2.83cm); 
   \draw (0,0)-- (90:2cm);
   \draw (0,0)-- (135:2.83cm);
   \draw (0,0)-- (180:2cm);
   \draw (0,0)-- (225:2.83cm);
   \draw (0,0)-- (270:2cm);
   \draw (0,0)-- (315:2.83cm);
   \draw (0,2cm)-- ++(315:2.83cm)-- ++(225:2.83cm)-- ++(135:2.83cm)-- ++(45:2.83cm);
   \draw (-1cm,2cm)-- ++(270:4cm);
   \draw (-2cm,-1cm)-- ++(0:4cm);
   \draw (1cm,2cm)-- ++(270:4cm);
   \draw (-2cm,1cm)-- ++(0:4cm);
   \draw[dotted] (-2cm,2cm)-- ++(135:0.5cm);
   \draw[dotted] (-1cm,2cm)-- ++(90:0.5cm);
   \draw[dotted] (0,2cm)-- ++(135:0.5cm);
   \draw[dotted] (0,2cm)-- ++(45:0.5cm);
   \draw[dotted] (0,2cm)-- ++(90:0.5cm);
   \draw[dotted] (1cm,2cm)-- ++(90:0.5cm);
   \draw[dotted] (2cm,2cm)-- ++(45:0.5cm);
   \draw[dotted] (2cm,1cm)-- ++(0:0.5cm);
   \draw[dotted] (2cm,0)-- ++(45:0.5cm);
   \draw[dotted] (2cm,0)-- ++(0:0.5cm);
   \draw[dotted] (2cm,0)-- ++(315:0.5cm);
   \draw[dotted] (2cm,-1cm)-- ++(0:0.5cm);
   \draw[dotted] (2cm,-2cm)-- ++(315:0.5cm);
   \draw[dotted] (1cm,-2cm)-- ++(270:0.5cm);
   \draw[dotted] (0,-2cm)-- ++(315:0.5cm);
   \draw[dotted] (0,-2cm)-- ++(270:0.5cm);
   \draw[dotted] (0,-2cm)-- ++(225:0.5cm);
   \draw[dotted] (-1cm,-2cm)-- ++(270:0.5cm);
   \draw[dotted] (-2cm,-2cm)-- ++(225:0.5cm);
   \draw[dotted] (-2cm,-1cm)-- ++(180:0.5cm);
   \draw[dotted] (-2cm,0)-- ++(135:0.5cm);
   \draw[dotted] (-2cm,0)-- ++(180:0.5cm);
   \draw[dotted] (-2cm,0)-- ++(225:0.5cm);
   \draw[dotted] (-2cm,1cm)-- ++(180:0.5cm);

   \draw[->,red] (0,0)-- ++(90:1cm)--++(0:1cm);
   \coordinate (o) at (0,0);
   \coordinate (-o_1) at (0,1cm);
   \fill[red](o) circle (1pt);
   \fill[red](-o_1) circle (1pt);
   
   \draw[left] (0,0) node {\tiny{$\mathfrak{o}$}};
   \draw[right] (0,2cm) node {\tiny{$\omega_1=\epsilon_1$}};
   \draw[above] (2cm,0) node {\tiny{$\epsilon_2$}};
   \draw[right] (1cm,1cm) node {\tiny{$\omega_2=\frac{1}{2}(\epsilon_1+\epsilon_2)$}};
   \fill (0,2cm) circle (1pt);
   \fill (2cm,0) circle (1pt);
   
   \end{tikzpicture}
   
\end{center}   
In order to define some properties of combinatorial one-skeleton galleries we need to introduce 2-step one-skeleton galleries:
\begin{defi}
A 2-step one-skeleton gallery $\delta$ in $\mathbb{A}$ is a one-skeleton gallery in $\mathbb{A}$ of the following form
$$ \delta  =(V_0\subset E \supset V \supset F \subset V_2).$$
We omit the vertices $V_0$ and $V_2$ in the following. 
\end{defi}
Note that the 2-step one-skeleton galleries do not need to be combinatorial one-skeleton galleries.
\begin{defi}
Let $(E \supset V \subset F)$ be a 2-step one-skeleton gallery in $\mathbb{A}$. We call $(E \supset V \subset F)$ minimal if there exists a sector $S$ at $V$ in $\mathbb{A}$ such that $F \in S$ and $E \in -S$.\\
A one-skeleton gallery $ \delta = ( V_0 \subset E_0 \supset V_1 \subset \dots \supset V_{r+1})$ in $\mathbb{A}$ is called locally minimal if the 2-step one-skeleton gallery $\delta_i= (E_{i-1} \supset V_i \subset E_{i})$ is minimal for $i \in \{1, \dots, r \}$.\\
We call $\delta$ (globally) minimal if there exists an equivalence class of sectors $\bar{S}$ such that there exists a sector $S_i \in \bar{S}$ with $E_{i-1} \in S_i$ and $E_i \in -S_i$ for $i \in \{1, \dots, r\}$.
\end{defi}
Clearly, global minimality implies local minimality.\\ \\
We can associate to a given 2-step one-skeleton gallery $(E \supset V \subset F)$ in $\mathbb{A}$ a pair of faces $(E_V,F_V)$ in $\mathbb{A}_V$ and vice versa. We call the pair $(E_V,F_V)$ a minimal pair if and only if the associated 2-step one-skeleton gallery $(E \supset V \subset F)$ is minimal in $\mathbb{A}$.
\begin{defi}
We say we obtain $(E \supset V \subset F)$ by a positive folding from $(E \supset V \subset F')$ if there exists $(\alpha,n)$ in $(\phi^+ \times \mathbb{Z})$ with $s_{(\alpha, n)}(F')=F$, $V \in H_{(\alpha, n)}$ and $H_{(\alpha,n)}$ separates $C_V^-$ and $F'$ from $F$.\\ \\
A 2-step one-skeleton gallery $(E \supset V \subset F)$ in $\mathbb{A}$  is called positively folded if it is minimal or if there exist faces $F_0, \dots, F_l$ in $\mathbb{A}$ such that
\begin{itemize}
  \item $F=F_l$,
  \item $(E \supset V \subset F_0)$ is minimal and
  \item $(E \supset V \subset F_i)$ is obtained from $(E \supset V \subset F_{i-1})$ by a positive folding for every $i \in \{1, \dots ,l\}$.
\end{itemize}
A one-skeleton gallery $\delta=(V_0 \subset E_0 \supset \dots \supset V_{r+1})$ in $\mathbb{A}$ is called locally positively folded if
\begin{itemize}
   \item $(E_{i-1}\supset V_i \subset E_i)$ is positively folded for every $i \in \{1, \dots, r\}$.
\end{itemize}
\end{defi}
As already mentioned it holds that the set of equivalence classes of sectors in $\mathbb{A}$ is in bijection with the Weyl group $W$. Consequentley we can carry the Bruhat order from  $W$ over to the equivalence classes: Let $\overline{S}$ and $\overline{S'}$ be two equivalence classes of sectors in $\mathbb{A}$.  Then $\overline{S}\geq \overline{S'}$ if and only if $w_1 \geq w_2$ with $\overline{w_1(C^+)}=\overline{S}$ and $\overline{w_2(C^+)}= \overline{S'}$.
\begin{defi}
A one-skeleton gallery $\delta=(V_0 \subset E_0 \supset \dots \supset V_{r+1})$ in $\mathbb{A}$ is called (globally) positively folded if
\begin{itemize}
   \item[(i)] $\delta$ is locally positively folded and
   \item[(ii)] there exists a sequence of sectors $S_0, \dots, S_r$ such that $S_i$ is a sector at $V_i$ and contains $E_i$ and $\overline{S_0}\geq \dots \geq \overline{S_r}$ for every $i \in \{1, \dots, r\}$.
\end{itemize}
\end{defi}
In \cite{GL1} Gaussent and Littelmann show under which conditions locally\\ positively folded implies positively folded and local minimality implies minimality. In this article we concentrate on the root systems for type $A_n$, $B_n$ and $C_n$. In these cases the theorem provides the following:
\begin{satz}\label{globloc}
Let $\lambda$ be a dominant coweight and let $\omega_1, \dots, \omega_n$ be the Bourbaki enumeration of the fundamental coweights. Let $\delta$ be a combinatorial one-skeleton gallery in $\mathbb{A}$ of type $\lambda$. If $\delta$ is locally positively folded (resp. locally minimal) then it is positively folded (resp. minimal). 
\end{satz}
\subsection{Galleries of residue chambers}
The next objects we need in order to compute the formula are the galleries of (residue) chambers in $\mathbb{A}_V$ for some vertex $V$ in $\mathbb{A}$. In this section we define what galleries of chambers are and introduce some of their properties.\\ \\
Let $V$ be a vertex in $\mathbb{A}$. 
\begin{defi}
A gallery of chambers in $\mathbb{A}_V$ is a sequence $C=(C_0\supset H_1 \subset C_1\supset  \dots \subset C_{r} )$ where
\begin{itemize}
   \item[(i)] $H_i$ is a face of $\mathbb{A}_V$ coming from a codimension one face in $\mathbb{A}$ for $i \in \{1,\dots,r\}$ and  
   \item[(ii)]$C_i$ is a chamber of $\mathbb{A}_V$ for all $i \in \{0, \dots r\}$.
\end{itemize}
\end{defi}
The type of a codimension one face in a gallery of chambers is only a number in the set $\{1, \dots, m\}$. We define \textbf{the type of a gallery of chambers} $C=(C_0\supset H_1 \subset C_1\supset \dots \subset C_{r} )$ to be the sequence $(t(H_1), \dots, t(H_r))$. Instead of describing a gallery of chambers as sequence of chambers and faces of codimension one, we write only the sequence of chambers and additionally the type of the gallery throughout the article.\\
One crucial definition is the following:
\begin{defi}
Let $\textbf{c}=(C_0, \dots, C_{r})$ be a gallery of chambers in $\mathbb{A}_V$ of type $\textbf{i}=(i_1,\dots, i_r)$.  If $C_j=C_{j+1}$ for some $j$ we call the pair $(C_j, C_{j+1})$ a folding of $\textbf{c}$. Let $S_V$ be a chamber of $\mathbb{A}_V$. We say a folding $(C_j,C_{j+1})$ is positive (resp. negative) with respect to $S_V$ if $H_{(\beta_{i_{j+1}},n_{i_{j+1}})}$ separates $S_V$ from $C_j=C_{j+1}$ (resp. separates $-S_V$ from $C_j=C_{j+1}$). The gallery of chambers  $\textbf{c}$ is said to be positively folded (resp. negatively folded) with respect to $S_V$ if all foldings of $\textbf{c}$ are positive (resp. negative) with respect to $S_V$.\\ \\
Every pair $(C_j,C_{j+1})$ that is not a folding is called a wall-crossing of type $i_{j+1}$ or a wall-crossing of $H_{(\beta_{i_{j+1}},n_{i_{j+1}})}$. We say a wall-crossing $(C_j,C_{j+1})$ of type $i_{j+1}$ is positive (resp. negative) with respect to $S_V$ if $H_{(\beta_{i_{j+1}},n_{i_{j+1}})}$ separates $S_V$ and $C_j$ from $C_{j+1}$ (resp. separates $-S_V$ and $C_j$ from $C_{j+1}$).
\end{defi}
Associated to a gallery of chambers we can define the so-called $\pm$-sequence:
\begin{defi}
Let $\textbf{c}=(C_0, \dots, C_{r})$ be a gallery of chambers in $\mathbb{A}_V$ of type $\textbf{i}=(i_1,\dots, i_r)$ and $S_V$ be a chamber. We call the sequence $ pm(\textbf{c})=(c_1, \dots, c_r)$ with
$$ c_i=\begin{cases}
  +, \text{ if } (C_{i-1},C_i) \text{ is a positive wall-crossing or positive folding}\\
     \hspace{0.5cm} \text{ with respect to $S_V$, }\\
  -, \text{ else}
      \end{cases} $$
the $\pm$-sequence of \textbf{c} with respect to $S_V$.
\end{defi}
\section{Gaussent-Littelmann formula}
In this section we introduce a few more definitions and then finally state the Gaussent-Littelmann formula for Hall-Littlewood polynomials of arbitrary type as in \cite{GL1}.\\ \\
Let $\lambda$ and $\mu$ be dominant coweights and fix an enumeration $\omega_1, \dots, \omega_n$ of the fundamental weights. We define $\Gamma^+(\delta_{\lambda}, \mu)$ to be the set of all positively folded combinatorial one-skeleton galleries $\delta =(\mathfrak{o}=V_0 \subset E_0 \supset \dots \supset V_{r+1})$ in $\mathbb{A}$ with $V_{r+1}=\mu$ and $t(\delta)=t(\delta_{\lambda})$. Now let $\delta=(\mathfrak{o}=V_0 \subset E_0 \supset \dots \supset V_{r+1})$ be in $\Gamma^+(\delta_{\lambda}, \mu)$. For every $j \in \{1, \dots, r  \}$ we define the following:
\begin{itemize}
  \item $D_j$ is the chamber in the standard apartment of the residue building at $V_j$ closest to $C_{V_j}^-$ containing $(E_j)_{V_j}$,
  \item $S^j$ is a sector at $V_j$ containing $E_{j-1}$ and $-S^j$ contains a face $F$ that has the same type as $E_j$,
  \item $s_{{i_j}_1}\cdots s_{{i_j}_{r_j}}$ is a reduced expression for $w\in W_{V_j}^a$ with $w(C_{V_{j}}^-)=D_j$, define $\textbf{i}_j=({i_j}_1,\dots, {i_j}_{r_j} )$.
\end{itemize}
We denote by $\Gamma^+_{S^j_{V_j}}(\textbf{i}_j, op)$ the set of all galleries of chambers $\textbf{c}=(C_{V_j}^-,\\C_1, \dots ,C_{r_j} )$ in $\mathbb{A}_{V_j}$ with type $\textbf{i}_j$ that are positively folded with respect to $S^j_{V_j}$ and have the property that the face $F'_{V_j}$ that is contained in $C_{r_j}$ and has the same type as $(E_j)_{V_j}$ forms a minimal pair with $(E_{j-1})_{V_j}$ in $\mathbb{A}_{V_j}$. Let $w_{D_0}$ be the element of $W$ that sends $C_{V_0}^-$ to $D_0$.\\ \\
In the Gaussent-Littelmann formula we need two statistics for a given gallery of chambers $\textbf{c}=(C_0=C_V^-, C_1, \dots, C_{r_j})$ in $\Gamma^+_{S^j_{V_j}}(\textbf{i}_j, op)$:
 
  $$ r(\textbf{c}) \; \text{ is the number of positive foldings of } \; \textbf{c} \text{ and} $$ 
  $$ t(\textbf{c}) \text{ is the number of positive wall-crossings of } \textbf{c}.$$
   
Now we can finally state the Gaussent-Littelmann formula for the Laurent polynomial $L_{\lambda,\mu}(q)$:
\begin{satz}
$$ L_{\lambda, \mu}(q)= \sum_{\delta \in \Gamma^+(\delta_{\lambda}, \mu)}q^{l(w_{D_0})}(\prod_{j=1}^{r}\sum_{\textbf{c} \in \Gamma^+_{S_{V_j}^j}(\textbf{i}_j, op)}q^{t(\textbf{c})}(q-1)^{r(\textbf{c})}).$$
\end{satz}
One significant property of this formula is that the summand of the first sum for $\delta=(\mathfrak{o}=V_0 \subset E_0 \supset \dots \supset V_{r+1})$  can be interpreted as the product over all edges $E_j$ for $j \in \{0,\dots,r\}$ as follows:\\
Define $$c(E_0)= q^{l(w_{D_0})} \text{ and }$$
$$ c((E_{j-1} \supset V_j \subset E_{j}))=\sum_{\textbf{c} \in \Gamma^+_{S_{V_j}^j}(\textbf{i}_j, op)}q^{t(\textbf{c})}(q-1)^{r(\textbf{c})} \; \text{ for } j \in \{1, \dots, r\}.$$
We obtain 
\begin{ecor}
$$ L_{\lambda, \mu}(q)= \sum_{\delta \in \Gamma^+(\delta_{\lambda}, \mu)}c(E_0)\prod_{j=1}^rc((E_{j-1} \supset V_j \subset E_{j})).$$
\end{ecor}
Define $$ c(\delta)= c(E_0)\prod_{j=1}^rc((E_{j-1} \supset V_j \subset E_{j})).$$
We want to point out that the contribution $c(E_0)$ of $E_0$ only depends on the edge $E_0$ itself and that the contribution $c((E_{j-1} \supset V_j \subset E_{j}))$ of $E_j$ for $j \in \{1, \dots, r\} $ only depends on the 2-step gallery $(E_{j-1} \supset V_j \subset E_j)$. 
\section{Combinatorial Gaussent-Littelmann formula}
In this section we state a recurrence for the set of galleries of chambers $\Gamma^+_{S^j_{V_j}}(\textbf{i}_j, op)$ that is used in the Gaussent-Littelmann formula in order to compute $L_{\lambda, \mu}$. This recurrence holds for arbitrary type. Furthermore, we introduce Young tableaux of type $A_n$,$B_n$ and $C_n$ that can be identified with one-skeleton galleries in the associated standard apartment of the affine building. Using the language of Young tableaux and the recurrence leads to a combinatorial version of the Gaussent-Littelmann formula for type $A_n$, $B_n$ and $C_n$ in terms of Young tableaux. 
\subsection{Recurrence}
Let $\lambda$ and $\mu$ be dominant coweights and fix an enumeration $\omega_1, \dots, \omega_n$ of the fundamental weights. Let $\delta =(\mathfrak{o}=V_0 \subset E_0 \supset \dots \supset V_{r+1})$ be a positively folded combinatorial one-skeleton gallery in $\Gamma^+(\delta_{\lambda}, \mu)$. In this section we want to take a closer look at the set $\Gamma^+_{S^j_{V_j}}(\textbf{i}_j, op)$ for $j\in \{1, \dots,r\}$. To calculate $\Gamma^+_{s^j_{V_j}}(\textbf{i}_j,op)$ we only need the 2-step one-skeleton gallery $(E_{j-1} \supset V_j \subset E_j)$. This gallery on its own is again positively folded. For this reason we work in this section only with positively folded 2-step one-skeleton galleries. Let $(E \supset V \subset F)$ be a 2-step one-skeleton gallery in $\mathbb{A}$. Throughout the following we identify ($E \supset V \subset F$) in $ \mathbb{A}$ with the associated gallery in $\mathbb{A}_V$. First of all we need to fix some more definitions:\\ \\
Let $k_2$ be the type of the face $F$ in $\mathbb{A}_V$. Further let $F_f$ be the face in $C_V^-$ with type $k_2$. Then $w_F \in W_V^a$ denotes the minimal Weyl group element that sends $F_f$ to $F$. The \textbf{length} of a face $F$ is defined by
$$ l(F):=l(w_F).$$ \\ 
This is equivalent to saying that $l(F)$ is the length of the minimal Weyl group element that sends $C_V^-$ to a chamber that contains $F$.\\ \\
Let $D$ be the chamber in $\mathbb{A} _V$ that contains $F$ and is closest to $C_V^-$. Further let $s_{i_1}\cdots s_{i_l}$ be a reduced expression of the Weyl group element $w_F$ in $W_V^a$. Then $s_{i_1}\cdots s_{i_l}$ sends $C_V^-$ to $D$. Define $ \textbf{i}=(i_1, \dots, i_l)$. Let $S_V$ be a chamber in $\mathbb{A}_V$ that contains $E$. Now define $ \Gamma ^+_{S_V}(\textbf{i}, op)= \Gamma ^+_{S_V}((E\supset V \subset F), \textbf{i})$ to be the set of all galleries of chambers $\textbf{c} = (C_V^-,C_1 \dots, C_l)$ with type $\textbf{i}$ that are positively folded with respect to $S_V$ and with the property that the face that is contained in $C_l$ and has the same type as $F$ forms a minimal pair with $E$ in $\mathbb{A}_V$. Define $\Gamma^+_{S_V}((E \supset V \subset F))$ to be the disjoint union of all $\Gamma ^+_{S_V}((E\supset V \subset F), \textbf{j}=(j_1, \dots, j_l))$ where $s_{j_1}\dots s_{j_l}$ is a reduced expression of $w_F$ in $W_V^a$.\\ \\
Let $k_1$ be the type of $E$ and $k_2$ the type of $F$ in $\mathbb{A}_V$, then $\textbf{k}=(k_1,k_2)$ denotes the type of the 2-step one-skeleton gallery $(E \supset V \subset F)$. Define the following sets:

$$ \Gamma^+(\textbf{k})= \{ (E' \supset V \subset F') \; \vert \; \text{positively folded with type } \textbf{k}\},$$
$$ \Gamma^+_a(\textbf{k})= \underset{S_V' \text{ a chamber in }  \mathbb{A}_V \; \text{that contains } \; E' }{\bigcup_{(E' \supset V \subset F') \in \Gamma^+(\textbf{k})}}\Gamma^+_{S_V'}((E'\supset V \subset F')) .$$  
Further we can write $\Gamma^+_{S_V}((E \supset V \subset F))$ as the disjoint union of the set of galleries $\textbf{c} \in \Gamma^+_{S_V}((E \supset V \subset F))$ with a folding in the first position, call this set $\Gamma^f_{S_V}((E \supset V \subset F))$, with the set of galleries $\textbf{c}$ in $\Gamma^+_{S_V}((E \supset V \subset F))$ with a crossing in the first position, call this set $\Gamma^c_{S_V}((E \supset V \subset F))$:
$$ \Gamma^+_{S_V}((E \supset V \subset F)) = \Gamma^f_{S_V}((E \supset V \subset F)) \; \cup \; \Gamma^c_{S_V}((E \supset V \subset F)).$$
\\ \\
We are going to state a recurrence for the set $\Gamma^+_a(\textbf{k})$. For this purpose we need to define two different ways to construct a gallery of chambers in $\Gamma^+_a(\textbf{k})$ out of a given one. Let us begin with the first construction:\\ \\
We want to start with a gallery of chambers in $\Gamma^+_{S_V}((E\supset V \subset F))$ and end up with one in $\Gamma^+_{s_j(S_V)}((s_j(E)\supset V \subset s_j(F)))$ where $l(s_j(F))=l(F)+1$. 
\begin{rem}
For $F \notin C_V^+$ there always exists such a reflection.
\end{rem}
Let $\textbf{c}= (C_{V}^-,C_1,\dots, C_r) \in \Gamma ^+_{S_V}((E \supset V \subset F), \textbf{i})$ be a gallery of chambers.\\ 
Let $s_j$ be a simple reflection in $W_V^a$ with $$ l(s_j(F))=l(F)+1.$$
Define $$ \textbf{c}_1 = (C_{V}^-, s_j(C_{V}^-), s_j(C_1), \dots, s_j(C_l))$$
with type $(j, \textbf{i})$.
\begin{satz}\label{konst1}
Then $$ \textbf{c}_1 \in \Gamma^c_{s_j(S_V)}(s_j(E) \supset V \subset s_j(F),(j,\textbf{i}) ) $$
holds and we get the following map:
\begin{align*} \nonumber
 f_1^j \; : \; &\Gamma^+_{S_V}((E \supset V \subset F), \textbf{i} ) &\longrightarrow  \; \;\;\;\;\;\;\;&\Gamma^c_{s_j(S_V)}((s_j(E) \supset V \subset s_j(F)), (j, \textbf{i})) \\    
 &\textbf{c}=(C_{V}^-,C_1, \dots, C_l) &\longmapsto  \; \;\;\;\;\;\;\;&\textbf{c}_1= (C_{V}^-, s_j(C_{V}^-),s_j(C_1), \dots , s_j(C_l)).
\end{align*}
Moreover the one-skeleton gallery $(s_j(E) \supset V \subset s_j(F))$ is positively folded.
\end{satz}
\begin{proof}
In order to prove the statement we need to check three properties of the new gallery $\textbf{c}_1$:
\begin{itemize}
    \item[(i)] type of  $\textbf{c}_1$  has to be correct, 
    \item[(ii)] all foldings in $\textbf{c}_1$ are positive with respect to $s_j(S_V)$, 
    \item[(iii)] the face that is contained in $s_j(C_l)$ and has the same type as $s_j(F)$ forms a minimal pair with $s_j(E)$.
\end{itemize}
For (i): The element $s_jw_F$ sends $C_V^-$ to the chamber that contains $s_j(F)$ and is closest to  $C_V^-$ because $l(s_jw_F)= l(w_F)+1$. Further $s_{i_1}\cdots s_{i_l}$ is a reduced expression of $w_F$. It follows:
$$ s_{j}s_{i_1}\cdots s_{i_l}$$ is a reduced expression of $s_jw_F$.\\
For (ii): Since all foldings in $\textbf{c}$ are positive with respect to $S_V$ all foldings in $\textbf{c}_1$ are positive with respect to $s_j(S_V)$.\\
For (iii): As $ \textbf{c} \in \Gamma ^+((E \supset V \subset F), \textbf{i})$ the following holds: The face $F'$ that is contained in $C_l$ and has the same type as $F$ forms a minimal pair with $E$ in $\mathbb{A}_V$. Subsequently, $s_j(E)$ and $s_j(F')$ form a minimal pair and since $s_j(F')\in s_j(C_l)$ we get the desired property.\\
From Theorem 6.5 in \cite{GL1} it is clear that ($s_j(E) \supset V \subset s_j(F)$) is positively folded.
\end{proof}
\begin{rem}
Informally speaking we reflect the gallery and extend it in such a way that the new gallery starts again in $C_V^-$.
\end{rem}
We have the following lemma.
\begin{elem}\label{pm1}
Let pm($\textbf{c}$) be the $\pm$-sequence of $\textbf{c}$ with respect to $S_V$. Then
$$ \text{pm}(\textbf{c}_1)=(a, \text{pm}(\textbf{c})) \; \text {, where }$$
$$ a=  \begin{cases} +\; ,  & l(s_jw_{S_V})= l(w_{S_V})-1 \\
                     - \; , & \text{else}
       \end{cases} ,$$
where $w_{S_V} \in W_V^a$ is the element that sends $C_V^-$ to $S_V$ is the $\pm$-sequence of \textbf{$c_1$} with respect to $s_j(S_V)$.

\end{elem}
\begin{proof}
That the two $\pm$-sequences coincide except for the first crossing is clear. It remains to calculate the sign of the first crossing:\\
Consider $(s_j(E) \supset V \subset s_j(F))$. The element $s_jw_{S_V}$ sends $C_V^-$ to $s_j(S_V)$.
 For the first crossing we get
$$   + \; \; \; \; \text{, if }\; l(s_js_jw_{s_V})= l(s_jw_{s_V})+1,$$
$$ - \; \; \; \; \text{, if } \; l(s_js_jw_{s_V})= l(s_jw_{s_V})-1.$$
And this is equivalent to saying
$$   - \; \; \; \; \text{, if }\; l(s_jw_{s_V})= l(w_{s_V})+1,$$
$$ + \; \; \;\; \text{, if } \; l(s_jw_{s_V})= l(w_{s_V})-1.$$
\end{proof}
\textbf{Example}\\
Let $(E \supset V \subset F)$ be a 2-step one-skeleton gallery in the standard apartment of the residue building of type $A_2$ with $E=F=\{ V + t (\epsilon_1+\epsilon_3) \; \vert  \; t \in \left[0, \infty \right]\}$:
\begin{center}

     \begin{tikzpicture}

    \coordinate(o) at (0,0) node[red, above right] {\tiny{$V$}};

   	\draw[->] (0,0)-- ++(60:2cm);
   	\draw[->] (0,0)-- ++(120:2cm);
   	\draw[->] (0,0)-- ++(180:2cm);
    \draw[->] (0,0)-- ++(240:2cm);
   	\draw[->] (0,0)-- ++(300:2cm);
   	\draw[->] (0,0)-- ++(0:2cm);

   	
   	\draw[above] (-1cm,0) node {\tiny{$V+\epsilon_2$}};
   	\draw[right] (0.5cm,0.87cm) node {\tiny{$V+\epsilon_1$}};
   	\draw[right] (0.5cm,-0.87cm) node {\tiny{$V+\epsilon_3$}};
   	\fill (-1cm,0) circle (1pt);
   	\fill (0.5cm,0.87cm) circle (1pt);
   	\fill (0.5cm,-0.87cm) circle (1pt);
   	\fill[red] (0,0) circle (1pt);
   	 
   	\draw[red] (0,0)--(2cm,0);
   	\draw[red,above] (1cm,0) node {\tiny{$F$}};
   	\draw[red,below] (1cm,0) node {\tiny{$E$}}; 
   	\draw[red, ->] (0,0)-- ++(90:0.05cm)-- ++(0:2cm);

   	\draw (0,-1.5cm) node {\small{$C_V^-$}};
   	\draw (1.4cm, -0.7cm) node {\small{$S_V$}};

   	\draw[dotted,->] (0,-1cm)-- ++(30:0.5cm)-- ++(300:0.05cm)-- ++(210:0.5);

     \end{tikzpicture}
     
\end{center}
The dotted line in the picture illustrates the only gallery of chambers  $\textbf{c}$ in the set $\Gamma^+_{S_V}((E \supset V \subset F),(1))$. For the simple reflection $s_2 \in W_V^a$ which is the reflection at the hyperplane containing $V+\epsilon_1$ it holds that $l(s_2(F))=l(F)+1$. Consequently, we can apply $f_1^2$ to $\Gamma^+_{S_V}((E \supset V \subset F),(1))$:\\
\begin{center}

     \begin{tikzpicture}

    \coordinate(o) at (0,0) node[red, above right] {\tiny{$V$}};

   	\draw[->] (0,0)-- ++(60:2cm);
   	\draw[->] (0,0)-- ++(120:2cm);
   	\draw[->] (0,0)-- ++(180:2cm);
    \draw[->] (0,0)-- ++(240:2cm);
   	\draw[->] (0,0)-- ++(300:2cm);
   	\draw[->] (0,0)-- ++(0:2cm);

   	
   	\draw[above] (-1cm,0) node {\tiny{$V+\epsilon_2$}};
   	\draw[right] (0.5cm,0.87cm) node {\tiny{$V+\epsilon_1$}};
   	\draw[right] (0.5cm,-0.87cm) node {\tiny{$V+\epsilon_3$}};
   	\fill (-1cm,0) circle (1pt);
   	\fill (0.5cm,0.87cm) circle (1pt);
   	\fill (0.5cm,-0.87cm) circle (1pt);
   	\fill[red] (0,0) circle (1pt);
   	 
   	\draw[red] (0,0)--(120:2cm);
   	\draw[red,right] (-0.6cm,1.1cm) node {\tiny{$s_2(F)$}};
   	\draw[red,left] (-0.6cm,1.1cm) node {\tiny{$s_2(E)$}}; 
   	\draw[red, ->] (0,0)-- ++(30:0.05cm)-- ++(120:2cm);

   	\draw (0,-1.5cm) node {\small{$C_V^-$}};
   	\draw (-1.4cm, +0.7cm) node {\small{$s_2(S_V)$}};

   	\draw[dotted,->] (0,-1cm)-- ++(150:1cm)-- ++(90:0.5cm)-- ++(180:0.05cm)-- ++(270:0.5cm);

     \end{tikzpicture}
     
\end{center}
The dotted line in the picture illustrates the new gallery  of chambers $\textbf{c}_1$.\\ \\

Let us now get to the second construction. In this construction we start with a gallery of chambers in $\Gamma^+_{S_V}((E \supset V \subset F))$ and end up with a gallery of chambers in $\Gamma^+_{S_V}((E \supset V \subset s_j(F)))$ where $l(s_j(F))=l(F)+1$.
Let again $\textbf{c}= (C_{V}^-,C_1,\dots, C_r) \in \Gamma ^+_{S_V}((E \supset V \subset F), \textbf{i})$ be a gallery of chambers.\\
Let $s_j$ be a simple reflection in $W_V^a$ with $l(s_j(F))=l(F)+1$ and $l(s_jw_{S_V})= l(w_{S_V})-1.$\\
Define $$ \textbf{c}_2 = (C_{V}^-, C_{V}^-, C_1, \dots, C_l)$$
with type $(j,\textbf{i}).$
\begin{satz}
Then $$ \textbf{c}_2 \in \Gamma^f_{S_V}( (E  \supset V \subset s_j(F)), (j,\textbf{i})) $$
holds and we get the following map:
\begin{align*} \nonumber
 f_2^j \; : \; &\Gamma^+_{S_V}((E \supset V \subset F), \textbf{i} ) &\longrightarrow  \; \;\;\;\;\;\;\;&\Gamma^f_{S_V}((E \supset V \subset s_j(F)), (j,\textbf{i})) \\    
 &\textbf{c}=(C_{V}^-,C_1, \dots, C_l) &\longmapsto  \; \;\;\;\;\;\;\; &\textbf{c}_2= (C_{V}^-, C_{V}^-,C_1, \dots , C_l).
 \end{align*}
 Moreover the one-skeleton gallery $(E \supset V \subset s_j(F))$ is positively folded.
\end{satz} 
\begin{proof}
Like in the first construction we need to check three properties of the new gallery $\textbf{c}_2$:
\begin{itemize}
    \item[(i)] type of $\textbf{c}_2$ has to be correct, 
    \item[(ii)] all foldings have to be positive with respect to $S_V$,
    \item[(iii)] the face that is contained in $C_l$ with the same type as $s_j(F)$ forms a minimal pair with $E$.
\end{itemize}
For (i): see the proof of the last construction.\\
For (ii): neither $E$ nor the walls on which the foldings are are changed, it remains to show that the new folding in the first step is positive. For having a positive folding we need
$$ l(s_jw_{S_V})= l(w_{S_V})-1.$$ But this was our assumption.\\
For (iii): the type of $F$ coincides with the type of $s_j(F)$ and $C_l$ does not change, consequently the desired property follows.\\
It again follows from Theorem 6.5 in \cite{GL1} that $(E \supset V \subset s_j(F))$ is positively folded.
\end{proof}
As in the first construction we are interested in the $\pm$-sequence of the new gallery $\textbf{c}_2$:
\begin{elem}\label{pm2}
Let pm($\textbf{c}$) be the $\pm$-sequence of the gallery $\textbf{c}$ with respect to $S_V$.\\
Then:
$$ pm(\textbf{c}_2)=(+, pm(\textbf{c}))$$
is the $\pm$-sequence of $\textbf{c}_2$ with respect to $S_V$.
\end{elem}
\begin{proof}
There is nothing to show.
\end{proof}
\textbf{Example}\\
Let $(E \supset V \subset F)$ be a 2-step one-skeleton gallery in the standard apartment of the residue building of type $A_2$ with $E= \{V + t(\epsilon_1+\epsilon_2) \, \vert \; t \in \left[0,\infty\right]\}$ and $F=\{ V + t (\epsilon_1+\epsilon_3) \; \vert  \; t \in \left[0, \infty \right]\}$:
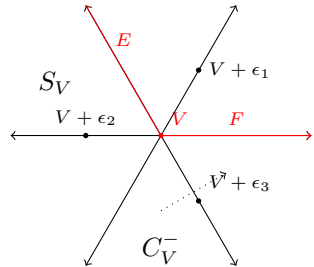
\begin{center}

     \begin{tikzpicture}

    \coordinate(o) at (0,0) node[red, above right] {\tiny{$V$}};

   	\draw[->] (0,0)-- ++(60:2cm);
   	\draw[->] (0,0)-- ++(120:2cm);
   	\draw[->] (0,0)-- ++(180:2cm);
    \draw[->] (0,0)-- ++(240:2cm);
   	\draw[->] (0,0)-- ++(300:2cm);
   	\draw[->] (0,0)-- ++(0:2cm);

   	
   	\draw[above] (-1cm,0) node {\tiny{$V+\epsilon_2$}};
   	\draw[right] (0.5cm,0.87cm) node {\tiny{$V+\epsilon_1$}};
   	\draw[above right] (0.5cm,-0.87cm) node {\tiny{$V+\epsilon_3$}};
   	\fill (-1cm,0) circle (1pt);
   	\fill (0.5cm,0.87cm) circle (1pt);
   	\fill (0.5cm,-0.87cm) circle (1pt);
   	\fill[red] (0,0) circle (1pt);
   	 
   	\draw[->, red] (0,0)--(2cm,0);
   	\draw[red,above] (1cm,0) node {\tiny{$F$}};
   	\draw[red,below] (-0.5cm,1.5cm) node {\tiny{$E$}}; 
   	\draw[red] (0,0)-- ++(120:2cm);

   	\draw (0,-1.5cm) node {\small{$C_V^-$}};
   	\draw (-1.4cm, 0.7cm) node {\small{$S_V$}};

   	\draw[dotted,->] (0,-1cm)-- ++(30:1cm);

     \end{tikzpicture}
     
\end{center}
The dotted line in the picture illustrates the only gallery of chambers  $\textbf{c}$ in the set $\Gamma^+_{S_V}((E \supset V \subset F),(1))$. For the simple reflection $s_2 \in W_V^a$ which is the reflection at the hyperplane containing $V+\epsilon_1$ it holds that $l(s_2(F))=l(F)+1$ and $l(s_2w_{S_V})=l(w_{S_V})-1$. Consequently, we can apply $f_2^2$ to $\Gamma^+_{S_V}((E \supset V \subset F),(1))$:\\
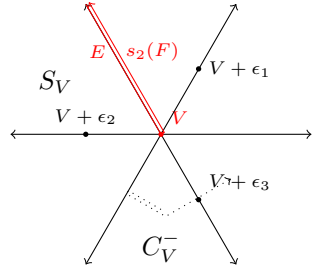
\begin{center}

     \begin{tikzpicture}

    \coordinate(o) at (0,0) node[red, above right] {\tiny{$V$}};

   	\draw[->] (0,0)-- ++(60:2cm);
   	\draw[->] (0,0)-- ++(120:2cm);
   	\draw[->] (0,0)-- ++(180:2cm);
    \draw[->] (0,0)-- ++(240:2cm);
   	\draw[->] (0,0)-- ++(300:2cm);
   	\draw[->] (0,0)-- ++(0:2cm);

   	
   	\draw[above] (-1cm,0) node {\tiny{$V+\epsilon_2$}};
   	\draw[right] (0.5cm,0.87cm) node {\tiny{$V+\epsilon_1$}};
   	\draw[above right] (0.5cm,-0.87cm) node {\tiny{$V+\epsilon_3$}};
   	\fill (-1cm,0) circle (1pt);
   	\fill (0.5cm,0.87cm) circle (1pt);
   	\fill (0.5cm,-0.87cm) circle (1pt);
   	\fill[red] (0,0) circle (1pt);
   	 
   	\draw[red] (0,0)--(120:2cm);
   	\draw[red,right] (-0.6cm,1.1cm) node {\tiny{$s_2(F)$}};
   	\draw[red,left] (-0.6cm,1.1cm) node {\tiny{$E$}}; 
   	\draw[red, ->] (0,0)-- ++(30:0.05cm)-- ++(120:2cm);

   	\draw (0,-1.5cm) node {\small{$C_V^-$}};
   	\draw (-1.4cm, +0.7cm) node {\small{$S_V$}};

   	\draw[dotted,->] (0,-1cm)-- ++(150:0.5cm)-- ++(240:0.05cm)-- ++(330:0.6cm)-- ++(30:1cm);

     \end{tikzpicture}
     
\end{center}
The dotted line in the picture illustrates the new gallery  of chambers $\textbf{c}_1$.\\ \\
Using the two constructions presented above we show that the set $\Gamma^+_a(\textbf{k})$ can be build recursively from special galleries which are described in the following lemma:
\begin{elem}
Let $(E \supset V \subset F)$ be a positively folded one-skeleton gallery with $F \in C_V^-$ and let $S_V$ be a chamber that contains $E$ and $V$ in the standard apartment of the residue building at $V$.\\
Then $$ \Gamma^+_{S_V}((E \supset V \subset F)) = \{\textbf{c}_0\}=\{(C_V^-)\}.$$
\end{elem}
\begin{proof}
As $F$ is in $C_V^-$ it follows: The chamber that contains $F$ and is closest to $C_V^-$ is $C_V^-$ itself. Therefore the Weyl group element in $W_V^a$ that sends $C_V^-$ to this chamber is $id$ with $l(id)=0$. The claim follows.
\end{proof}
Now we can finally formulate the recurrence for $\Gamma^+_a(\textbf{k})$:
\begin{satz}{Recurrence for the galleries of chambers}\label{rec}
\\Let $(E \supset V \subset F)$ be a positively folded one-skeleton gallery of type $\textbf{k}$. Let $\textbf{c}$ be a gallery of chambers in $ \Gamma^+_{S_V}((E \supset V \subset F),\textbf{i}=(i_1,\dots, i_l)) \subset \Gamma^+_a(\textbf{k})$. Then there exists a unique $(E' \supset V \subset F') \in \Gamma^+(\textbf{k})$ with $F' \subset C_V^-$ and a unique chamber $S_V'$ with $S_V' \supset E'$ and a unique sequence $(p_1, \dots, p_l)$ with $p_m \in \{1,2\}$ for every $m \in \{1, \dots, l \}$ such that:
$$ \textbf{c} = f_{p_1}^{i_1}(f_{p_2}^{i_2}\dots f_{p_l}^{i_l}(\textbf{c}_0)\dots),$$
where $\textbf{c}_0 \in \Gamma^+_{S_V'}((E' \supset V \subset F')).$
\end{satz}
\begin{proof}
In order to prove the recurrence we define the inverse maps to $f_1^j$ and $f_2^j$. Recall that we can write $\Gamma_{S_V}^+((E \supset V \subset F))$ as the disjoint union of $\Gamma_{S_V}^f((E \supset V \subset F))$ and $\Gamma_{S_V}^c((E \supset V \subset F))$. If the crossing in the first position of the galleries is of type $i$ then we write $\Gamma_{S_V}^{c_i}((E \supset V \subset F))$ and if the folding in the first position is of type $i$ then we write $\Gamma_{S_V}^{f_i}((E \supset V \subset F))$. Further recall the two constructions:\\
Let $s_j$ be a simple reflection in $W_V^a$ with $l(s_j(F))= l(F) + 1$. We get:
\begin{align*} \nonumber
 f_1^j \; : \; &\Gamma^+_{S_V}((E \supset V \subset F), \textbf{i} ) &\longrightarrow  \; \;\;\;\;\;\;\;&\Gamma^{c_j}_{s_j(S_V)}((s_j(E) \supset V \subset s_j(F)), (j,\textbf{i})) \\
 &\textbf{c}=(C_{V}^-,C_1, \dots, C_l) &\longmapsto  \; \;\;\;\;\;\;\;&\textbf{c}_1= (C_{V}^-, s_j(C_{V}^-),s_j(C_1), \dots , s_j(C_l)).\\
 \intertext{If additionally $l(s_jw_{S_V})=l(w_{S_V})-1 $ holds  for $j$ then we get:}
 f_2^j \; : \; &\Gamma^+_{S_V}((E \supset V \subset F), \textbf{i} ) &\longrightarrow  \; \;\;\;\;\;\;\;&\Gamma^{f_j}_{S_V}((E \supset V \subset s_j(F)),(j,\textbf{i})) \\
&\textbf{c}=(C_{V}^-,C_1, \dots, C_l) &\longmapsto  \; \;\;\;\;\;\;\; &\textbf{c}_2= (C_{V}^-, C_{V}^-,C_1, \dots , C_l).\\
\intertext{Now we want to construct the inverse maps for these. Let $(E \supset V \subset F)$ be a positively folded gallery with $F \not\subset C_V^-$. Define}
 \tilde{f}_1^j \; : \; &\Gamma^{c_j}_{S_V}((E \supset V \subset F),(j,\textbf{i})) &\longrightarrow \; \; \; \; \; \; \; \; &\Gamma^+_{s_j(S_V)}((s_j(E) \supset V \subset s_j(F)), \textbf{i})  \\
&\textbf{c}=(C_{V}^-,C_1, \dots, C_l) &\longmapsto \;\;\;\; \; \; \; \; &\widetilde{\textbf{c}}_1= (s_j(C_1)=C_{V}^-, \dots , s_j(C_l)).\\
\intertext{Further define:} 
\tilde{f}_2^j \; : \; &\Gamma^{f_j}_{S_V}((E \supset V \subset F),(j,\textbf{i})) &\longrightarrow \; \; \; \; \; \; \; \; &\Gamma^+_{S_V}((E \supset V \subset s_j(F)), \textbf{i}) \\ 
&\textbf{c}=(C_{V}^-,C_1=C_{V}^-, \dots, C_l) &\longmapsto \;\;\; \; \; \; \; \; &\widetilde{\textbf{c}}_2= (C_1=C_{V}^-, \dots , C_l).
\end{align*}
With the same arguments as in the constructions of $f_1^j$ and $f_2^j$ we get that $\tilde{f}_1^j$ and $\tilde{f}_2^j$ are defined in a proper way i.e. that $\widetilde{\textbf{c}_1} \in \Gamma^+_{s_j(S_V)}((s_j(E) \supset V \subset s_j(F)),(\textbf{i}))$ and that $\widetilde{\textbf{c}}_2 \in \Gamma^+_{S_V}((E \supset V \subset s_j(F)),\textbf{i})$ and that $(s_j(E) \supset V \subset s_j(F))$ and $(E \supset V \subset s_j(F))$ are positively folded.\\
It is easy to check that these maps are the inverse maps to $f^j_1$ and $f^j_2$.\\
We now want to prove the statement.  Let $\textbf{c}$ be a gallery of chambers in $ \Gamma^+_{S_V}((E \supset V \subset F),\textbf{i}) \subset \Gamma^+_a(\textbf{k})$. Depending on whether the first step is a crossing or a folding we can apply either $\tilde{f}_1^{i_1}$ or $\tilde{f}_2^{i_1}$. Now we take the resulting gallery of chambers and check again which one of the two maps $\tilde{f}_1^{i_2}$ and $\tilde{f}_2^{i_2}$ can be applied and so one. In other words, how to reduce a given gallery of chambers step by step is determined by the shape of the gallery itself. Therefore there exists a unique $(E' \supset V \subset F') \in \Gamma^+(\textbf{k})$ with $F' \subset C_V^-$ and a unique chamber $S_V'$ with $S_V' \supset E'$ and $S_V' \supset V$ and a unique sequence $(p_1, \dots, p_l)$ with $p_m \in \{1,2\}$ for every $m \in \{1, \dots, l \}$ such that
$$ \textbf{c} = f_{p_1}^{i_1}(f_{p_2}^{i_2}\dots f_{p_l}^{i_l}(\textbf{c}_0)\dots),$$
where $\textbf{c}_0 \in \Gamma^+_{S_V'}((E' \supset V \subset F')).$ 
\end{proof}
\subsection{Type $A_n$}
Let $\mathbb{A}$ be the standard apartment of the affine building of type $A_n$. The coweight lattice $X\check{}$ can be identified with $\mathbb{Z}^{n+1}/(1,\dots,1)$ and we identify the weight lattice $X$ with $X\check{}$ using the standard inner product on $\mathbb{Z}^{n+1}$. The simple coroots are $\alpha_i\check{}= \epsilon_i-\epsilon_{i+1}$ for $i \in\{1,\dots,n\}$ where $\epsilon_i$ is the $i$th unit vector of $\mathbb{Z}^{n+1}$, the simple roots $\alpha_i$ coincide with the simple coroots $\alpha_i\check{}$. The positive roots are $\epsilon_i- \epsilon_j$ with $i < j$ and $\rho= 1/2 \sum{\text{positive roots}}= 1/2(n)\epsilon_1+1/2(n-2)\epsilon_2+ \dots+ 1/2(-n)\epsilon_{n+1}$. For the fundamental coweights we choose the Bourbaki enumeration i.e. $\omega_i = \epsilon_1 +\dots +\epsilon_i$ for $i \in \{1, \dots, n\}$ \cite{B}. The combinatorial one-skeleton gallery associated to a fundamental coweight $\omega_i$ coincides with the minuscule one-skeleton gallery in $\omega_i$-direction i.e. $\delta_{\omega_i}=\delta_{E_{\omega_i}}=(\mathfrak{o}\subset E_{\omega_i} \supset \omega_i)$ for $i \in \{1, \dots,n\}$ since all fundamental coweights are minuscule for type $A_n$. The combinatorial one-skeleton galleries of the same type as $\omega_i$ are $\delta_{\tau(E_{\omega_i})}=(\mathfrak{o} \subset \tau(E_{\omega_i}) \supset \tau(\omega_i))$ with $\tau\in W/W_{\omega_i}$.
A dominant coweight $\lambda= \lambda_1 \epsilon_1 + \dots \lambda_{n} \epsilon_{n}$ is given by a weakly decreasing sequence $\lambda_1\geq\dots\geq \lambda_{n}\geq 0$. Recall that $\delta_{\lambda}$ denotes the combinatorial one-skeleton gallery $\delta_{a_1\omega_1}*\dots*\delta_{a_n\omega_n}$ where $\lambda= \sum{a_i\omega_i}$. In this section we restrict ourselves to combinatorial one-skeleton galleries $\delta=\delta_{w_0(E_{\omega_{i_0}})}*\cdots*\delta_{w_r(E_{\omega_{i_r}})}$ with $i_0\leq\dots\leq i_r$. Note that for every one-skeleton gallery $\delta$ of type $A_n$ there exists a dominant coweight $\lambda$ such that $t(\delta)=t(\delta_{\lambda})$. The Weyl group $W$ is the symmetric group $S_{n+1}$. Consider the action of $W$ on the coweights $X\check{}$. Let $\lambda \in X\check{}$ be given in the basis $\{\epsilon_1, \dots, \epsilon_{n+1}\}$. Then applying the simple reflection $s_i$ to $\lambda$ interchanges the coefficient of $\epsilon_i$ with the coefficient of $\epsilon_{i+1}$ and fixes the coefficient of $\epsilon_j$ for $j  \notin \{i,i+1\}$ for every $i \in \{1,\dots,n\}$.\\
Since every vertex $V$ is special the Weyl group $W_V^v$ coincides with $W$.
\subsubsection{Young tableau of type $A_n$}
Let $\lambda=\lambda_1\epsilon_1+\dots+ \lambda_{n}\epsilon_n$ be a dominant coweight and let $r$ be the smallest index with $\lambda_{r+1}=0 $. We associate to $\lambda$ a diagram consisting of $r$ left-aligned rows where the $i$th row consists of $\lambda_i$ boxes (from top to bottom). In the following we denote the diagram by \textbf{dg($\lambda$)}.\\
\textbf{Example.}\\
For $\lambda=3\epsilon_1+3\epsilon_2+2\epsilon_3+1\epsilon_4$ we obtain
$$ dg(\lambda)=\Yvcentermath1 \yng(3,3,2,1) \; .$$
A \textbf{Young tableau} $T$ of type $A_n$ of \textbf{shape} $\lambda$ is the diagram dg($\lambda$) where each box is filled with a number in $\{1,\dots,n+1\}$ such that the entries are strictly increasing in the columns (from top to bottom).
The coweight $\mu=\mu_1\epsilon_1+\dots+ \mu_{n+1}\epsilon_{n+1}$, where $\mu_i$ is the number of boxes in the diagram in which an $i$ is inserted is called the \textbf{content} of the Young tableau.\\
\textbf{Example.}\\
The Young tableau
$$ \young(121,232,34,4)$$
has shape $\lambda$ as in the first example and content $\mu=2\epsilon_1+3\epsilon_2+2\epsilon_3+2\epsilon_4$.\\ \\
Let $T$ be a Young tableau and $C_i$ denotes the $i$th column of $T$ for $i \in \{0,\dots,l\}$ (from left to right).
We call $T$ \textbf{minimal} if all entries of $C_i$ also appear in $C_{i-1}$ for $i \in \{1,\dots,l\}$.\\
\textbf{Example.}\\
$$ T= \Yvcentermath1 \young(113,234,34,4)\; .$$\\
A Young tableau $T$ is called \textbf{semistandard} if the entries are weakly increasing from left to right in the rows.\\
\textbf{Example.}\\
$$ T= \Yvcentermath1 \young(122,234,35,4)\; .$$\\
Note that minimal Young tableaux are always semistandard. Let SSYT($\lambda, \mu$) denote the set of all semistandard Young tableau with shape $\lambda$ and content $\mu$.\\ 
We can identify one-skeleton galleries in $\mathbb{A}$ with Young tableaux of type $A_n$ as follows:\\
Let $\delta = \delta_{w_0(E_{\omega_{i_0}})}*\dots*\delta_{w_r(E_{\omega_{i_r}})} $ be a one-skeleton gallery. Let $w_k(\omega_{i_k})= \epsilon_{k_1}+\dots+\epsilon_{k_j}$ for $k \in \{0,\dots, r\}$. Then we associate to $\delta_{w_k(E_{\omega_{i_k}})}$ the column $C_{r-k}$ consisting of $j$ boxes filled with the numbers $k_1< \dots < k_j$ in increasing order from top to bottom. The Young tableau $T_{\delta}=(C_{0},\dots,C_{r})$ that we associate to the one-skeleton gallery $\delta$ is the tableau that we obtain by putting the columns next to each other aligned at the top. This tableau has shape $\lambda=\sum_{j=0}^r\omega_{i_j}$.\\
\textbf{Example.}\\
For  $\delta=(\delta_{s_1s_3s_2(E_{\omega_2})}*\delta_{s_1s_2s_4s_3(E_{\omega_3})}*\delta_{E_{\omega_4}})$ the associated Young tableau is
$$ T_{\delta}= \Yvcentermath1 \young(122,234,35,4)\; .$$\\
This assignment is clearly a bijection between the set of Young tableaux of type $A_n$ of shape $\lambda$ and the combinatorial one-skeleton galleries of type $A_n$ of type $\lambda$ for a dominant coweight $\lambda$ starting in the origin $\mathfrak{o}$. Thus for a Young tableau $T=(C_0, \dots,C_r)$, the one-skeleton gallery $\delta_T=(\mathfrak{o} \subset E_0 \supset V_1 \subset  \dots \supset V_{r+1})$ denotes the associated combinatorial one-skeleton gallery where $E_{r-j}$ is the edge corresponding to the column $C_j$ for $j \in \{0,\dots, r\}$. A 2-step one-skeleton gallery $(E \supset V \subset F)$ becomes a 2-column Young tableau, a Young tableau consisting of only 2 columns, at vertex $V$.\\ \\
Now we can investigate how being minimal and positively folded for a one-skeleton gallery translates in terms of Young tableaux: 
\begin{satz}
This assignment defines a bijection between the set of all minimal Young tableaux of shape $\lambda$ and the set of all minimal combinatorial one-skeleton galleries of type $\lambda$ starting in the origin $\mathfrak{o}$. Further, it also defines a bijection between the set of all semistandard Young tableaux of shape $\lambda$ and the set of all positively folded combinatorial one-skeleton galleries of type $\lambda$ starting in the origin $\mathfrak{o}$. 
\end{satz}
\begin{rem}
The proof that semistandard Young tableaux of type $A_n$ and positively folded combinatorial one-skeleton galleries of type $A_n$ are the same can be found in \cite{LMS} and in \cite{GL1}. For the convenience of the reader we also give a detailed proof.
\end{rem}
\begin{proof}
Because of Theorem \ref{globloc} it suffices to show that the set of all minimal 2-step one-skeleton galleries starting in the origin $\mathfrak{o}$ is in bijection with the set of all minimal 2-column Young tableaux and that the set of all positively folded 2-step one-skeleton galleries  starting in the origin $\mathfrak{o}$ is in bijection with all semistandard 2-column Young tableaux:\\
We begin with minimality: Let $(E \supset V \subset F)$ be a minimal 2-step one-skeleton gallery in $\mathbb{A}$. We can assume that $V$ is a special vertex because we are in type $A_n$. Because $(E \supset V \subset F)$ is minimal we know that there exist fundamental coweights $\omega_{i_1}$ and $\omega_{i_2}$ with $i_1 \leq i_2$ and a minimal Weyl group element $w \in W$ such that the minuscule presentation of $(E \supset V \subset F)$ is $\delta_{w(E_{\omega_{i_1}})}*\delta_{w(E_{\omega_{i_2}} )}$. Consider
$$ w(\omega_{i_2})=w(\omega_{i_1}+\epsilon_{i_1+1}+\dots+\epsilon_{i_2})= w(\omega_{i_1} ) +w(\epsilon_{i_1+1}+\dots+\epsilon_{i_2}).$$
Thus the set of all entries in the column $C_{F}$ associated to $F$ contains the set of all entries of the column $C_{E}$ associated to $E$.\\
Conversely, let $T$ be a minimal 2-column Young tableau where $C_1$ denotes the first and $C_2$ the second column of $T$. Let $l_1$ denote the number of boxes in $C_1$ and $l_2$ the number of boxes in $C_2$. Let $(E \supset V \subset F)$ be the associated 2-step one-skeleton gallery. We now give an explicit construction of the minimal Weyl group element $w \in W$ such that $\delta_{w(E_{\omega_{l_1}})}*\delta_{w(E_{\omega_{l_2}})}$ is the minuscule presentation of $(E \supset V \subset F)$: Without loss of generality we assume that $l_1 > l_2$. Because if $l_1=l_2$ then the set of entries is the same in both columns. Let $\{i_1, \dots, i_{l_1}\}$ be the set of entries in $C_1$ and $\{j_1, \dots, j_{l_2}\}$ the set of entries in $C_2$. If $i_{l_1}$ is not an entry of $C_2$ then $s_{i_{l_1}-1}\cdots s_{l_1}$ does not change $\omega_{l_2}$ but $s_{i_{l_1}-1}\cdots s_{l_1}(\omega_{l_1})= \omega_{l_1}-\epsilon_{l_1}+\epsilon_{i_{l_1}}$. Go on like this until $i_m=j_{l_2}$. Now applying $s_{j_{l_2}-1}\cdots s_{l_2}$ to $s_{i_{l_1}-1}\cdots s_{l_1}(\omega_{l_1})$ does not change anything and $\omega_{l_2}$ becomes $\omega_{l_2}-\epsilon_{l_2}+\epsilon_{j_{l_2}}$. By iterating these steps we derive the Weyl group element $w \in W$ with the desired properties.\\
Now consider a 2-step one-skeleton gallery $(E \supset V \subset F)$ that is positively folded. Without loss of generality we assume that $(E \supset V \subset F)$ is not minimal. Then there exists a minimal 2-step one-skeleton gallery $(E \supset V \subset F')$ such that we obtain $F$ from $F'$ by a positive folding. This means that there exist fundamental coweights $\omega_{i_1}$ and $\omega_{i_2}$ with $i_1 \geq i_2$ and Weyl group elements $w_1$ and $w_2$  such that $\delta_{w_1(E_{\omega_{i_1}})}*\delta_{w_1(E_{\omega_{i_2}})}$ is the minuscule presentation of $(E \supset V \subset F')$ and that $\delta_{w_1(E_{\omega_{i_1}})}*\delta_{w_2w_1(E_{\omega_{i_2}})}$ is the minuscule presentation of $(E \supset V \subset F)$. In type $A_n$ we can conclude that $w_2$ interchanges the coefficient of some $\epsilon_i$ with the coefficient of some $\epsilon_j$ with $j > i$. Let $w_2 \in W$ be the minimal Weyl group element with this property. Consider
$$ w_2w_1(\omega_{i_2})=  w_2(w_1(\omega_1))+w_2w_1(\epsilon_{{i_1}+1}+ \dots + \epsilon_{i_2}).$$
Because $w_2$ is minimal it interchanges only the coefficient of some $\epsilon_i$ with the coefficient of some $\epsilon_j$ with $j < i$. The desired property follows.\\
Conversely, let $T$ be a semistandard Young tableau where $C_1$ denotes the first and $C_2$ denotes the second column. Let $(E \supset V \subset F)$ be the associated 2-step one-skeleton gallery and $\delta_{w_1(\omega_{i_1})}*\delta_{w_2(\omega_{i_2})}$ the minuscule presentation. Let $w\in W$ be the minimal Weyl group element such that the Young tableau associated to the 2-step one-skeleton gallery $(E \supset V \subset w(F)=F')$ is minimal. Then $\delta_{w_1(E_{\omega_{i_1}})}*\delta_{w(w_2(E_{\omega_{i_2}}))}$ is the minuscule presentation of $(E \supset V \subset F')$. Because $T$ is semistandard and $w$ is minimal, applying $w$ to $F$ means changing the coefficient of $\epsilon_i$ with the coefficient of $\epsilon_j$ with $j>i$ in $w_2(\omega_{i_2})$ to obtain $w(w_2(\omega_{i_2}))$. Thus we know that applying $w^{-1}$ to $F'$ changes the coefficient of $\epsilon_i$ with the coefficient of $\epsilon_j$ with $j < i$. The claim follows.
\end{proof}
We now translate the notion of reflections on Young tableaux. Let $C$ be a column of a Young tableau of type $A_n$ with entries $\{i_1 < \dots < i_l\}$ and let $s_k \in W$ be a simple reflection. Then $s_k(C)$ is defined to be the column with entries $\{j_1 < \dots < j_l\}$  where $s_k(\epsilon_{i_1}+\dots+ \epsilon_{i_l})=\epsilon_{j_1}+\dots +\epsilon_{j_l}$. More precisely, a simple reflection $s_k$ interchanges the entries $k$ and $k+1$ for $k \in \{1, \dots, n\}$.\\
\subsection{Type $B_n$}
Let $\mathbb{A}$ be the standard apartment of the affine building of type $B_n$. The simple coroots are $\alpha_i\check{}=\epsilon_i-\epsilon_{i+1}$ for $i= 1, \dots, n-1$ and $\alpha_n\check{}=\epsilon_n$ where $\epsilon_i$ is the $i$th unit vector of $\mathbb{Z}^n$. For the fundamental coweights $\omega_i$ we choose Bourbaki enumeration i.e. $\omega_i=\epsilon_1+\dots+\epsilon_i$ for $i=1,\dots, n-1$ and $\omega_n=1/2(\epsilon_1+\dots+\epsilon_n)$ \cite{B}. The combinatorial one-skeleton galleries associated to the fundamental coweights are as follows: For the minuscule fundamental coweight $\omega_n$ we get $ \delta_{\omega_n}=\delta_{E_{\omega_n}}=(\mathfrak{o} \subset E_{\omega_n} \supset \omega_n)$. Since we are in type $B_n$ we have $\left\langle \omega_i, \beta \right\rangle \leq 2$ for all positive roots $\beta$. The fundamental coweights $\omega_i$  with $i\neq n$ are not minuscule, so there exists at least one positive root $\beta$ with $\left\langle \omega_i, \beta \right\rangle= 2$. Hence, $\delta_{\omega_i}=(\mathfrak{o} \subset E_0 \supset V \subset E_1 \supset \omega_i)$ where $E_0=E_{\omega_i}=\{t\omega_i \; \vert \; t \in \left[0,\frac{1}{2}\right]\}$, $ V=V_{\omega_i}=\frac{1}{2}\omega_i$ and $E_1=\{t\omega_i \; \vert \; t \in \left[\frac{1}{2},1\right]\}$. The combinatorial one-skeleton galleries of type $\omega_n$ are $(\mathfrak{o} \subset \tau(E_{\omega_n}) \supset \tau(\omega_n))$ with $\tau \in W$. For a non-minuscule fundamental coweight $\omega_i$ the fundamental one-skeleton galleries of type $\omega_i$ are $(\mathfrak{o} \subset \tau(E_0) \supset \frac{\tau(\omega_i)}{2} \subset \tau\sigma(E_0)+\frac{1}{2}\tau(\omega_i) \supset \frac{\tau(\omega_i)+ \tau\sigma(\omega_i)}{2})$ where $\tau \in W$ and $\sigma \in W_V^a$. Let $\lambda= \sum_{i=1}^n a_i \omega_i$ be a dominant coweight. Recall that $\delta_{\lambda}$ denotes the combinatorial one-skeleton gallery $\delta_{\lambda}=\delta_{a_1\omega_1}*\cdots*\delta_{a_n\omega_n}$. In this section we restrict ourselves to one-skeleton galleries $\delta=\delta_{w_0(E_{\omega_{i_0}})}*\cdots*\delta_{w_r(E_{\omega_{i_r}})}$ with $i_0\leq\dots\leq i_r$. Consider the action of the Weyl group $W$ on the coweights $X\check{}$: Let $\lambda$ be a coweight given in the basis $\{\epsilon_1, \dots, \epsilon_n\}$. Applying the simple reflection $s_i$ for $i=1, \dots, n-1$ to $\lambda$ interchanges the coefficient of $\epsilon_i$ and the coefficient of $\epsilon_{i+1}$ and fixes all other coefficients. Applying the simple reflection $s_n$ changes the sign of the coefficient of $\epsilon_n$ and fixes all other coefficients.\\
We also need to take a closer look at the Weyl group $W_V^v$ for some vertex $V$ in $\mathbb{A}$. It is sufficient to investigate the Weyl group $W_V^v$ where $V=\frac{1}{2}(\epsilon_{i_1}+\dots+\epsilon_{i_l})$ where $i_1 < \dots < i_l$ and $l < n$ and where $V=V_{\omega_n}$ since we obtain all other vertices in $\mathbb{A}$ by translating these vertices at coweights. For every vertex $V=\frac{1}{2}(\epsilon_{i_1}+\dots+\epsilon_{i_l})$ where $i_1 < \dots < i_l$ and $l < n$ in $\mathbb{A}$ the fundamental vertex $V_{\omega_l}$ is of the same type as $V$ such that the Weyl group $W_{V_{\omega_l}}^v$ and the Weyl group $W_V^v$ are isomorphic. More precisely, there exists a minimal element $\sigma \in \left\langle s_1,\dots, s_{n-1} \right\rangle\cong S^n$  such that $\sigma( \omega_l)=\frac{1}{2}(\epsilon_{i_1}+\dots +\epsilon_{i_l})$. Let $\{\beta_1, \dots, \beta_k\}$ be the set of simple roots for $\phi_{V_{\omega_l}}$ such that $C_{V_{\omega_l}}^-$ is the anti-dominant chamber of $\mathbb{A}_{V_{\omega_l}}$. Then $\{\sigma(\alpha_1),\dots,\sigma(\alpha_k)\}$ is the set of simple roots for $\phi_{V}$ such that $C_V^-$ is the anti-dominant chamber of $\mathbb{A}_V$. The simple reflections of $W_{\sigma(V_{\omega_l})}^v$ are $s_{\sigma({\alpha_i})}=\sigma  s_i \sigma^{-1}$ for $i \in \{1, \dots, k\}$. Because of this fact we have a bijection between the root system $\phi_{V}$ and $\phi_{V_{\omega_l}}$ in the following way: We identify the linearly ordered set $\{i_1<\dots< i_l\}$ with the linearly ordered set $\{1 <\dots < l\}$. Consequently, the description of the Weyl group $W_V^v$ reduces to the case where $V= V_{\omega_l}$ for $l=1, \dots, n$:\\
For the minuscule fundamental coweight $\omega_n$ the vertex $V_{\omega_n}$ is special and $W_{V_{\omega_n}}^v$ coincides with $W$. Now consider a non-minuscule fundamental coweight $\omega_i$ with $i\neq n$. The set of simple coroots such that $C_{V_{\omega_i}}^-$ is the anti-dominant chamber of the standard apartment of the residue building at $V_{\omega_i}$ is $\{\alpha_{i_0}\check{}=\epsilon_i, \alpha_1\check{},\dots, \alpha_{i-1}\check{}, \alpha_{i+1}\check{},\dots, \alpha_n\check{}\}.$ Let $\lambda$ be a coweight given in the basis $\{\epsilon_1, \dots, \epsilon_n\}$. Applying the simple reflection $s_{i_0}=s_{\alpha_{i_0}\check{}}$  to $\lambda$ changes the sign of the coefficient of $\epsilon_i$ and does not change any other coefficient.\\
\subsubsection{Young tableaux of type $B_n$}
Let $\lambda=a_1\omega_1+\dots+a_n\omega_n$ be a dominant coweight and define $p=(p_1, \dots,p_n)$ with $p_i=2a_i+\dots+ 2a_{n-1}+a_n$. Let $r$ be the smallest index with $p_{r+1}=0$. We associate to $\lambda$ a diagram consisting of $r$ left-aligned rows where the $i$th row consists of $p_i$ boxes (from top to bottom). In the following we denote the diagram by $dg(\lambda)$.\\
\textbf{Example}\\
For $n=3$ and $\lambda=\omega_1+\omega_2+\omega_3$ we obtain
$$ dg(\lambda)=\Yvcentermath1 \yng(5,3,1) \; .$$
A \textbf{Young tableau} $T$ of type $B_n$ of \textbf{shape} $\lambda$ is the diagram $dg(\lambda)$ where each box is filled with a letter of the linearly ordered alphabet $\{1 < \dots n < \bar{n} < \dots < \bar{1}\}$ such that the entries are strictly increasing in the columns and that never $i$ and $\bar{i}$ are in the same column. Additionally it holds that for each pair of columns $(C_{2j-1},C_{2j})$ with $j=1, \dots, a_1+\dots a_{n-1}$  either $C_{2j-1}=C_{2_j}$ or we obtain $C_{2j}$ from $C_{2j-1}$ by exchanging some entries $k$ by $\bar{k}$ and some $\bar{l}$ by some $l$ with $k,l \in \{1,\dots,n\}$. The coweight $\mu=\mu_1\epsilon_1+\dots\mu_n\epsilon_n$, where $2\mu_i=$ number of boxes with entry $i$ $-$ number of boxes with entry $\bar{i}$ is called the \textbf{content} of the Young tableau $T$.\\
\textbf{Example}\\
The Young tableau
$$ \young(11{{\bar{2}}}3{{\bar{3}}},{{\bar{3}}}2{{\bar{1}}},{{\bar{2}}})$$
has shape $\lambda$ and content $\mu=\frac{1}{2}(\epsilon_1-\epsilon_2-\epsilon_3)$.\\
Let $T$ be a Young tableau and $C_i$ denotes the $i$th column of $T$ for $i \in \{0,\dots,l\}$ (from left to right).
We call $T$ \textbf{minimal} if all entries of $C_i$ also appear in $C_{i-1}$ for $i \in \{1,\dots,l\}$.\\
A Young tableau $T$ is called \textbf{semistandard} if the entries are weakly increasing from left to right in the rows. Let SSYT($\lambda,\mu$) denote the set of all semistandard Young tableaux with shape $\lambda$ and content $\mu$.\\ 
We can identify one-skeleton galleries in $\mathbb{A}$ with type $\lambda$ for some dominant coweight $\lambda$ starting at the origin $\mathfrak{o}$ with Young tableaux of type $B_n$ of shape $\lambda$ as follows:\\
Let $\delta=\delta_{w_0(E_{\omega_{i_0}})}*\cdots*\delta_{w_r(E_{\omega_{i_r}})}$ be a one-skeleton gallery of type $\lambda= \frac{1}{2}(a_1\omega_1+ \dots + a_{n-1}\omega_{n-1})+a_n\omega_n$ where $a_k$ is the number of indices $j$ such that $E_{\omega_{i_j}}=E_{\omega_k}$. Define $\epsilon_{\bar{i}}$ to be $-\epsilon_i$.  Let $w_k(\omega_{i_k})= \epsilon_{k_1}+\dots +\epsilon_{k_j}$ for $i_k \neq n$ and $w_k(\omega_{i_k})=\frac{1}{2}(\epsilon_{k_1}+\dots +\epsilon_{k_j})$ for $i_k=n$. Then we associate to $\delta_{w_k(E_{\omega_{i_k}})}$ the column $C_{r-k}$ consisting of $j$ boxes filled with the letters $ 1\leq k_1 < \dots < k_j \leq \bar{1} $ in increasing order from top to bottom. The Young tableau $T_{\delta}=(C_{0},\dots,C_{r})$ that we associate to the one-skeleton gallery $\delta$ is the tableau that we obtain by putting the columns next to each other aligned at the top. This Young tableau $T$ has shape $\lambda$.\\
\textbf{Example}\\
For $n=3$ the corresponding Young tableau to the combinatorial one-skeleton gallery  $\delta=\delta_{s_3s_2s_1(E_{\omega_1})}*\delta_{s_2s_1(E_{\omega_1})}*\delta_{s_1s_2s_3s_2(E_{\omega_2})}*\delta_{s_2s_3s_2(E_{\omega_2})}*\delta_{E_{\omega_3}}$ is
$$ T_{\delta}=\Yvcentermath1 \young(1123{{\bar{3}}},2{{\bar{2}}}{{\bar{1}}},3).$$
This assignment is clearly a bijection between the set of Young tableaux of type $B_n$ of shape $\lambda$ for some dominant coweight $\lambda$ and the combinatorial one-skeleton galleries of type $B_n$ of type $\lambda$ starting in the origin $\mathfrak{o}$. Thus for a Young tableau $T=(C_0, \dots,C_r)$ of shape $\lambda$, the one-skeleton gallery $\delta_T=(\mathfrak{o} \subset E_0 \supset V_1 \subset  \dots \supset V_{r+1})$ denotes the associated combinatorial one-skeleton gallery of type $\lambda$ starting in the origin $\mathfrak{o}$ where $E_{r-j}$ is the edge corresponding to the column $C_j$ for $j \in \{0,\dots, r\}$. A 2-step one-skeleton gallery $(E \supset V \subset F)$ becomes a 2-column Young tableau at vertex $V$.\\ \\
As in type $A_n$ we have the following statement:
\begin{satz}
This assignment defines a bijection between the set of all minimal combinatorial one-skeleton galleries of type $\lambda$ starting in the origin $\mathfrak{o}$ and the set of all minimal Young tableaux of shape $\lambda$. Further, it also defines a bijection between the set of all positively folded combinatorial one-skeleton galleries of type $\lambda$ starting in the origin $\mathfrak{o}$ and semistandard Young tableaux of shape $\lambda$.
\end{satz}
\begin{proof}
See type $A_n$ for the first part of the proof. The second part can be found in \cite{GL1}.
\end{proof}
Now we need to translate the notion of reflections on Young tableaux. Let $C$ be a column of a Young tableau of type $B_n$ with entries $\{1\leq i_1<\dots< i_l \leq \bar{1}\}$. Let $\omega$ be a fundamental coweight and $s_{\sigma(\alpha_k)} \in W_{\sigma(V_{\omega})}^v$ a simple reflection where $\sigma \in \left\langle s_1, \dots, s_{n-1}\right\rangle \cong S^n$. Then $s_{\sigma(\alpha_k)}(C)$ is defined to be the column with entries $\{1\leq j_1<\dots< j_l \leq \bar{1}\}$ where $s_{\sigma(\alpha_k)}(\epsilon_{i_1}+\dots+\epsilon_{i_l})=\epsilon_{j_1}+\dots+\epsilon_{j_l}$. More precisely, consider $\sigma=id$: Applying $s_k$ for $k \in \{1, \dots, n-1\}$ interchanges the entry $k$ with $k+1$ and $\bar{k}$ with $\overline{k+1}$, $s_n$ interchanges the entry $n$ with $\bar{n}$ and $s_{i_0}$ for $i\in \{1,\dots,n-1\}$ interchanges $i$ with $\bar{i}$. For $\sigma\neq id$ we have: Applying $s_{\sigma(\alpha_k)}$ for $k \in \{1, \dots, n-1\}$ interchanges the entry $\sigma(k)$ with $\sigma(k+1)$ and $\overline{\sigma(k)}$ with $\overline{\sigma(k+1)}$, $s_{\sigma(\alpha_n)}$ interchanges the entry $\sigma(n)$ with $\overline{\sigma(n)}$ and $s_{\sigma(\alpha_{i_0})}$ for $i\in \{1,\dots,n-1\}$ interchanges $\sigma(i)$ with $\overline{\sigma(i)}$ where we identify the elements of $ \left\langle s_1, \dots, s_{n-1}\right\rangle $ in the natural way with the elements of the symmetric group $S^n$. \\
\subsection{Type $C_n$}
Let $\mathbb{A}$ be the standard apartment of the affine building of type $C_n$. The simple coroots are $\alpha_i\check{}=\epsilon_i-\epsilon_{i+1}$ for $i= 1, \dots, n-1$ and $\alpha_n\check{}=2\epsilon_n$ where $\epsilon_i$ is the $i$th unit vector of $\mathbb{Z}^n$. For the fundamental coweights $\omega_i$ we choose Bourbaki enumeration i.e. $\omega_i=\epsilon_1+\dots+\epsilon_i$ for $i=1,\dots, n-1$ and $\omega_n=\epsilon_1+\dots+\epsilon_n$ \cite{B}. The combinatorial one-skeleton galleries associated to the fundamental coweights are as follows: For the minuscule fundamental coweight $\omega_1$ we get $ \delta_{\omega_1}=\delta_{E_{\omega_1}}=(\mathfrak{o} \subset E_{\omega_1} \supset \omega_1)$. Since we are in type $C_n$ we have $\left\langle \omega_i, \beta \right\rangle \leq 2$ for all positive roots $\beta$. The fundamental coweights $\omega_i$  with $i\neq 1$ are not minuscule, so there exists at least one positive root $\beta$ with $\left\langle \omega_i, \beta \right\rangle= 2$. Hence, $\delta_{\omega_i}=(\mathfrak{o} \subset E_0 \supset V \subset E_1 \supset \omega_i)$ where $E_0=E_{\omega_i}=\{t\omega_n \; \vert \; t \in \left[0,\frac{1}{2}\right]\}$, $ V=V_{\omega_i}=\frac{1}{2}\omega_i$ and $E_1=\{t\omega_n \; \vert \; t \in \left[\frac{1}{2},1\right]\}$. The combinatorial one-skeleton galleries of type $\omega_1$ are $(\mathfrak{o} \subset \tau(E_{\omega_1}) \supset \tau(\omega_1))$ with $\tau \in W$. For a non-minuscule fundamental coweight $\omega_i$ the combinatorial one-skeleton galleries of type $\omega_i$ are $(\mathfrak{o} \subset \tau(E_0) \supset \frac{\tau(\omega_i)}{2} \subset \tau\sigma(E_0)+\frac{1}{2}\tau(\omega_i) \supset \frac{\tau(\omega_i)+ \tau\sigma(\omega_i)}{2})$ where $\tau \in W$ and $\sigma \in W_V^a$. Let $\lambda= \sum_{i=1}^n a_i \omega_i$ be a dominant coweight. Recall that $\delta_{\lambda}$ denotes the combinatorial one-skeleton gallery $\delta_{\lambda}=\delta_{a_1\omega_1}*\cdots*\delta_{a_n\omega_n}$. In this section we restrict ourselves to one-skeleton galleries $\delta=\delta_{w_0(E_{\omega_{i_0}})}*\cdots*\delta_{w_r(E_{\omega_{i_r}})}$ with $i_0\leq\dots\leq i_r$. Consider the action of the Weyl group $W$ on the coweights $X\check{}$: Let $\lambda$ be a coweight given in the basis $\{\epsilon_1, \dots, \epsilon_n\}$. Applying the simple reflection $s_i$ for $i=1, \dots, n-1$ to $\lambda$ interchanges the coefficient of $\epsilon_i$ and the coefficient of $\epsilon_{i+1}$ and fixes all other coefficients. Applying the simple reflection $s_n$ changes the sign of the coefficient of $\epsilon_n$ and fixes all other coefficients.\\
We also need to take a closer look at the Weyl group $W_V^v$ at a vertex $V$ for some vertex $V$ in $\mathbb{A}$. It is sufficient to investigate the Weyl group $W_V^v$ at $V$ where $V=\frac{1}{2}(\epsilon_{i_1}+\dots+\epsilon_{i_l})$ where $i_1 < \dots < i_l$ and $ 1<l $ and where $V=V_{\omega_1}$ since we obtain all other vertices in $\mathbb{A}$ by translating these vertices at coweights. For every vertex $V=\frac{1}{2}(\epsilon_{i_1}+\dots+\epsilon_{i_l})$ where $i_1 < \dots < i_l$ and $1 < l$ in $\mathbb{A}$ the fundamental vertex $V_{\omega_l}$ is of the same type as $V$ such that the Weyl group $W_{V_{\omega_l}}^v$ and the Weyl group $W_V^v$ are isomorphic. More precisely, there exists a minimal element $\sigma \in \left\langle s_1, \dots, s_{n-1}\right\rangle\cong S^n$ such that $\sigma( \omega_l)=\frac{1}{2}(\epsilon_{i_1}+\dots +\epsilon_{i_l})$. Let $\{\alpha_1, \dots, \alpha_k\}$ be the set of simple roots for $\phi_{V_{\omega_l}}$ such that $C_{V_{\omega_l}}^-$ is the anti-dominant chamber of $\mathbb{A}_{V_{\omega_l}}$. Then $\{\sigma(\alpha_1),\dots,\sigma(\alpha_k)\}$ is the set of simple roots for $\phi_{V}$ such that $C_V^-$ is the anti-dominant chamber of $\mathbb{A}_V$. The simple reflections of $W_{\sigma(V_{\omega_l})}$ are $s_{\sigma({\alpha_i})}=\sigma  s_i \sigma^{-1}$ for $i \in \{1, \dots, k\}$. Because of this fact we have a bijection between the root system $\phi_{V}$ and $\phi_{V_{\omega_l}}$ in the following way: We identify the linearly ordered set $\{i_1<\dots< i_l\}$ with the linearly ordered set $\{1 <\dots < l\}$. Consequently, the description of the Weyl group $W_V^v$ reduces to the case where $V= V_{\omega_l}$ for $l=1, \dots, n$:\\
For the minuscule fundamental coweight $\omega_1$ the vertex $V_{\omega_1}$ is special so $W$ and $W_{V_{\omega_1}}^v$ are the same. Now consider a non-minuscule fundamental coweight $\omega_i$ with $i\neq 1$. The set of simple coroots such that $C_{V_{\omega_i}}^-$ is the anti-dominant chamber of the standard apartment of the residue building at $V_{\omega_i}$ is $\{\alpha_{i_0}\check{}=\epsilon_{i-1}+\epsilon_i, \alpha_1\check{},\dots, \alpha_{i-1}\check{}, \alpha_{i+1}\check{},\dots, \alpha_n\check{}\}.$ Let $\lambda$ be a coweight given in the basis $\{\epsilon_1, \dots, \epsilon_n\}$. Applying the simple reflection $s_{i_0}=s_{\alpha_{i_0}\check{}}$  to $\lambda$ interchanges the coefficient of $\epsilon_{i-1}$ with the coefficient of $\epsilon_i$ and the sign and does not change any other coefficient.\\
\subsubsection{Young tableaux of type $C_n$}
Let $\lambda=a_1\omega_1+\dots+a_n\omega_n$ be a dominant coweight and define $p=(p_1, \dots,p_n)$ with $p_1=a_1+\sum_{j=2}^{n}2a_j$ and  for $i\geq 2$ with $p_i=2a_i+\dots+ 2a_n$. Let $r$ be the smallest index with $p_{r+1}=0$. We associate to $\lambda$ a diagram consisting of $r$ left-aligned rows where the $i$th row consists of $p_i$ boxes (from top to bottom). In the following we denote the diagram by $dg(\lambda)$.\\
\textbf{Example}\\
For $n=3$ and $\lambda=\omega_1+\omega_2+\omega_3$ we obtain
$$ dg(\lambda)=\Yvcentermath1 \yng(5,4,2) \; .$$
A \textbf{Young tableau} $T$ of type $C_n$ of \textbf{shape} $\lambda$ is the diagram $dg(\lambda)$ where each box is filled with a letter of the linearly ordered alphabet $\{1 < \dots n < \bar{n} < \dots < \bar{1}\}$ such that the entries are strictly increasing in the columns and that never $i$ and $\bar{i}$ are in the same column. Additionally it holds that for each pair of columns $(C_{a_1+2j-1},C_{a_1+2j})$ with $j=1, \dots, a_2+\dots a_{n}$  either $C_{a_1+2j-1}=C_{a_1+2j}$ or we obtain $C_{a_1+2j}$ from $C_{a_1+2j-1}$ by exchanging for an even number of times some entries $k$ by $\bar{k}$ and some $\bar{l}$ by some $l$ with $k,l \in \{1,\dots,n\}$.
The coweight $\mu=\mu_1\epsilon_1+\dots\mu_n\epsilon_n$, where $2\mu_i=$ number of boxes with entry $i$ $-$ number of boxes with entry $\bar{i}$ $+$ number of boxes with entry $i$ in a column with a single box $-$ number of boxes with entry $\bar{i}$ in a column with a single box, is called the \textbf{content} of the Young tableau $T$.\\
\textbf{Example}\\
The Young tableau
$$ \young(122{{\bar{3}}}1,2{{\bar{3}}}3{{\bar{2}}},3{{\bar{1}}})$$
has shape $\lambda$ and content $\mu=\epsilon_1+\epsilon_2$.\\ \\
Let $T$ be a Young tableau and $C_i$ denotes the $i$th column of $T$ for $i \in \{0,\dots,l\}$ (from left to right).
We call $T$ \textbf{minimal} if all entries of $C_i$ also appear in $C_{i-1}$ for $i \in \{1,\dots,l\}$.\\
A Young tableau $T$ is called \textbf{semistandard} if the entries are weakly increasing from left to right in the rows. Let SSYT($\lambda,\mu$) denote the set of all semistandard Young tableaux with shape $\lambda$ and content $\mu$.\\ 
We can identify one-skeleton galleries of type $\lambda$ for some dominant coweight $\lambda$ in $\mathbb{A}$ with Young tableaux of type $C_n$ of shape $\lambda$ as follows:\\
Let $\delta=\delta_{w_0(E_{\omega_{i_0}})}*\cdots*\delta_{w_r(E_{\omega_{i_r}})}$ be a one-skeleton gallery of type $\lambda= a_1+\frac{1}{2}(a_2\omega_1+ \dots + a_{n-1}\omega_{n})$ where $a_k$ is the number of indices $j$ such that $E_{\omega_{i_j}}=E_{\omega_k}$. Define $\epsilon_{\bar{i}}$ to be $-\epsilon_i$.  Let $w_k(\omega_{i_k})= \epsilon_{k_1}+\dots +\epsilon_{k_j}$ for $i_k \neq n$ and $w_k(\omega_{i_k})=\epsilon_{k_1}+\dots +\epsilon_{k_j}$ for $i_k=n$. Then we associate to $\delta_{w_k(E_{\omega_{i_k}})}$ the column $C_{r-k}$ consisting of $j$ boxes filled with the letters $ 1\leq k_1 < \dots < k_j \leq \bar{1} $ in increasing order from top to bottom. The Young tableau $T_{\delta}=(C_{0},\dots,C_{r})$ that we associate to the one-skeleton gallery $\delta$ is the tableau that we obtain by putting the columns next to each other aligned at the top. This Young tableau has shape $\lambda$.\\
\textbf{Example}\\
For $n=3$ the corresponding Young tableau to the one-skeleton gallery  $\delta=\delta_{s_1s_2s_3s_2s_1(E_{\omega_1})}*\delta_{s_2s_3s_1s_2(E_{\omega_2})}*\delta_{s_3s_1s_2(E_{\omega_2})}*\delta_{s_3s_2s_3(E_{\omega_3})}*\delta_{E_{\omega_3}}$ is
$$ T_{\delta}=\Yvcentermath1 \young(1123{{\bar{1}}},2{{\bar{3}}}{{\bar{3}}}{{\bar{2}}},3{{\bar{2}}}).$$
This assignment is clearly a bijection between the set of Young tableaux of type $C_n$ and shape $\lambda$ for some dominant coweight $\lambda$ and the combinatorial one-skeleton galleries of type $C_n$ and type $\lambda$ starting in the origin $\mathfrak{o}$. Thus for a Young tableau $T=(C_0, \dots,C_r)$ of shape $\lambda$, the one-skeleton gallery $\delta_T=(\mathfrak{o} \subset E_0 \supset V_1 \subset  \dots \supset V_{r+1}))$ denotes the associated combinatorial one-skeleton gallery of type $\lambda$ starting in the origin $\mathfrak{o}$ where $E_{r-j}$ is the edge corresponding to the column $C_j$ for $j \in \{0,\dots, r\}$. A 2-step one-skeleton gallery $(E \supset V \subset F)$ becomes a 2-column Young tableau at vertex $V$.\\ \\
As in type $A_n$ we have the following statement:
\begin{satz}
This assignment defines a bijection between the set of all minimal combinatorial one-skeleton galleries of type $\lambda$ starting in the origin $\mathfrak{o}$ and the set of all minimal Young tableaux of shape $\lambda$. Further, it also defines a bijection between the set of all positively folded combinatorial one-skeleton galleries of type $\lambda$ starting in the origin $\mathfrak{o}$ and semistandard Young tableaux of shape $\lambda$.
\end{satz}
\begin{proof}
See type $A_n$ for the first part of the proof. The second part can be found in \cite{GL1}.
\end{proof}
Now we need to translate the notion of reflections on Young tableaux. Let $C$ be a column of a Young tableau of type $C_n$ with entries $\{1\leq i_1<\dots< i_l \leq \bar{1}\}$. Let $\omega$ be a fundamental coweight and $s_{\sigma(\alpha_k)} \in W_{\sigma(V_{\omega})}^v$ a simple reflection where $\sigma \in \left\langle s_1, \dots, s_{n-1}\right\rangle$. Then $s_{\sigma(\alpha_k)}(C)$ is defined to be the column with entries $\{1\leq j_1<\dots< j_l \leq \bar{1}\}$ where $s_{\sigma(\alpha_k)}(\epsilon_{i_1}+\dots+\epsilon_{i_l})=\epsilon_{j_1}+\dots+\epsilon_{j_l}$. More precisely, consider $\sigma=id$: Applying $s_k$ for $k \in \{1, \dots, n-1\}$ interchanges the entry $k$ with $k+1$ and $\bar{k}$ with $\overline{k+1}$, $s_n$ interchanges the entry $n$ with $\bar{n}$ and $s_{i_0}$ for $i\in \{1,\dots,n-1\}$ interchanges $i$ with $\overline{i-1}$ and $i-1$ with $\bar{i}$. For $\sigma\neq id$ we have: Applying $s_{\sigma(\alpha_k)}$ for $k \in \{1, \dots, n-1\}$ interchanges the entry $\sigma(k)$ with $\sigma(k+1)$ and $\overline{\sigma(k)}$ with $\overline{\sigma(k+1)}$, $s_{\sigma(\alpha_n)}$ interchanges the entry $\sigma(n)$ with $\overline{\sigma(n)}$ and $s_{\sigma(\alpha_{i_0})}$ for $i\in \{1,\dots,n-1\}$ interchanges $\sigma(i)$ with $\overline{\sigma(i-1)}$ and $\sigma(i-1)$ with $\overline{\sigma(i)}$ where we identify the elements in $\left\langle s_1, \dots, s_{n-1}\right\rangle$ in the natural way with the elements of the symmetric group $S^n$.
\subsection{Combinatorial formula}
A natural question now is whether it is possible to calculate the contribution $c(\delta)$ of a positively folded combinatorial one-skeleton gallery $\delta$ to the Gaussent-Littelmann formula for type $A_n$, $B_n$ and $C_n$ only with the associated semistandard Young tableau $T_{\delta}$. It turns out that the recurrence in Theorem \ref{rec} leads to a very simple algorithm how to do this:\\
In order to explain the algorithm we first need some more notation:\\ \\
We say a simple reflection $s_k$ in $W_V^v$ for some vertex $V$ increases (resp. decreases) a column $C$ if there is at least one entry in $C$ that is increased (resp. decreased) by applying $s_k$.\\ \\
Let now $(E \supset V \subset F)$ be a 2-step one-skeleton gallery in $\mathbb{A}$ and $T=(C_F,C_E)$ be the associated 2-column Young tableau at vertex $V$. Consider $(s_j(E) \supset V \subset s_j(F))$ for a simple reflection $s_j$ in $W_V^a$. We denote the associated Young tableau by $s_j(T)$. Clearly, $s_j(T)=(s_j(C_F),s_j(C_E))$, where $s_j$ on the right hand side is the simple reflection in $W_V^v$. We can also consider $(E \supset V \subset s_j(F))$. We denote the associated Young tableau by $id_j(T)$ and again we have $id_j(T)=(s_j(C_F), C_E)$, where $s_j$ on the right hand side is the simple reflection in $W_V^v$.\\ \\
Let $\delta=(\mathfrak{o}=V_0 \subset E_0 \supset \dots \supset V_{r+1})$ be a positively folded combinatorial one-skeleton gallery with type $\lambda$ and target $\mu$. Recall that the contribution of this gallery to the Gaussent-Littelmann formula for $L_{\lambda,\mu}$ is a product of contributions $c((E_{j-1} \supset V_j \subset E_j))$  for $j \in \{1,\dots,r\}$ and $c(E_0)$. We now explain how to compute $c((E_{j-1} \supset V_j \subset E_j))$  for $j \in \{1,\dots,r\}$ in a very simple way with the associated Young tableau:\\ \\
Let therefore $\delta= (E \supset V \subset F)$ be a positively folded 2-step one-skeleton gallery. Let $T_{\delta}=(C_F,C_E)$ be the associated semistandard 2-column Young tableau at vertex $V$.\\ \\
We build a tree $G_{\delta}$ where the vertices are 2-column semistandard Young tableaux and the root is $T_{\delta}$. Let $T=(C_1,C_2)$ be a semistandard 2-column Young tableau at vertex $V$:\\ \\
\textbf{Step 1:} Find a simple reflection $s_j \in W_V^v$ that increases $C_1$. If $s_j(T)$ is still semistandard then $s_j(T)$ is a vertex of the tree connected to the vertex $T$ by an edge. If $s_j(T)$ is not semistandard then $id_j(T)$ is semistandard and $id_j(T)$ is a vertex of the tree connected to the vertex $T$ by an edge.\\ \\
\textbf{Step 2:} Label the edge created in the first step as follows: If the edge connects $T$ and $id_j(T)$ label the edge with $id_j^+$. If the edge connects $T$ and $s_j(T)$ the labeling depends on $s_j(C_2)$:\\ We label the edge with
$\begin{cases}
     s_j^+,  & \text{if } s_j \text{ increases }C_2 \text{ or if } s_j(C_2)=C_2 \\
  s_j^-, & \text{if }s_j\text{ decreases }C_2.
\end{cases}$\\ \\
\textbf{Step 3:} If we have labeled an edge in the second step with $s_j^-$ then the tree branches at the vertex $T$ as follows: $id_j(T)$ is also a vertex of the tree connected to $T$ by an edge. This edge is labeled with $id^+_j$.\\ \\
Build the tree $G_{\delta}$ starting by applying step 1 to step 3 to the vertex $T_{\delta}$ and go on by applying step 1 to step 3 to the new created vertices using the same simple reflection in step 1 for all tableaux with the same first column and so on. This procedure stops when there is no simple reflection $s_j$ that increases $C_1$ or in other words when the edge corresponding to the column $C_1$ in the associated 2-step one-skeleton gallery $\delta_T$ is contained in $C_V^-$. We call these vertices final. We denote the subset of all simple paths in $G_{\delta}$ starting at the root $T_{\delta}$ and ending at a final vertex by $F_{\delta}$.\\
\begin{defi}\label{C_i}
$$    c(C_F,C_E)= \sum_{\sigma  \in F_{\delta}}q^{pr(\sigma)}(q-1)^{pf(\sigma)},$$
where $pr(\sigma)$ is the number of edges in $\sigma$ labeled with an $s_j^+$ for some $j$ and $pf(\sigma)$ is the number of edges in $\sigma$ labeled with an $id_j^+$ for some $j$.
\end{defi}
\begin{epr}
$$ c((E \supset V \subset F))=c(C_F,C_E).$$  
\end{epr}
\begin{proof}
Consider $c((E \supset V \subset F))$. Choosing a reduced decomposition $s_{i_1}\dots s_{i_m}$ in $W_V^a$ for $w_D$ is equivalent to choose the sequence $s_{i_1}, \dots, s_{i_m}$ with $s_{i_j}$ increases the column $ s_{i_{j-1}}\dots s_{i_1}(C_F)$ for every $j \in \{1, \dots, m\}$ and there does not exist a simple reflection that increases the column $s_{i_{m}}\dots s_{i_1}(C_F)$ further. Fix a reduced decomposition $s_js_{i_1}\dots s_{i_m}$ for $w_D$. We now want to apply the recurrence in Theorem \ref{rec} and Lemma \ref{pm1} and \ref{pm2}. Therefore we need to consider two cases:\\
\textbf{1. case: $l(s_jw_{S_V})=l(w_{S_V})+1$}\\
We know that all galleries of chambers in $\Gamma^+_{S_V}((E \supset V \subset F),(j,i_1, \dots, i_m))$ must have a crossing in the first step and that we therefore obtain all galleries of
$$\Gamma^+_{S_V}((E \supset V \subset F),(j,i_1, \dots, i_m))$$
by applying $f_1^j$ to the galleries of $\Gamma^+_{s_j(S_V)}((s_j(E) \supset V \subset s_j(F)), (i_1, \dots, i_m))$. This means that we extend the galleries of chambers in the first step by a crossing of type $j$. In this case the crossing is positive and with Theorem \ref{rec} and Lemma \ref{pm1} it follows that we obtain $c((E \supset V \subset F))$ by multiplying every contribution of an element in $\Gamma^+_{s_j(S_V)}((s_j(E) \supset V \subset s_j(F)),(i_1, \dots, i_m))$, and therefore the whole product $c((s_j(E) \supset V \subset s_j(F)))$, by $q$.\\
\textbf{2. case: $l(s_jw_{S_V})=l(w_{S_V})-1$}\\
Then we know by Theorem \ref{rec} that we obtain all galleries in
$$\Gamma_{S_V}^+((E \supset V \subset F),(j,i_1, \dots, i_m))$$
by applying $f_1^j$ to all galleries in $\Gamma^+_{s_j(S_V)}((s_j(E) \supset V \subset s_j(F)),(i_1, \dots, i_m))$ and $f_2^j$ to all galleries in $\Gamma^+_{S_V}((E \supset V \subset s_j(F)),(i_1, \dots, i_m))$.
Applying $f_1^j$ means extending the galleries of chambers in the first step by a crossing of type $j$. In this case the crossing is negative. Applying $f_2^j$ extends the galleries by a positive folding in the first step. Hence, with Theorem \ref{rec} and with Lemma \ref{pm1} and \ref{pm2}, we obtain
$$c((E \supset V \subset F))= c((s_j(E) \supset V \subset s_j(F)))+(q-1)c((E \subset V \subset s_j(F))).$$
It remains to explain how to decide whether we are in the first or in the second case just by consideration of the
associated Young tableau: If $s_j$ increases $C_2$ we are clearly in the first case because the face corresponding to $C_2$ and therefore also $S_V$ is on the same side of the hyperplane corresponding to the reflection $s_j$ as $C_V^-$, if $s_j$ decreases $C_2$ we are clearly in the second case because the face corresponding to $C_2$ and therefore also $S_V$ is on the different side of the hyperplane corresponding to $s_j$ independent of the choice of $S_V$. If $s_j$ does not change the column $C_2$ this means that the face corresponding to $C_2$ is contained in the hyperplane corresponding to $s_j$ thus the case in which we are depends on our choice of $S_V$. Recall that the Gaussent-Littelmann formula is independent of the choice of $S_V$. We choose $S_V$ to be on the same side of the hyperplane corresponding to the simple reflection $s_j$ as $C_V^-$, thus we are in the first case.  The claim follows by induction over the length $l(F)$.
\end{proof}
\begin{rem}
\begin{itemize}
     \item[(1)] Note that $c(C_F,C_E)$ is independent of the choice of the simple reflection $s_j$ in the first step because the Gaussent-Littelmann formula is independent of the choice of the reduced expression for $w_D$.
     \item[(2)] We want to point out that the simple paths in the tree $F_{\delta}$ do not correspond to the galleries of chambers in the Gaussent-Littelmann formula. The reason for this is that we make a new choice of chamber $S_V$ in every step of the induction. In fact, not keeping track of where $S_V$ is reflected to makes our algorithm as simple as it is.
      \item[(3)]
In the proof above, if $s_j(E)=E$, we have chosen $S_V$ to be on the same side of the hyperplane corresponding to $s_j$ as $C_V^-$ because it creates a tree as simple as possible. Let us explain this in more detail: In every step of the induction of the proof where the simple reflection $s_j$ does not change the second column of the tableau $T$ we can choose whether $S_V$ lies on the same side of the hyperplane corresponding to the reflection $s_j$ as $C_V^-$ column or not. In our construction of the tree this means that we can decide if we label the edge connecting $T$ and $s_j(T)$ with $s_j^+$ or $s_j^-$. Since the tree branches everytime we choose $s_j^-$ we obtain a tree as simple as possible if we choose $S_V$ in every step to be on the same side of the hyperplane corresponding to the reflection $s_j$.
\end{itemize}
\end{rem}
For the sake of completeness we now explain how to compute $c(E_0)$ with the associated column $C_{r}$: As already mentioned in the previous proof to choose a reduced expression $s_{i_1}\dots s_{i_k}$ for the Weyl group element $w_{D_0} \in W_{\mathfrak{o}}^v=W$ that sends $C_{{\mathfrak{o}}}^-$ to the chamber $D_0$ of $\mathbb{A}_{\mathfrak{o}}$ that contains $E_0$ and is closest to $C_{{\mathfrak{o}}}^-$ is equivalent to choosing a sequence $s_{i_1}, \dots, s_{i_k}$ with $s_{i_j}$ increases the column $s_{i_{j-1}}\dots s_{i_1}(C_r)$ for every $j \in \{1, \dots, m\}$. Thus,
$$ c(E_0)=c(C_r)= q^k.$$
Summing up we obtain the following formula:\\
\begin{satz}{Combinatorial version of the Gaussent-Littelmann formula}\\
Let $\lambda$ and $\mu$ be dominant coweights. Then
$$ L_{\lambda, \mu}(q)=\sum_{T \in SSYT(\lambda, \mu)} c(C_r)\prod_{i=0}^{r-1}c(C_i,C_{i+1}),$$
for $T=(C_0, \dots, C_r)$  where $c(C_r)$ is as above and $c(C_{i},C_{i+1})$ as in definition \ref{C_i} for $i \in \{0, \dots, r-1\}$.
\end{satz}
We define 
$$c(T)=c(C_r)\prod_{i=0}^{r-1}c(C_i,C_{i+1}).$$
\begin{rem}
For the algorithm it is not required to know exactly at what vertex $V$ the 2-column Young tableau is. It is sufficient to know at what type of vertex it is:\\
Starting with a Young tableau $T=(C_0,\dots, C_r)$ we only need to calculate the type of the vertex at which the 2-column Young tableau $(C_i,C_{i+1})$ is for $i \in \{0, \dots, n-1\}$: 
\begin{itemize}
    \item[(1)] In \textbf{type $A_n$}: All vertices are special and therefore for all 2-column Young tableau $(C_i,C_{i+1})$ at a vertex $V$ the Weyl group $W_V^v$ is $W$.
    \item[(2)] In \textbf{type $B_n$}: If the columns $C_i$ and $C_{i+1}$ both have $j$ boxes where $j$ is stricty smaller then $n$ and $r-i$ is odd, then the 2-column Young tableau $(C_i,C_{i+1})$ is at a vertex $V$ of the same type as the fundamental vertex $V_{\omega_j}$ in $\omega_j$-direction. Let $\{k_1, \dots k_{j-s}, \overline{l_1},\dots, \overline{l_{s}}\}$ be the set of entries of $C_i$. Let $\sigma \in \left\langle s_1, \dots s_{n-1}\right\rangle\cong S^n$ be the permutation that identifies the linearly ordered set $\{1<\dots< j\}$ with the set $\{k_1, \dots, k_{j-s}, l_1, \dots, l_s\}$ in ascending order. Then the Weyl group $W_V^v$ is $W_{\sigma(V_{\omega_j})}^v$. Recall that the simple reflections of $W_{\sigma(V_{\omega_j})}^v$ are $s_{\sigma(\alpha_m)}$ where $m \in \{1, \dots j-1,j_0,j+1, \dots, n\}$. Applying $s_{\sigma(\alpha_m)}$ for $m \in \{1, \dots, j-1,j+1, \dots, n-1\}$ to the column $C_i$ interchanges the entry $\sigma(m)$ with $\sigma(m+1)$ and $\overline{\sigma(m)}$ with $\overline{\sigma(m+1)}$ and applying $s_{\sigma(\alpha_n)}$ interchanges $\sigma(n)$ and $\overline{\sigma(n)}$. Applying $s_{\sigma(\alpha_{j_0})}$ interchanges $\sigma(j)$, which is the highest element in the set $\{k_1, \dots, k_{j-s}, l_1, \dots, l_s\}$ with $\overline{\sigma(j)}$. All other 2-column Young tableaux are at a special vertex.
    \item[(3)] In \textbf{type $C_n$}: If the columns $C_i$ and $C_{i+1}$ both have $j$ boxes where $j$ is bigger then $1$ and $i$ is even, then the 2-column Young tableau $(C_i,C_{i+1})$ is at a vertex $V$ of the same type as the fundamental vertex $V_{\omega_j}$ in $\omega_j$-direction. Let $\{k_1, \dots k_{j-s}, \overline{l_1},\dots, \overline{l_{s}}\}$ be the set of entries of $C_i$. Let $\sigma \in \left\langle s_1, \dots, s_{n-1}\right\rangle\cong S^n$ be the permutation that identifies the linearly ordered set $\{1<\dots< j\}$ with the set $\{k_1, \dots, k_{j-s}, l_1, \dots, l_s\}$ in ascending order. Then the Weyl group $W_V^v$ is $W_{\sigma(V_{\omega_j})}^v$. Recall that the simple reflections of $W_{\sigma(V_{\omega_j})}^v$ are $s_{\sigma(\alpha_m)}$ where $m \in \{1, \dots j-1,j_0,j+1, \dots, n\}$. Applying $s_{\sigma(\alpha_m)}$ for $m \in \{1, \dots, j-1,j+1, \dots, n-1\}$ to the column $C_i$ interchanges the entry $\sigma(m)$ with $\sigma(m+1)$ and $\overline{\sigma(m)}$ with $\overline{\sigma(m+1)}$ and applying $s_{\sigma(\alpha_n)}$ interchanges $\sigma(n)$ and $\overline{\sigma(n)}$. Applying $s_{\sigma(\alpha_{j_0})}$ interchanges $\sigma(j)$ which is the highest element in the set $\{k_1, \dots, k_{j-s}, l_1, \dots, l_s\}$ with $\overline{\sigma(j-1)}$ and $\sigma(j-1)$ with $\overline{\sigma(j)}$. All other 2-column Young tableaux are at a special vertex.
\end{itemize}
\end{rem}
\begin{rem}
The combinatorial Gaussent-Littelmann formula is recursively defined. It is desirable to have a closed formula. In the next chapter we prove that the well-known Macdonald formula for Hall-Littlewood polynomials for type $A_n$ is the closed formula for our combinatorial Gaussent-Littelmann formula for type $A_n$. Finding a closed formula for type $B_n$ and $C_n$ seems much more complicated. To see this consider for example the combinatorial formula for type $A_n$: For every simple path $\sigma$ in $F_{\delta}$ the number $pf(\sigma)$ is the same. Namely, it is the number of boxes in $C_{i+1}$ such that the entry of this box is not an entry of a box in $C_i$. We do not have this property in the other two types. 
\end{rem}
\section{Macdonald formula}
Our point of departure in this section is the Macdonald formula as it is presented in \cite{Mac1}. This formula is a sum over all semistandard Young tableaux of type $A_n$ and the computation uses the so-called $\lambda$-chain associated to a semistandard Young tableau. For our purposes we need to restate it explicitly in terms of Young tableaux i.e. in terms of boxes and entries of these boxes.\\ 
We work with a root sytem of type $A_n$ throughout this section. First we need some more definitions and notation as in \cite{Mac1}: \\
Let $\lambda \in X\check{}_+$ be a dominant coweight. Consider the diagram dg($\lambda$). By reflecting in the main diagonal (from top-left to bottom-right) we obtain a new diagram. We refer to the dominant coweight of the new diagram as the conjugate of $\lambda$ and denote it by $\lambda'$. Note that $\lambda_i'$ is the number of boxes in the $(i-1)$th column of dg($\lambda$) (the 0th column is left). Define $$ m_i(\lambda)= \lambda_i'-\lambda_{i+1}'.$$
Let $\lambda$ and $\mu$ be dominant coweights with $\lambda \supset \mu$ i.e. $\lambda_i \geq \mu_i$ for all $i$. In other words, the diagram dg($\lambda$) contains the whole diagram dg($\mu$). The skew-diagram dg($\lambda - \mu$) is what we get when we cancel out all boxes in dg($\lambda$) that also appear in dg($\mu$). The coefficient $\lambda_i-\mu_i$ is again the number of boxes in the $i$th row of the diagram dg($\lambda-\mu$) and $\lambda_i'-\mu_i'$ is the number of boxes in the $i$th column. Let $T$ be a semistandard Young tableau of shape $\lambda$. Define dg($\lambda^{(i)}$) to be the diagram consisting of all boxes in $T$ with entries $\leq i$ for $i  \in \{0, \dots ,n+1\}$ and where $\lambda^{(i)}_j$ is again the number of boxes in the $j$th column of dg($\lambda^{(i)}$). We get the following chain: 
 $$\lambda^{(0)} \subset \dots \subset \lambda^{(n+1)}= \lambda.$$ 
We are now in the position to state Macdonald's formula:
\begin{satz}{\cite{Mac1} Macdonald formula for Hall-Littlewood polynomials of type $A_n$.}\\
Let $\lambda$ be a dominant coweight and $t=q^{-1}$ a variable. Then the Hall-Littlewood polynomial $P_{\lambda}(t)$ is
$$ P_{\lambda}(t)=\sum_{\mu \in \mathbb{Z}^{n+1}}\sum_{T \in SSYT(\lambda,\mu)}  \frac{\varphi_T(t)}{b_{\lambda}(t)}x^{\mu},$$
where

$$ b_{\lambda}(t)= \prod_{i \geq 1}  \varphi_{m_i(\lambda)}(t),$$
where
$$ \varphi_k(t)= (1-t)(1-t^2)\cdots(1-t^k) \; \text{ for } \; k\in \mathbb{N}$$
and 
$$ \varphi_{T}(t)= \prod_{i=1}^{n+1} \varphi_{\lambda^{(i)}/\lambda^{(i-1)}}(t),$$
where
$$ \varphi_{\lambda/\mu}(t)= \prod_{j \in I}(1-t^{m_j(\mu)}),$$
where
$$ I= \{j \geq 1 \; \; \vert \; \; (\lambda-\mu)_j'  > (\lambda-\mu)_{j+1}'\}.$$
\end{satz} 
\begin{rem}
Note that $b_\lambda(t)$ divides $\varphi_T(t)$. A detailed proof can be found in Macdonald's book \cite{Mac1}.
\end{rem}
It turns out that the Macdonald and the combinatorial Gaussent-Littelmann formula for type $A_n$ are the same. The most noticeable property of the combinatorial Gaussent-Littelmann formula is that we calculate the contribution of a Young tableau columnwise. More precisely this means, that we obtain the contribution of a Young tableau $T$ as a product of contributions of every column of $T$ where as these contributions only depend on the column itself and, if existing, on the column to the right. 
In order to prove that the two formulas are the same the first step is to express the Macdonald formula in such a way that it becomes clear how to calculate it columnwise. Therefore we have to avoid the usage of the associated $\lambda$-chain in the new presentation. Later we show that the formulas are the same by showing that the contribution of each column is the same.\\
Before proceeding further we need to establish some more notation:\\
Let $\lambda=\lambda_1\epsilon_1+\dots+\lambda_n\epsilon_n$ be a dominant coweight and let $r$ be the smallest index such that $\lambda_{r+1}=0$ and let $T=(C_0,\dots,C_k)$ be a Young tableau of shape $\lambda$ with columns $C_i$ with $i \in \{0,\dots,k\}$. Consider the diagram dg($\sum_{i=1}^r (\lambda_i+1)\epsilon_i$). Note that this diagram contains $T$. We now define the augmented tableau $\hat{T}$ to be the diagram dg($\sum_{i=1}^r (\lambda_i+1)\epsilon_i$) where every box that is contained in $T$ and dg($\sum_{i=1}^r (\lambda_i+1)\epsilon_i$) has the same filling as in $T$ and all other boxes are filled with $\infty$. We extend the order on $\mathbb{N}$ to an order on the set $\mathbb{N}\cup\{\infty\}$ by defining $i < \infty$ for every $i \in \mathbb{N}$.\\
Let $u$ be a box of $\hat{T}$. Then $c(u)$ denotes the entry of the box $u$ in $\hat{T}$. Now let $u$ be a box in the $i$th column and the $j$th row of $T$. Then $h(u)$ denotes the head of $u$ which is the set of all boxes $v$ in the $(i+1)$th column and $k$th row for $k \leq j$ in $\hat{T}$ such that $c(u)\leq c(v)$.\\
\textbf{Example}\\
For ${\tiny T= \Yvcentermath1 \young(133,245,36,5)}$ the augmented Young tableau is $$\hat{T}= \Yvcentermath1 \young(133\infty,245\infty,36\infty,5\infty)\; .$$
Let $u$ be the box of $T$ in the 4th row with entry 5. The head of $u$ consists of the box in the 4th row of $\hat{T}$ with entry $\infty$ and the box in the third row of $\hat{T}$ with entry 6.\\ \\
Using this new definitions we can reformulate the original Macdonald formula: 
\begin{satz}{2. version of the Macdonald formula}\\
Let $T$ be a semistandard Young tableau of shape $\lambda$ and content $\mu$. Then
$$ \varphi_T(t)=  \underset{\nexists v \in h(u): c(u)=c(v)}{\prod_{\text{ box } u \in T}}(1-t^{\vert h(u)\vert}).$$
\end{satz}
\begin{proof}
We want to write $\varphi_T(t)$ as a product over all boxes in $T$:
$$ \varphi_T(t)= \prod_{u \in T}\text{contribution}(u).$$
For this purpose we assign every factor $(1-t^{m_j(\lambda^{(i)})})$ of $\varphi_T(t)$ to exactly that box in $T$ that is in the $j$th column and has the entry $i$. These are exactly the boxes in $T$ whose entries do not show up in the column to the right and hence not in the head of it.\\
Let $u$ in $T$ be a box such that $c(u)$ does not appear in the column to the right. Then $m_j(\lambda^{(i)})$ is the number of boxes in the head of $u$.   
\end{proof}
Let $T=(C_0,\dots,C_r)$ be a semistandard Young tableau and $C_i$ the $i$th column with $i \in  \{0,\dots,r\}.$\\ Define $$\varphi_{C_i}(t)=\underset{\nexists v \in h(u): c(u)=c(v)}{\prod_{u \in C_i}}(1-t^{\vert h(u)\vert}) .$$ Then 
$$ \varphi_{T}(t)= \prod_{i \in \{0,\dots,r\}}\varphi_{C_i}(t).$$
Consider $b_{\lambda}(t)= \prod_{i \in \{0,\dots,r\}}\varphi_{m_{i+1}(\lambda)}(t)$. We can assign the factor $\varphi_{m_{i+1}(\lambda)}(t)$ to the $i$th column. With this last observation it follows that also the Macdonald formula can be calculated columnwise:
$$ \frac{\varphi_{T}(t)}{b_{\lambda}(t)}=\frac{\prod_{i \in \{0,\dots,r\}}\varphi_{C_i}(t)}{\prod_{i \in \{0,\dots,r\}}\varphi_{m_{i+1}(\lambda)}(t)}=\prod_{i \in \{0,\dots,r\}}\frac{\varphi_{C_i}(t)}{\varphi_{m_{i+1}(\lambda)}(t)}.$$
\section{Comparison of the formulas}
Let $\lambda$ and $\mu$ be dominant coweights. Recall that the coefficient $L_{\lambda,\mu}(q)$ is defined as
$$ P_{\lambda}(x,q)= \sum_{\mu \in X\check{}_+}q^{-\left\langle \lambda+\mu, \rho\right\rangle}L_{\lambda, \mu}(q)m_{\mu}(x)$$
and that the combinatorial Gaussent-Littelmann formula calculates $L_{\lambda, \mu}(q)$ as a sum over all semistandard Young tableau of shape $\lambda$ and content $\mu$:
$$ L_{\lambda, \mu}(q)=\sum_{T \in SSYT(\lambda, \mu)} c(T).$$
On the other side the Macdonald formula calculates $P_{\lambda}(t)$ as follows:
$$ P_{\lambda}(x,t)=\sum_{\mu \in X_+\check{}} \sum_{T \in SSYT(\lambda,\mu)}  \frac{\varphi_T(t)}{b_{\lambda}(t)}m_{\mu}(x).$$
Thus we derive 
\begin{equation}\label{satz1}
\sum_{T \in SSYT(\lambda,\mu)} q^{-\left\langle \lambda+\mu, \rho\right\rangle} c(T)= \sum_{T \in SSYT(\lambda,\mu)}\frac{\varphi_T(t)}{b_{\lambda}(t)}.
\end{equation}
As we have seen in the sections before $\text{c}(T)$, $\varphi_T(t)$ and $b_{\lambda}(t)$ can be calculated columnwise. Consider $q^{-\left\langle \lambda+\mu, \rho\right\rangle}$. Let $T$ be in $SSYT(\lambda, \mu)$ and $C_i$ the $i$th column of $T$ and $i \in \{0, \dots, r\}$. Every column $C_i$ on its own is again a semistandard Young tableau with shape $\lambda_{(i)}$ and content $\mu_{(i)}$ and $\lambda= \sum_{i}\lambda_{(i)}$ and $\mu = \sum_{i} \mu_{(i)} $ holds. Consequently we can also calculate $q^{-\left\langle \lambda+\mu, \rho\right\rangle}$ columnwise. We show that
\begin{align}\label{main}
c(T)= t^{-\left\langle \lambda+\mu, \rho\right\rangle}\frac{\varphi_T(t)}{b_{\lambda}(t)}
\end{align}
holds for every $T \in SSYT(\lambda, \mu)$ and that especially the contribution of every column of the Young tableau $T$ is the same on both sides:
\begin{satz}\label{maint}
Let $T=(C_0, \dots, C_r)$ be a semistandard Young tableau of type $A_n$ with shape $\lambda$ and content $\mu$ and let $C_i$ be the $i$th column of $T$ for $i \in \{0,\dots,r\}.$ Then the following holds:
$$ c(C_i,C_{i+1})= t^{-\left\langle \lambda_{(i)}+\mu_{(i)}, \rho\right\rangle}\frac{\varphi_{C_i}(t)}{\varphi_{m_{i+1}(\lambda)}(t)}$$
for $i \in \{0, \dots, r-1\}$ and
$$ c(C_r)= t^{-\left\langle \lambda_{(r)}+\mu_{(r)}, \rho\right\rangle}\frac{\varphi_{C_r}(t)}{\varphi_{m_{r+1}(\lambda)}(t)}.$$
\end{satz}
We know from \cite{Mac1} that $\varphi_{m_{i+1}(\lambda)}(t)$ divides $\varphi_{C_i}(t)$. Additionally we know that $\left\langle \lambda_{(i)}+\mu_{(i)}, \rho\right\rangle$ is bigger then the highest exponent in $\frac{\varphi_{C_i}(t)}{\varphi_{m_{i+1}(\lambda)}(t)}$. Consequently, we obtain:
\begin{align*}
&t^{-\left\langle \lambda_{(i)}+\mu_{(i)}, \rho\right\rangle}\frac{\varphi_{C_i}(t)}{\varphi_{m_{i+1}(\lambda)}(t)}\\
&=
 t^{-a}(1-t)^{a_1}(1-t^2)^{a_2}\dots\\
&=t^{-a+\sum_iia_i}(t^{-1}-1)^{a_1}(t^{-2}-1)^{a_2}\dots \\
&=q^{a-\sum_iia_i}(q-1)^{a_1}(q^2-1)^{a_2}\dots \\
&=q^{a-\sum_iia_i}(q-1)^{a_1}(1+q)^{a_2}(q-1)^{a_2}(1+q+q^2)^{a_3}(q-1)^{a_3}\dots\\
&=q^{a-\sum_iia_i}(q-1)^{\sum_ia_i}(1+q)^{a_2}(1+q+q^2)^{a_3}\dots
\end{align*}
for $i \in \{1,\dots ,r\}$.\\
Hence for every column $C_i$ of a semistandard Young tableau there exist numbers $b,b_1,\dots \in \mathbb{N}$ such that 
$$ t^{-\left\langle \lambda_{(i)}+\mu_{(i)}, \rho\right\rangle}\frac{\varphi_{C_i}(t)}{\varphi_{m_{i+1}(\lambda)}(t)}= q^{b}(q-1)^{\sum_ib_i}(1+q)^{b_1}(1+q+q^2)^{b_2}\dots .$$
To simplify the notation define $$M(C_i,C_{i+1}) := t^{-\left\langle \lambda_{(i)}+\mu_{(i)}, \rho\right\rangle}\frac{\varphi_{C_i}(t)}{\varphi_{m_{i+1}(\lambda)}(t)}$$
for $i \neq r$ and $$M(C_r) := t^{-\left\langle \lambda_{(r)}+\mu_{(r)}, \rho\right\rangle}\frac{\varphi_{C_r}(t)}{\varphi_{m_{r+1}(\lambda)}(t)}.$$
Like in the Gaussent-Littelmann formula this notation shall underline that the contribution of the columns of the tableau to the formula depends on the column itself and, if existing, on the column to the right.\\
\begin{rem}
\begin{itemize}
\item[(1)] In the Gaussent-Littelmann formula it is not at all obvious that the contribution $$ c((E_{j-1} \supset V_j \subset E_{j}))=\sum_{\textbf{c} \in \Gamma^+_{S_{V_j}^j}(\textbf{i}_j, op)}q^{t(\textbf{c})}(q-1)^{r(\textbf{c})}$$ is always a product $$q^{b}(q-1)^{\sum_ib_i}(1+q)^{b_1}(1+q+q^2)^{b_2}\dots $$ for some $b,b_1,\dots \in \mathbb{N}$.
\item[(2)] As already mentioned the Gaussent-Littelmann formula is developed by describing a certain intersection in the affine Grassmannian and counting the points over a finite field $\mathbb{F}_q$. In fact, the polynomial $c(C_i,C_{i+1})$ counts the points over $\mathbb{F}_q$ of a subvariety $Min(C_i,C_{i+1})$ of a generalized Grassmannian $H/R$, where $H$ and $R$ are determined by the positively folded combinatorial one-skeleton gallery $\delta_T$. For more details regarding the geometry behind the Gaussent-Littelmann formula see \cite{GL1}. Now, our Theorem \ref{maint} suggests that $Min(C_i,C_{i+1})$ is isomorphic to a product of $b_j$ times $\mathbb{P}^j$ for $j \in \mathbb{N}$, $b$ times $\mathbb{C}$ and $\sum_{j}b_j$ times $\mathbb{C^*}$.
\end{itemize}  
\end{rem}
\subsection{Proof of Equality}
This section is devoted to proving the equality of the two formulas by showing that the contribution of every column of the semistandard Young tableau $T\in SSYT(\lambda, \mu)$ is the same to both formulas.
\begin{proof}
First we consider the last column $C_r$ of $T$ with content $\mu_{(r)}$ and shape $\lambda_{(r)}$. Because $C_r$ is the only semistandard Young tableau with content $\mu_{(r)}$ and shape $\lambda_{(r)}$ the coefficent $L_{\mu_{(r)},\lambda_{(r)}}$ only consists of one summand coming from the column $C_r$ itself. With (\ref{satz1}) the claim follows.\\
The proof for the other columns is more intricate: Let $C_i$ be a column of $T$ that has a column to the right i.e. $i \neq r$. As mentioned before the contribution depends on $C_{i+1}$ on both sides of the equation. Thus our study in this part of the proof concerns semistandard 2-column Young tableaux: Let $T_i$ denote the semistandard 2-column Young tableau where $C_i$ is the first and $C_{i+1}$ the second column. Let $s_{i_1}, \dots, s_{i_l}$ be a sequence of reflections such that $s_{i_j}$ increases the column $s_{i_{j-1}}\dots s_{i_1}(C_i)$ for every $j \in \{1, \dots, l\}$ and there exists no simple reflection that increases $s_{i_l}\dots s_{i_1}(C_i)$. We prove the claim by induction over $l$.\\
\textbf{Basis:  $l = 0$} \\
Let us start with the left hand side, the combinatorial Gaussent-Littelmann formula: For $l=0$ the contribution $c(C_i,C_{i+1})$ is 1 because we only have one single simple path consisting of one vertex.\\
Now consider the right hand side, the Macdonald formula: Suppose $C_i$ consists of $m$ boxes. For $l=0$ the column $C_i$ has the entries $n-m+2, \dots, n+1$. 
 The shape of $C_i$ is $\lambda_{(i)}=\epsilon_1+\epsilon_2+\dots +\epsilon_m$ and the content is $\mu_{(i)}= \epsilon_{n-m+2}+ \dots +\epsilon_{n+1}$ and we derive 
\begin{align*}
\left\langle \lambda_{(i)}+ \mu_{(i)}, \rho \right\rangle &= 1/2(\sum_{j=0}^{m-1}n-1-2j)+1/2(\sum_{k=n+1-m}^{n+1}n-1-2k)\\
&=1/2(\sum_{j=0}^{m-1}n-1-2j)-1/2(\sum_{k=0}^{m-1}n-1-2k)\\
&=0.
\end{align*}
Thus $t^{-\left\langle \lambda_{(i)}+ \mu_{(i)}, \rho \right\rangle}=1$. It remains to show that $\frac{\varphi_{C_i}(t)}{\varphi_{m_{i+1}(\lambda)}(t)}$ equals 1:\\ 
Consider $\varphi_{C_i}(t)=\underset{\nexists v \in h(u): c(u)=c(v)}{\prod_{u \in C_i}}(1-t^{\vert h(u)\vert})$. The entries in $C_{i+1}$ have to be bigger than or equal to the entries to the left in $C_i$ because the Young tableau $T$ is semistandard. But the entries in $C_i$ are already as big as possible, thus the entries of $C_{i+1}$ are also entries in $C_i$. Consequently there are exactly $m_{i+1}(\lambda)$ boxes in $C_i$ that make a contribution to the formula. Let $b_{j}$ be the $j$th box (from top to bottom) in $C_i$ whose entry $e_j$ is not an entry in $C_{i+1}$. The contribution of this box to the formula is $(1-t^{\vert h(b_j)\vert})$ with 
$$\vert h(b_j)\vert= \vert \{ \text{ boxes with } \infty \text{ in } h(b_j)\} \vert +\vert\{\text{boxes with entry in $\mathbb{N}$ in } h(b_j) \}\vert.$$
We need to check two cases:\\
\textbf{1. case: $\exists v \in h(b_j) \text{ with entry } \infty$:}\\
Then we have
$$ \vert \{ \text{ boxes with } \infty \text{ in } h(b_j)\} \vert =\\
m_{i+1}(\lambda) -\vert\{\text{boxes in } C_i \text{ at the top of } b_j\}\vert $$
and
\begin{align*}  
\vert \{\text{boxes with entry in } & \mathbb{N} \text{ in } h(b_j) \} \vert \\
= &\vert\{\text{boxes in } C_i \text{ at the top of } b_j\}\vert - (m_{i+1}(\lambda)-j).
\end{align*}
Using the last two equations we obtain $\vert h(b_j)\vert= j$.\\
\textbf{2. case: $\nexists v \in h(b_j) \text{ with entry } \infty$:}\\
Then we clearly have $\vert h(b_j)\vert= j$.\\
Thus, for all $j \in \{1, \dots, m_{i+1}\}$ we have $\vert h(b_j)\vert= j$ and consequently 
$$ \varphi_{C_i}(t)= \prod_{k=1 }^{m_{i+1}(\lambda)}(1-t^{k})=\varphi_{m_{i+1}(\lambda)}(t).$$
It follows that $M(C_i,C_{i+1})$ is 1.\\
\textbf{Induction step:} $l \mapsto l+1$\\
Let $T_i$ be a semistandard 2-column Young tableau where $C_i$ is the first and $C_{i+1}$ is the second column. Because $l \neq 0$ there exists a simple reflection $s_j$ that increases the column $C_i$. 
There a three cases to check:\\
\textbf{1. case: $s_j(C_{i+1})=C_{i+1}$}\\
By induction hypothesis we know that $c(s_j(C_i),s_j(C_{i+1}))$ equals \\$M(s_j(C_i),s_j(C_{i+1}))$ say 
$$q^a(q-1)^b(1+q)^{a_1}\dots(1+q+\dots+q^{k})^{a_k}\dots .$$
The combinatorial version of the Gaussent-Littelmann formula now tells us that we obtain $c(C_i,C_{i+1})$ simply by multiplying $c(s_j(C_i),s_j(C_{i+1}))$  by $q$ hence we have:
$$ c(C_i,C_{i+1})=q^{a+1}(q-1)^b(1+q)^{a_1}\dots(1+q+\dots+q^{k})^{a_k}\dots.$$
Next, we need to calculate the contribution $M(C_i,C_{i+1})$ to the Macdonald formula from  $M(s_j(C_i),s_j(C_{i+1}))$:  
Because $s_j$  increases the column $C_i$ we know that there is a $j$ but no $j+1$ in $C_i$  and $s_j(C_{i+1})=C_{i+1}$ means that either $j$ and $j+1$ or neither $j$ nor $j+1$ are in $C_{i+1}$.\\
Let us first consider the case when $j$ and $j+1$ are entries in $C_{i+1}$. We know that $M(C_i,C_{i+1})$ is a product over all boxes in $C_i$ where the contribution of a box depends on the entries in $C_{i+1}$. Apply $s_j$ to $s_j(T_i)$ only changes a single box in the first column, thus we only need to exchange the contribution of this box, which is the box with entry $j+1$, and the contribution of the box with entry $j$ to $M(C_i,C_{i+1})$. But because $j$ and $j+1$ are in the second column of $T_i$ and of $s_j(T_i)$ the contribution of the box is in both cases 1, thus we derive that only the content of the tableau changes. We get that we have to multiply the contribution $M(s_j(C_i),s_j(C_{i+1}))$ by $t^{-1}=q$ to obtain the contribution  $M(C_i,C_{i+1})$. This is what we did in the combinatorial Gaussent-Littelmann formula.\\
Now consider the case where neither $j$ nor $j+1$ are entries in $C_{i+1}$. We again need to exchange the contribution of the box with entry $j+1$ to $M(s_j(C_i),s_j(C_{i+1}))$ with the contribution of the box with entry $j$ to \\ $M(C_i,C_{i+1})$. But the contributions of the boxes are the same because the heads are the same. Again only the content of the tableau changes and we obtain the same result as above.\\
\textbf{2. case: $s_j$ increases $C_{i+1}$}\\
We again know by hypothesis that the contributions of $(s_j(C_i),s_j(C_{i+1}))$ is the same on both sides, say
$$ q^{a}(q-1)^b(1+q)^{a_1}\dots(1+q+\dots+q^{k})^{a_k}\dots .$$
As in the first case it follows that 
$$ q^{a+1}(q-1)^b(1+q)^{a_1}\dots(1+q+\dots+q^{k})^{a_k}\dots $$
is $c(s_j(C_i),s_j(C_{i+1}))$. Now we consider the Macdonald formula: In both columns of the Young tableau $T_i$ there is a $j$ but no $j+1$. Consequently in both columns of the Young tableau $s_j(T_i)$ there is a $j+1$ but no $j$. The box in $C_i$ with entry $j$ has contribution 1 to $M(C_i,C_{i+1})$ because there is a $j$ in $C_{i+1}$. By the same argument we derive that the contribution of the box with entry $j+1$ to $M(s_j(C_i),s_j(C_{i+1}))$ is 1. As in the first case only the content of the tableau changes and we have to multiply the contribution $M(s_j(C_i),s_j(C_{i+1}))$ by $t^{-1}=q$ to obtain the contribution $M(C_i,C_{i+1})$. And this is what we did in the combinatorial Gaussent-Littelmann formula.\\
\textbf{3. case: $s_j$ decreases $C_{i+1}$}\\
Suppose that $s_j(T_i)$ is semistandard. We know by induction hypothesis that $M(s_j(C_i),s_j(C_{i+1}))=c(s_j(C_i),s_j(C_{i+1}))$ and $M(s_j(C_i),C_{i+1})=c(s_j(C_i),C_{i+1})$. Let $$ M(s_j(C_i),s_j(C_{i+1}))=q^a(q-1)^b(1+q)^{a_1}\dots(1+q+\dots+q^{k})^{a_k}\dots .$$ 
We can express $M(s_j(C_i),C_{i+1})$ depending on $M(s_j(C_i),s_j(C_{i+1}))$:
\begin{align*}
 M(s_j(C_i),C_{i+1})&= \frac{1}{(1-t^{l+1})}M(s_j(C_i),s_j(C_{i+1}))\\
 &= \frac{t^{-(l+1)}}{(t^{-(l+1)}-1)}M(s_j(C_i),s_j(C_{i+1}))\\
 &= \frac{q^{l+1}}{(q-1)(1+q+\dots q^{l})}M(s_j(C_i),s_j(C_{i+1}))\\
 &=q^{a+l+1}(q-1)^{b-1}(1+q)^{a_1}\dots(1+q+\dots+q^{l})^{a_l-1}\dots
\end{align*}
where $l+1=h(u)$ and $u$ is the box in $s_j(C_{i})$ with entry $j+1$ because $u$ has contribution $(1-t^{l+1})$ to $M(s_j(C_i),s_j(C_{i+1}))$ and contribution $1$ to $M(s_j(C_i),C_{i+1})$ and $u$ is the only box with different contributions.\\  
In the combinatorial Gaussent-Littelmann formula we obtain $c(C_i,C_{i+1})$ by multiplying $c(s_j(C_i),s_j(C_{i+1}))$ by 1 and add $c(s_j(C_i),C_{i+1})$ multiplied by $(q-1)$:
\begin{align*}
 & q^a(q-1)^b(1+q)^{a_1}\dots(1+q+\dots+q^{k})^{a_k}\dots\\
& \; \; \;\; \; \; \; + q^{a+l+1}(q-1)^{b}(1+q)^{a_1}\dots(1+q+\dots+q^{l})^{a_l-1}\dots\\
= \; \; \; &((1+\dots+ q^{l})+q^{l+1})(q^a(q-1)^b(1+q)^{a_1}\dots(1+q+\dots+q^{l})^{a_l-1})\dots\\
= \; \; \; &q^a(q-1)^b(1+q)^{a_1}\dots(1+ \dots+q^l)^{a_l-1}(1+\dots q^{l+1})^{a_{l+1}+1}\dots .\\ 
\end{align*}
We now need to compute $M(C_i,C_{i+1})$ from $M(s_j(C_i),s_j(C_{i+1}))$: We know that there is a $j$ but no $j+1$ in $C_i$ and that we have no box with entry $j$ but one with entry $j+1$ in $C_{i+1}$. Because $s_j(T_i)$ and $id_j(T_i)$ are both semistandard we know that the box with entry $j$ in $C_i$ cannot be next to the box in $C_{i+1}$ with entry $j+1$. Consider the box $u$ in $s_j(C_i)$ with the entry $j+1$. Then the box in $s_j(C_{i+1})$ with entry $j$ is not in the head of $u$. After applying $s_j$ to $s_j(T_i)$ the box $u$ in $C_i$ has entry $j$ and now the box with entry $j+1$ in $C_{i+1}$ is in the head of $u$. Further, the content of the tableau also changes but again all other contributions of boxes stay the same. Consequently we obtain $M(C_i,C_{i+1})$ as
\begin{align*}
M(C_i,C_{i+1})&= \frac{(1-t^{l+2})}{(1-t^{l+1})}t^{-1}M(s_j(C_i),s_j(C_{i+1}))\\
 &= \frac{(1+\dots+q^{l+1})}{(1+\dots+q^l)}M(s_j(C_i),s_j(C_{i+1}))
\end{align*} 
and this is exactly what we did in the combinatorial Gaussent-Littelmann formula.\\
Now suppose that $s_j(T_i)$ is not semistandard. By induction hypothesis we know that $M(s_j(C_i),C_{i+1})=c(s_j(C_i),C_{i+1}).$ In the combinatorial Gaussent-Littelmann formula we obtain $c(C_i,C_{i+1})$ by multiplying $c(s_j(C_i),C_{i+1})$ by $(q-1)$. Now we need to consider what happens in the Macdonald formula: The column $C_i$ contains a $j$ but no $j+1$ and $C_{i+1}$ contains a $j+1$ but no $j$. Because $s_j(T_i)$  is not semistandard we know that the box that contains $j$ in $C_{i}$ is next to the box in $C_{i+1}$ that contains $j+1$. In $M(s_j(C_i),C_{i+1})$ the box in $s_j(C_i)$ with entry $j+1$ has contribution 1 but in $M(C_i,C_{i+1})$ this box has entry $j$ and because there is no $j$ in  $C_{i+1}$ the contribution is different from 1. The head of the box in $C_i$ with entry $j$ in $T_i$ only consists of the box in $C_{i+1}$ with entry $j+1$.  The content of the tableau also changes but all contributions of the other boxes to the formula stay the same and we derive that we have to multiply $M(s_j(C_i),C_{i+1})$ by $(1-t)t^{-1}=(q-1)$ to get $M(C_i,C_{i+1})$. This is again what we did in the combinatorial Gaussent-Littelmann formula.
\end{proof}
\section{Examples}
\subsection{Type $A_2$}\text{} \\
For the semistandard Young tableau ${\tiny T=(C_0,C_1)= \Yvcentermath1 \young(13,2)}$ the corresponding combinatorial positively folded one-skeleton gallery in the standard apartment of the affine buidling is $\delta_T=\delta_{s_2s_1(E_{\omega_1})}*\delta_{E_{\omega_2}}=(\mathfrak{o}\subset E_0\supset V_1 \subset E_1 \supset \mathfrak{o}):$
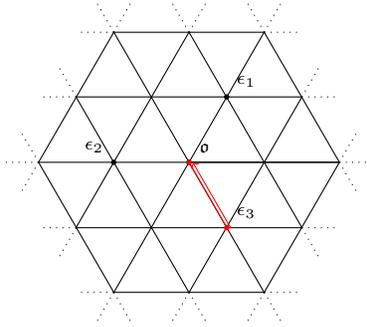
	\begin{center}
	
   	\begin{tikzpicture}
   
   \draw (0,0)-- (0:2cm);
   \draw (0,0)-- (60:2cm); 
   \draw (0,0)-- (120:2cm);
   \draw (0,0)-- (180:2cm);    
   \draw (0,0)-- (240:2cm);  
   \draw (0,0)-- (300:2cm);
   \draw (0,0)-- (360:2cm); 
   \draw (-1cm,0)-- ++(60:2cm);
   \draw (-1cm,0)-- ++(240:1cm);
   \draw (-1cm,0)-- ++(120:1cm);
   \draw (-1cm,0)-- ++(300:2cm);
   \draw (1cm,0)-- ++(60:1cm);
   \draw (1cm,0)-- ++(120:2cm);
   \draw (1cm,0)-- ++(240:2cm);
   \draw (1cm,0)-- ++(300:1cm);
   \draw (-1.5cm,0.87cm)--(1.5cm,0.87cm);
   \draw (-1.5cm,-0.87cm)--(1.5cm,-0.87cm);
   \coordinate (o) at (0,0) node[above right] {\tiny{$\mathfrak{o}$}} ;
   \fill(o) circle (1pt);
   \coordinate (e1) at (0.5cm,0.87cm);
   \fill(e1) circle (1pt);
   \node[above right] at (e1) {\tiny{$\epsilon_1$}};
   \coordinate (e2) at (-1cm,0cm);
   \fill(e2) circle (1pt);
   \node[above left] at (e2) {\tiny{$\epsilon_2$}};
   \coordinate (e3) at (0.5cm,-0.87cm);
   \node[above right] at (e3) {\tiny{$\epsilon_3$}};
   \draw(-1cm,1.73cm)--(1cm,1.73cm);
   \draw(-1cm,1.73cm)--(-2cm,0)--(-1cm,-1.73cm)--(1cm,-1.73cm)--(2cm,0)--(1cm,1.73cm);
   \draw[dotted] (-2cm,0)-- ++(120:0.5cm);
   \draw[dotted] (-2cm,0)-- ++(180:0.5cm);
   \draw[dotted] (-2cm,0)-- ++(240:0.5cm);
   \draw[dotted] (-1.5cm,0.87cm)-- ++(120:0.5cm);
   \draw[dotted] (-1.5cm,0.87cm)-- ++(180:0.5cm);
   \draw[dotted] (-1cm,1.73cm)-- ++(60:0.5cm);
   \draw[dotted] (-1cm,1.73cm)-- ++(120:0.5cm);
   \draw[dotted] (-1cm,1.73cm)-- ++(180:0.5cm);
   \draw[dotted] (0,1.73cm)-- ++(60:0.5cm);
   \draw[dotted] (0,1.73cm)-- ++(120:0.5cm);
   \draw[dotted] (1cm,1.73cm)-- ++(0:0.5cm);
   \draw[dotted] (1cm,1.73cm)-- ++(60:0.5cm);
   \draw[dotted] (1cm,1.73cm)-- ++(120:0.5cm);
   \draw[dotted] (1.5cm,0.87cm)-- ++(0:0.5cm);
   \draw[dotted] (1.5cm,0.87cm)-- ++(60:0.5cm);
   \draw[dotted] (2cm,0)-- ++(0:0.5cm);
   \draw[dotted] (2cm,0)-- ++(60:0.5cm);
   \draw[dotted] (2cm,0)-- ++(300:0.5cm);
   \draw[dotted] (1.5cm,-0.87cm)-- ++(0:0.5cm);
   \draw[dotted] (1.5cm,-0.87cm)-- ++(300:0.5cm);
   \draw[dotted] (1cm,-1.73cm)-- ++(0:0.5cm);
   \draw[dotted] (1cm,-1.73cm)-- ++(300:0.5cm);
   \draw[dotted] (1cm,-1.73cm)-- ++(240:0.5cm);
   \draw[dotted] (0,-1.73cm)-- ++(300:0.5cm);
   \draw[dotted] (0,-1.73cm)-- ++(240:0.5cm);
   \draw[dotted] (-1cm,-1.73cm)-- ++(300:0.5cm);
   \draw[dotted] (-1cm,-1.73cm)-- ++(180:0.5cm);
   \draw[dotted] (-1cm,-1.73cm)-- ++(240:0.5cm);
   \draw[dotted] (-1.5cm,-0.87cm)-- ++(240:0.5cm);
   \draw[dotted] (-1.5cm,-0.87cm)-- ++(180:0.5cm);
   
   \draw[->,red] (0,0)-- ++(300:1cm)--++(0:0.05cm)--++(120:1cm);
   \fill[red](o) circle (1pt);
   \fill[red](e3) circle (1pt);

     \end{tikzpicture}  

	\end{center}
We now calculate the contribution of $T$ (resp. $\delta_{T}$) to the formulas:\\	 \\
\textbf{Gaussent-Littelmann formula}\\
Consider the gallery $(\mathfrak{o}\subset (E_0)_{\mathfrak{o}})$ associated to the last column of $T$ at the origin $\mathfrak{o}$, namely ${\tiny  \Yvcentermath1 \young(3)}$, in the standard apartment of the residue building at $\mathfrak{o}$:
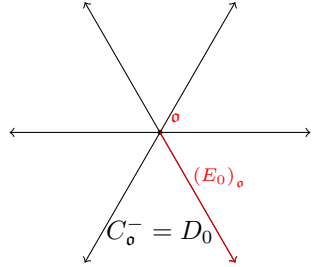
\begin{center}
	
   	\begin{tikzpicture}[x=0.4cm,y=0.4cm]
   	
   	\coordinate(o) at (0,0) node[red, above right] {\tiny{$\mathfrak{o}$}};
   	\fill[red](o) circle (1pt);
   	\draw[->] (0,0)-- ++(60:2cm);
   	\draw[->] (0,0)-- ++(120:2cm);
   	\draw[->] (0,0)-- ++(180:2cm);
    \draw[->] (0,0)-- ++(240:2cm);
   	\draw[->] (0,0)-- ++(300:2cm);
   	\draw[->] (0,0)-- ++(0:2cm);
   	\draw[red,->](0,0)-- ++(300:2cm);
   	\draw[red] (0.3cm,-0.6cm) node[ right] {\tiny{${(E_0)}_{\mathfrak{o}}$}}; 
   	\draw (0,-1cm) node[below] {\small{$C_{\mathfrak{o}}^-=D_0$}};
   	  	
   	\end{tikzpicture}
   	
  \end{center} 	 	
Thus, $w_{D_0}=id$ and $$c(E_0)=q^0=1.$$
Now consider the gallery $((E_0)_{V_1} \supset V_1 \subset (E_1)_{V_1})$ associated to the 2-column Young tableau ${\tiny \Yvcentermath1 \young(13,2)}$ at $V_1=\epsilon_3$, in the standard apartment of the residue building at $V_1$:   
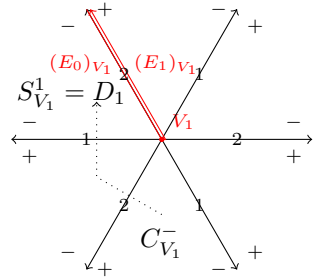
\begin{center}
	
   	\begin{tikzpicture}[x=0.4cm,y=0.4cm]
   	
   	\coordinate(o) at (0,0) node[red, above right] {\tiny{$V_1$}};
   	\draw[->] (0,0)-- ++(60:2cm);
   	\draw[->] (0,0)-- ++(120:2cm);
   	\draw[->] (0,0)-- ++(180:2cm);
    \draw[->] (0,0)-- ++(240:2cm);
   	\draw[->] (0,0)-- ++(300:2cm);
   	\draw[->] (0,0)-- ++(0:2cm);
   	\draw[red,->] (-1cm,1.73cm)-- ++(300:2cm)-- ++(0:0.05cm)-- ++(120:2cm);
   	\fill[red](o) circle (1pt);
   	\draw (0,-1cm) node[below] {\small{$C_{V_1}^-$}};
   	\draw (-1.2cm,0.6cm) node {\small{$S^1_{V_1}=D_1$}};
   	\draw (-0.5cm,1cm) node[left,red] {\tiny{$(E_0)_{V_1}$}};
   	\draw (-0.5cm,1cm) node[right,red] {\tiny{$(E_1)_{V_1}$}};
   	\draw (-2cm,0) node[above right] {\tiny{$-$}};
   	\draw (-2cm,0) node[below right] {\tiny{$+$}};
   	\draw (-1cm,1.73cm) node[below left] {\tiny{$-$}};
   	\draw (-1cm,1.73cm) node[right] {\tiny{$+$}};
   	\draw (1cm,1.73cm) node[left] {\tiny{$-$}};
   	\draw (1cm,1.73cm) node[below right] {\tiny{$+$}};
   	\draw (2cm,0) node[above left] {\tiny{$-$}};
   	\draw (2cm,0) node[below left] {\tiny{$+$}};
   	\draw (1cm,-1.73cm) node[left] {\tiny{$-$}};
   	\draw (1cm,-1.73cm) node[above right] {\tiny{$+$}};
   	\draw (-1cm,-1.73cm) node[above left] {\tiny{$-$}};
   	\draw (-1cm,-1.73cm) node[right] {\tiny{$+$}};
   	\draw (-1cm,0) node {\tiny{$1$}};
   	\draw (-0.5cm,0.87cm) node {\tiny{$2$}};
   	\draw (0.5cm,0.87cm) node {\tiny{$1$}};
   	\draw (1cm,0) node {\tiny{$2$}};
   	\draw (0.5cm,-0.87cm) node {\tiny{$1$}};
   	\draw (-0.5cm,-0.87cm) node {\tiny{$2$}};
   	\draw[dotted,->] (0,-1cm)-- ++(150:1cm)-- ++(90:1cm);

	\end{tikzpicture}
   	
  \end{center} 
In the picture above we label the walls of the anti-dominant chamber $C_{V_1}^-$ with $i$ if the wall is contained in the hyperplane $H_{(\alpha_i,n_i)}$ for a simple root $\alpha_i \in \phi_{V_1}$. The labeling of the walls of the chambers in $\mathbb{A}_{V_1}$ is $W_{V_1}^a$-equivariant so that the labeling of the walls of $w(C_{V_1}^-)$ for $w \in W_{V_1}^a$ is the image of the labeling of the walls of $C_{V_1}^-$. The labeling of a wall is exactly its type.\\
The signs on the hyperplanes indicate on which side of the hyperplane the chamber $S^1_{V_1}$ is. More precisely, the negative half-space of the hyperplane contains $S^1_{V_1}$.\\
In type $A_2$, $W_V^v$ coincides with $W$ for every vertex $V$ in $\mathbb{A}$. The reduced expression for the Weyl group element $w \in W_{V_1}^a$ that sends $C_{V_1}^-$ to $D_1$ is $s_2s_1$ indicated in the picture above by the dotted line. The galleries of chambers in $\Gamma^+_{S^1_{V_1}}((2,1),op)$ are illustrated below:\\
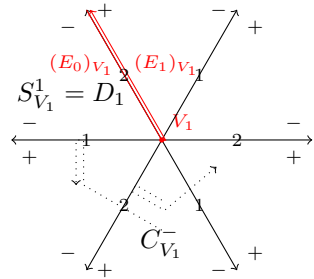
\begin{center}
	
   	\begin{tikzpicture}[x=0.4cm,y=0.4cm]
   	
   	\coordinate(o) at (0,0) node[red, above right] {\tiny{$V_1$}};
   	\draw[->] (0,0)-- ++(60:2cm);
   	\draw[->] (0,0)-- ++(120:2cm);
   	\draw[->] (0,0)-- ++(180:2cm);
    \draw[->] (0,0)-- ++(240:2cm);
   	\draw[->] (0,0)-- ++(300:2cm);
   	\draw[->] (0,0)-- ++(0:2cm);
   	\draw[red,->] (-1cm,1.73cm)-- ++(300:2cm)-- ++(0:0.05cm)-- ++(120:2cm);
   	\fill[red](o) circle (1pt);
   	\draw (0,-1cm) node[below] {\small{$C_{V_1}^-$}};
   	\draw (-1.2cm,0.6cm) node {\small{$S^1_{V_1}=D_1$}};
   	\draw (-0.5cm,1cm) node[left,red] {\tiny{$(E_0)_{V_1}$}};
   	\draw (-0.5cm,1cm) node[right,red] {\tiny{$(E_1)_{V_1}$}};
   	\draw (-2cm,0) node[above right] {\tiny{$-$}};
   	\draw (-2cm,0) node[below right] {\tiny{$+$}};
   	\draw (-1cm,1.73cm) node[below left] {\tiny{$-$}};
   	\draw (-1cm,1.73cm) node[right] {\tiny{$+$}};
   	\draw (1cm,1.73cm) node[left] {\tiny{$-$}};
   	\draw (1cm,1.73cm) node[below right] {\tiny{$+$}};
   	\draw (2cm,0) node[above left] {\tiny{$-$}};
   	\draw (2cm,0) node[below left] {\tiny{$+$}};
   	\draw (1cm,-1.73cm) node[left] {\tiny{$-$}};
   	\draw (1cm,-1.73cm) node[above right] {\tiny{$+$}};
   	\draw (-1cm,-1.73cm) node[above left] {\tiny{$-$}};
   	\draw (-1cm,-1.73cm) node[right] {\tiny{$+$}};
   	\draw (-1cm,0) node {\tiny{$1$}};
   	\draw (-0.5cm,0.87cm) node {\tiny{$2$}};
   	\draw (0.5cm,0.87cm) node {\tiny{$1$}};
   	\draw (1cm,0) node {\tiny{$2$}};
   	\draw (0.5cm,-0.87cm) node {\tiny{$1$}};
   	\draw (-0.5cm,-0.87cm) node {\tiny{$2$}};
   	\draw[dotted,->] (0,-0.8cm)-- ++(150:0.4cm)-- ++(240:0.1cm)-- ++(150:-0.5cm)-- ++(40:0.9cm);
   	\draw[dotted,->] (0,-1.2cm)-- ++(150:1.2cm)-- ++(90:0.6cm)-- ++(0:-0.1cm)-- ++(270:0.6cm);
   	   	
	\end{tikzpicture}
   	
\end{center} 

Let $\textbf{c}_1$ denote the upper and  $\textbf{c}_2$ the lower gallery of chambers. The gallery $\textbf{c}_1$ has one positive folding and one positive wall-crossing, the gallery $\textbf{c}_2$ has one positive folding and one negative wall-crossing. We obtain
$$ c(((E_0)_{V_1} \supset V_1 \subset (E_1)_{V_1}))=q(q-1)+(q-1)=(q+1)(q-1)=q^2-1.$$
Together with the first calculation we have
$$c(\delta_{T})=c(E_0)*c(((E_0)_{V_1} \supset V_1 \subset (E_1)_{V_1}))=1*(q^2-1)=q^2-1.$$
\textbf{Combinatorial Gaussent-Littelmann formula}\\
First consider the last column of the Young tableau $T$, namely ${\tiny C_1=  \Yvcentermath1 \young(3)}$. There is no simple reflection in $W$ that increases $C_1$. We obtain
$$ c(C_1)=q^0=1.$$
Now consider the 2-column Young tableau ${\tiny  \Yvcentermath1 \young(13,2)}$ in order to calculate $c(C_0,C_1)$. Applying the algorithm yields the following tree:\\
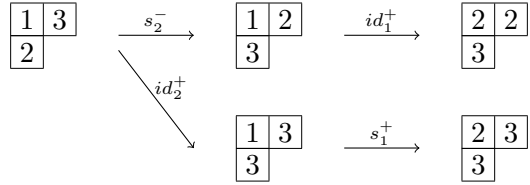
\begin{center}
\begin{tikzpicture}
 \draw (0,0) node  {$ \Yvcentermath1 \young(13,2)$};
 \draw[->] (1cm,0)-- ++(0:1cm);
 \draw (1.5cm,0.2cm) node {\tiny{$s_2^-$}};
 \draw (3cm,0) node {$ \Yvcentermath1 \young(12,3)$};
 \draw[->] (4cm,0)-- ++(0:1cm);
 \draw (4.5cm,0.2cm) node {\tiny{$id^+_1$}};
 \draw (6cm,0) node {$ \Yvcentermath1 \young(22,3)$};
 \draw[->] (1cm,-0.2cm)-- (2cm,-1.5cm);
 \draw (1.7cm,-0.7cm) node {\tiny{$id_2^+$}};
 \draw (3cm,-1.5cm) node {$ \Yvcentermath1 \young(13,3)$};
 \draw[->] (4cm,-1.5cm)-- ++(0:1cm); 
 \draw (6cm,-1.5cm) node {$\Yvcentermath1 \young(23,3)$};
 \draw (4.5cm,-1.3cm) node {\tiny{$s_1^+$}};
 
\end{tikzpicture}.
\end{center}
There are two simple paths in the tree that start in $\tiny{\Yvcentermath1 \young(13,2)}$ and end in a final vertex. Let $\sigma_1$ denote the upper and $\sigma_2$ the lower path. We have
$$ pr(\sigma_1)=0 \; \text{ and } \; pf(\sigma_1)=1$$
and 
$$ pr(\sigma_2)=1 \; \text{ and } \; pf(\sigma_2)=1.$$ 
It follows that 
$$ c(C_0,C_1)=(q-1)+q(q-1)=q^2-1$$
and $$ c(T)=1*(q^2-1)=q^2-1.$$
\textbf{Macdonald formula}\\
The Young tableau $T$ has shape $\lambda =2 \epsilon_1+\epsilon_2$ and content $\mu=0$. In order to calculate $\varphi_T(t)$ we need the augmented Young tableau $\hat{T}=$ {\tiny{$\Yvcentermath1 \young(13\infty,2\infty)$}}.\\
For $C_1= \tiny{\Yvcentermath1 \young(3)}$ we have $\varphi(C_1)(t)=(1-t)$, $\varphi_{m_2(\lambda)}(t)=(1-t)$, $\lambda_{(2)}=\epsilon_1$ and $\mu_{(2)}=\epsilon_3$ such that 
$$ c(C_1)=t^{-\left\langle \epsilon_1+\epsilon_3, \epsilon_1-\epsilon_3 \right\rangle}\frac{(1-t)}{(1-t)}=1.$$
For $C_0 = \tiny{\Yvcentermath1 \young(1,2)}$ we have $\varphi(C_0)(t)=(1-t)(1-t^2)$, $\varphi_{m_1(\lambda)}(t)=(1-t)$, $\lambda_{(1)}=\epsilon_1+\epsilon_2$ and $\mu_{(1)}=\epsilon_1+\epsilon_2$ such that
\begin{align*}
c(C_0,C_1)&=t^{-\left\langle 2\epsilon_1+2\epsilon_2, \epsilon_1-\epsilon_3 \right\rangle} \frac{(1-t)(1-t^2)}{(1-t)}\\
&=t^{-2}(1-t^2)=t^{-2}-1\\
&=q^2-1.
\end{align*}
We have $$ c(T)=1*(q^2-1)=q^2-1.$$
\subsection{Type $B_2$}\text{}\\
For the semistandard Young tableau $T=(C_0,C_1,C_2)=\tiny{\Yvcentermath1 \young(11{{\bar{1}}},{{\bar{2}}})}$ the corresponding positively folded combinatorial one-skeleton gallery in the standard apartment of the affine building of type $B_2$ is $\delta_T=\delta_{s_1s_2s_1(E_{\omega_1})}*\delta_{E_{\omega_1}}*\delta_{s_2(E_{\omega_2})}=(\mathfrak{o}\subset E_0 \supset V_1 \subset E_1 \supset V_2 \subset E_2 \supset V_3 )$:
	\begin{center}
	
   	\begin{tikzpicture}[x=0.4cm,y=0.4cm]
   
   \draw (0,0)-- (0:2cm);   
   \draw (0,0)-- (45:2.83cm); 
   \draw (0,0)-- (90:2cm);
   \draw (0,0)-- (135:2.83cm);
   \draw (0,0)-- (180:2cm);
   \draw (0,0)-- (225:2.83cm);
   \draw (0,0)-- (270:2cm);
   \draw (0,0)-- (315:2.83cm);
   \draw (0,2cm)-- ++(315:2.83cm)-- ++(225:2.83cm)-- ++(135:2.83cm)-- ++(45:2.83cm);
   \draw (-1cm,2cm)-- ++(270:4cm);
   \draw (-2cm,-1cm)-- ++(0:4cm);
   \draw (1cm,2cm)-- ++(270:4cm);
   \draw (-2cm,1cm)-- ++(0:4cm);
   \draw[dotted] (-2cm,2cm)-- ++(135:0.5cm);
   \draw[dotted] (-1cm,2cm)-- ++(90:0.5cm);
   \draw[dotted] (0,2cm)-- ++(135:0.5cm);
   \draw[dotted] (0,2cm)-- ++(45:0.5cm);
   \draw[dotted] (0,2cm)-- ++(90:0.5cm);
   \draw[dotted] (1cm,2cm)-- ++(90:0.5cm);
   \draw[dotted] (2cm,2cm)-- ++(45:0.5cm);
   \draw[dotted] (2cm,1cm)-- ++(0:0.5cm);
   \draw[dotted] (2cm,0)-- ++(45:0.5cm);
   \draw[dotted] (2cm,0)-- ++(0:0.5cm);
   \draw[dotted] (2cm,0)-- ++(315:0.5cm);
   \draw[dotted] (2cm,-1cm)-- ++(0:0.5cm);
   \draw[dotted] (2cm,-2cm)-- ++(315:0.5cm);
   \draw[dotted] (1cm,-2cm)-- ++(270:0.5cm);
   \draw[dotted] (0,-2cm)-- ++(315:0.5cm);
   \draw[dotted] (0,-2cm)-- ++(270:0.5cm);
   \draw[dotted] (0,-2cm)-- ++(225:0.5cm);
   \draw[dotted] (-1cm,-2cm)-- ++(270:0.5cm);
   \draw[dotted] (-2cm,-2cm)-- ++(225:0.5cm);
   \draw[dotted] (-2cm,-1cm)-- ++(180:0.5cm);
   \draw[dotted] (-2cm,0)-- ++(135:0.5cm);
   \draw[dotted] (-2cm,0)-- ++(180:0.5cm);
   \draw[dotted] (-2cm,0)-- ++(225:0.5cm);
   \draw[dotted] (-2cm,1cm)-- ++(180:0.5cm);

   \draw[->,red] (0,0)-- ++(270:1cm)--++(0:0.05cm)--++(90:1cm)-- ++(135:1.41cm);
   \coordinate (o) at (0,0);
   \coordinate (-o_1) at (0,-1cm);
   \fill[red](o) circle (1pt);
   \fill[red](-o_1) circle (1pt);
   
   \draw[left] (0,0) node {\tiny{$\mathfrak{o}$}};
   \draw[left] (0,2cm) node {\tiny{$\epsilon_1$}};
   \draw[above] (2cm,0) node {\tiny{$\epsilon_2$}};
   \fill (0,2cm) circle (1pt);
   \fill (2cm,0) circle (1pt);
 
   \end{tikzpicture}

  \end{center}
We now calculate the contribution of $T$ (resp. $\delta_{T}$) to the formulas:\\
\textbf{Gaussent-Littelmann formula}\\
Consider the gallery $(\mathfrak{o}\subset (E_0)_{\mathfrak{o}})$ associated to the last column of $T$ at the origin $\mathfrak{o}$, namely ${\tiny  \Yvcentermath1 \young({{\bar{1}}})}$, in the standard apartment of the residue building at $\mathfrak{o}$:
\begin{center}

	\begin{tikzpicture}[x=0.4cm,y=0.4cm]
   
   \draw[->] (0,0)-- (0:2cm);   
   \draw[->] (0,0)-- (45:2.83cm); 

   \draw[->] (0,0)-- (90:2cm);
   \draw[->] (0,0)-- (135:2.83cm);
   \draw[->] (0,0)-- (180:2cm);
   \draw[->] (0,0)-- (225:2.83cm);
   \draw[->] (0,0)-- (270:2cm);
   \draw[->] (0,0)-- (315:2.83cm);
   
   \fill[red] (0,0) circle (1pt);
   \draw[->,red] (0,0)-- ++(270:2cm);
   \draw[red, left] (0,0) node {\tiny{$\mathfrak{o}$}};
   \draw[red, right] (0,-1cm) node {\tiny{$(E_0)_{\mathfrak{o}}$}};
   \draw[right] (-1.6cm,-1.6cm) node {\small{$C_{\mathfrak{o}}^-=D_0$}};
   
  \end{tikzpicture} 

\end{center}
Thus, $w_{D_0}=id$ and 
$$c(E_0)=q^0=1.$$
Now consider the gallery $((E_0)_{V_1} \supset V_1 \subset (E_1)_{V_1})$ associated to the 2-column Young tableau ${\tiny \Yvcentermath1 \young(1{{\bar{1}}})}$ at $V_1$, in the standard apartment of the residue building at $V_1$:   
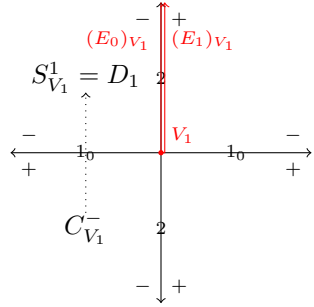
\begin{center}
	
   	\begin{tikzpicture}[x=0.4cm,y=0.4cm]
    
    \draw[->] (0,0)-- ++(0:2cm);
    \draw (1cm,0) node {\tiny{$1_0$}};
    \draw[->] (0,0)-- ++(90:2cm);
    \draw (0,1cm) node {\tiny{$2$}}; 
    \draw[->] (0,0)-- ++(180:2cm);
    \draw (-1cm,0) node {\tiny{$1_0$}};
    \draw[->] (0,0)-- ++(270:2cm);
    \draw (0,-1cm) node {\tiny{$2$}};

    \draw[red,->] (0,2cm)-- ++(270:2cm)-- ++(0:0.05cm)-- ++(90:2cm);
    
    \fill[red] (0,0) circle (1pt); 
    \draw[red, above right] (0,0) node {\tiny{$V_1$}};
    
    \draw (-1cm,-1cm) node {\small{$C_{V_1}^-$}};
    \draw (-1cm,1cm) node {\small{$S_{V_1}^1=D_1$}};
    
    \draw[above right] (-2cm,0) node {\tiny{$-$}};
    \draw[below right] (-2cm,0) node {\tiny{$+$}};
    \draw[above left] (2cm,0) node {\tiny{$-$}};
    \draw[below left] (2cm,0) node {\tiny{$+$}};
    \draw[above right] (0,-2cm) node {\tiny{$+$}};
    \draw[above left] (0,-2cm) node {\tiny{$-$}};
    \draw[below right] (0,2cm) node {\tiny{$+$}};
    \draw[below left] (0,2cm) node {\tiny{$-$}};
    \draw[red,left] (0,1.5cm) node {\tiny{$(E_0)_{V_1}$}};
    \draw[red,right] (0,1.5cm) node {\tiny{$(E_1)_{V_1}$}};
    
    \draw[dotted,->] (-1cm,-0.8cm)--(-1cm,0.8cm);

    \end{tikzpicture}

\end{center}   

The labeling of the walls and the signs at the walls are as explained in the example for type $A_2$.\\
The reduced expression for the Weyl group element $w \in W_{V_1}^a=W_{V_{\omega_1}}^a$ that sends $C_{V_1}^-$ to $D_1$ is $s_{1_0}$ indicated in the picture above by the dotted line. There is only one gallery of chambers $\textbf{c}_1$ in $\Gamma_{S_{V_1}^1}^+((1_0),op)$ which is illustrated below:
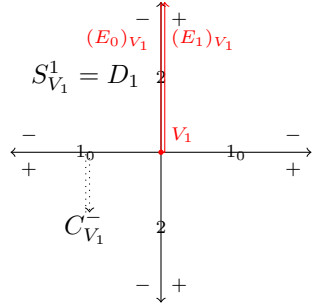
\begin{center}
	
   	\begin{tikzpicture}[x=0.4cm,y=0.4cm]
    
    \draw[->] (0,0)-- ++(0:2cm);
    \draw (1cm,0) node {\tiny{$1_0$}};
    \draw[->] (0,0)-- ++(90:2cm);
    \draw (0,1cm) node {\tiny{$2$}}; 
    \draw[->] (0,0)-- ++(180:2cm);
    \draw (-1cm,0) node {\tiny{$1_0$}};
    \draw[->] (0,0)-- ++(270:2cm);
    \draw (0,-1cm) node {\tiny{$2$}};

    \draw[red,->] (0,2cm)-- ++(270:2cm)-- ++(0:0.05cm)-- ++(90:2cm);
    
    \fill[red] (0,0) circle (1pt); 
    
    \draw (-1cm,-1cm) node {\small{$C_{V_1}^-$}};
    \draw (-1cm,1cm) node {\small{$S_{V_1}^1=D_1$}};
    
    \draw[red, above right] (0,0) node {\tiny{$V_1$}};
    
    \draw[above right] (-2cm,0) node {\tiny{$-$}};
    \draw[below right] (-2cm,0) node {\tiny{$+$}};
    \draw[above left] (2cm,0) node {\tiny{$-$}};
    \draw[below left] (2cm,0) node {\tiny{$+$}};
    \draw[above right] (0,-2cm) node {\tiny{$+$}};
    \draw[above left] (0,-2cm) node {\tiny{$-$}};
    \draw[below right] (0,2cm) node {\tiny{$+$}};
    \draw[below left] (0,2cm) node {\tiny{$-$}};
    \draw[red,left] (0,1.5cm) node {\tiny{$(E_0)_{V_1}$}};
    \draw[red,right] (0,1.5cm) node {\tiny{$(E_1)_{V_1}$}};
    
    \draw[dotted,->] (-1cm,-0.8cm)--(-1cm,0)-- ++(0:0.05cm)-- ++(270:0.8cm);

    \end{tikzpicture}

\end{center}  
The gallery of chambers $\textbf{c}_1$ has one positive folding. We obtain 
$$ c(((E_0)_{V_1} \supset V_1 \subset (E_1)_{V_1}))=(q-1).$$
Now consider the gallery $((E_1)_{V_2} \supset V_2 \subset (E_2)_{V_2})$  associated to the 2-column Young tableau $T= \tiny{\Yvcentermath1 \young(11,{{\bar{2}}})}$ at $V_2$ in the standard apartment of the residue building at $V_2$:

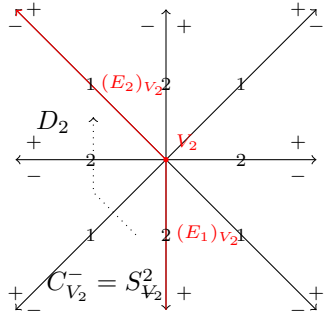
\begin{center}

	\begin{tikzpicture}[x=0.4cm,y=0.4cm]
   
   \draw[->] (0,0)-- (0:2cm);   
   \draw[->] (0,0)-- (45:2.83cm); 
   \draw[->] (0,0)-- (90:2cm);
   \draw[->] (0,0)-- (135:2.83cm);
   \draw[->] (0,0)-- (180:2cm);
   \draw[->] (0,0)-- (225:2.83cm);
   \draw[->] (0,0)-- (270:2cm);
   \draw[->] (0,0)-- (315:2.83cm);
   
   \fill[red] (0,0) circle (1pt);
   \draw[->,red] (0,-2cm)--(0,0)-- ++(135:2.82cm);
   \draw[red, above right] (0,0) node {\tiny{$V_2$}};
   \draw[red, right] (0,-1cm) node {\tiny{$(E_1)_{V_2}$}};
   \draw[red, right] (-1cm,1cm) node {\tiny{$(E_2)_{V_2}$}};
   \draw[below] (-0.8cm,-1.3cm) node {\small{$C_{V_2}^-=S^2_{V_2}$}};
   \draw (-1.5cm,0.5cm) node {\small{$D_2$}};
   
   \draw (0,2cm) node[below left] {\tiny{$-$}};
   \draw (0,2cm) node[below right] {\tiny{$+$}};
   \draw (2cm,2cm) node[left] {\tiny{$+$}};
   \draw (2cm,2cm) node[below] {\tiny{$-$}};
   \draw (2cm,0) node[above left] {\tiny{$+$}};
   \draw (2cm,0) node[below left] {\tiny{$-$}};
   \draw (2cm,-2cm) node[above] {\tiny{$+$}};
   \draw (2cm,-2cm) node[left] {\tiny{$-$}};
   \draw (0,-2cm) node[above left] {\tiny{$-$}};
   \draw (0,-2cm) node[above right] {\tiny{$+$}};
   \draw (-2cm,-2cm) node[above] {\tiny{$+$}};
   \draw (-2cm,-2cm) node[right] {\tiny{$-$}};
   \draw (-2cm,0) node[above right] {\tiny{$+$}};
   \draw (-2cm,0) node[below right] {\tiny{$-$}};
   \draw (-2cm,2cm) node[below] {\tiny{$-$}};
   \draw (-2cm,2cm) node[right] {\tiny{$+$}};

   \draw (0,-1cm) node {\tiny{$2$}};
   \draw (1cm,-1cm) node {\tiny{$1$}};
   \draw (1cm,0) node {\tiny{$2$}};
   \draw (1cm,1cm) node {\tiny{$1$}};
   \draw (0,1cm) node {\tiny{$2$}};
   \draw (-1cm,1cm) node {\tiny{$1$}};
   \draw (-1cm,0) node {\tiny{$2$}};
   \draw (-1cm,-1cm) node {\tiny{$1$}};
   
   \draw[dotted, ->] (-0.4cm,-1cm)--++ (135:0.8cm)-- ++(90:1cm);
   
  \end{tikzpicture} 

\end{center}
The reduced expression for the Weyl group element $w \in W_{V_2}^a$ that sends $C_{V_2}^-$ to $D_2$ is $s_1s_2$ indicated above by the dotted line. The gallery of chambers of type $(1,2)$ starting in $C_{V_2}^-$ that goes straight to $D_2$ (i.e. has only wall-crossings)  has only positive wall-crossings with respect to $S_{V_2}^2$. Consequently this is the only gallery of chambers in $\Gamma_{S^2_{V_2}}^+((1,2),op)$. We obtain
$$ c(((E_1)_{V_2} \supset V_2 \subset (E_2)_{V_2}))=q^2.$$
Together with the previous results we have
$$ c(\delta_T)=1*(q-1)*q^2.$$
\textbf{Combinatorial Gaussent-Littelmann formula} \text{}\\
First consider the last column of the Young tableau $T$, namely $C_2= \tiny{\Yvcentermath1 \young({{\bar{1}}})}$. There is no simple reflection in $W$ that increases $C_2$. We obtain
$$ c(C_2)=q^0=1.$$
Now consider the 2-column Young tableau $(C_1,C_2)=\tiny{\Yvcentermath1 \young(1{{\bar{1}}})}$ at vertex $V$. Since $i=1$ is smaller than $n=2$ and $r-i=2-1=1$ is odd and since $C_1$ and $C_2$ both consist of one box the 2-column Young tableau $\tiny{\Yvcentermath1 \young(1{{\bar{1}}})}$ is at a vertex $V$ of the same type as $V_{\omega_1}$. In fact, the Weyl group $W_V^v$ equals $W_{V_{\omega_1}}^v$. Applying the algorithm yields the following tree:

\begin{center}
    
    \begin{tikzpicture}[x=0.4cm,y=0.4cm]
    
    \draw (0,0) node  {$ \Yvcentermath1 \young(1{{\bar{1}}})$};
    \draw[->] (1cm,0)-- ++(0:1cm);
    \draw (1.5cm,0.2cm) node {\tiny{$id_{1_0}^+$}};
    \draw (3cm,0) node {$ \Yvcentermath1 \young({{\bar{1}}}{{\bar{1}}})$};

   \end{tikzpicture}.
   
\end{center}
It follows that 
$$ c(C_1,C_2)=(q-1).$$
It remains to calcuate $c(C_0,C_1)$. Consider the 2-column Young tableau $(C_0,C_1)=\tiny{\Yvcentermath1 \young(11,{{\bar{2}}})}$. Since the two columns consist of an unequal number of boxes it is at a special vertex. Applying the algorithm yields
\begin{center}

    \begin{tikzpicture}[x=0.4cm,y=0.4cm]
    
    \draw (0,0) node  {$ \Yvcentermath1 \young(11,{{\bar{2}}})$};
    \draw[->] (1cm,0)-- ++(0:1cm);
    \draw (1.5cm,0.2cm) node {\tiny{$s_1^+$}};
    \draw (3cm,0) node {$ \Yvcentermath1 \young(22,{{\bar{1}}})$};
     \draw[->] (4cm,0)-- ++(0:1cm);
    \draw (4.5cm,0.2cm) node {\tiny{$s_2^+$}};
    \draw (6cm,0) node {$ \Yvcentermath1 \young({{\bar{2}}}{{\bar{2}}},{{\bar{1}}})$};

   \end{tikzpicture}.

\end{center}
It follows
$$ c(C_0,C_1)=q^2$$
and in the whole
$$ c(T)=1*(q-1)*q^2.$$

\subsection{Type $C_2$}\text{}\\
For the semistandard Young tableau $T=(C_0,C_1,C_2)=\tiny{\Yvcentermath1 \young(12{{\bar{2}}},{{\bar{2}}}{{\bar{1}}})}$ the corresponding positively folded combinatorial one-skeleton gallery in the standard apartment of the affine building of type $C_2$ is $\delta_T=\delta_{s_2s_1(E_{\omega_1})}*\delta_{s_1s_2(E_{\omega_2})}*\delta_{s_2(E_{\omega_2})}=(\mathfrak{o}\subset E_0 \supset V_1 \subset E_1 \supset V_2 \subset E_2 \supset V_3 )$:
	
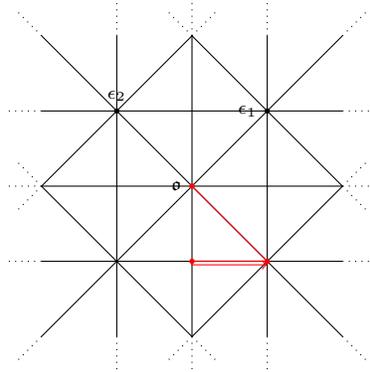
\begin{center}
	
   	\begin{tikzpicture}[x=0.4cm,y=0.4cm]
   
   \draw (0,0)-- (0:2cm);   
   \draw (0,0)-- (45:2.83cm); 
   \draw (0,0)-- (90:2cm);
   \draw (0,0)-- (135:2.83cm);
   \draw (0,0)-- (180:2cm);
   \draw (0,0)-- (225:2.83cm);
   \draw (0,0)-- (270:2cm);
   \draw (0,0)-- (315:2.83cm);
   \draw (0,2cm)-- ++(315:2.83cm)-- ++(225:2.83cm)-- ++(135:2.83cm)-- ++(45:2.83cm);7
   \draw (-1cm,2cm)-- ++(270:4cm);
   \draw (-2cm,-1cm)-- ++(0:4cm);
   \draw (1cm,2cm)-- ++(270:4cm);
   \draw (-2cm,1cm)-- ++(0:4cm);
   \draw[dotted] (-2cm,2cm)-- ++(135:0.5cm);
   \draw[dotted] (-1cm,2cm)-- ++(90:0.5cm);
   \draw[dotted] (0,2cm)-- ++(135:0.5cm);
   \draw[dotted] (0,2cm)-- ++(45:0.5cm);
   \draw[dotted] (0,2cm)-- ++(90:0.5cm);
   \draw[dotted] (1cm,2cm)-- ++(90:0.5cm);
   \draw[dotted] (2cm,2cm)-- ++(45:0.5cm);
   \draw[dotted] (2cm,1cm)-- ++(0:0.5cm);
   \draw[dotted] (2cm,0)-- ++(45:0.5cm);
   \draw[dotted] (2cm,0)-- ++(0:0.5cm);
   \draw[dotted] (2cm,0)-- ++(315:0.5cm);
   \draw[dotted] (2cm,-1cm)-- ++(0:0.5cm);
   \draw[dotted] (2cm,-2cm)-- ++(315:0.5cm);
   \draw[dotted] (1cm,-2cm)-- ++(270:0.5cm);
   \draw[dotted] (0,-2cm)-- ++(315:0.5cm);
   \draw[dotted] (0,-2cm)-- ++(270:0.5cm);
   \draw[dotted] (0,-2cm)-- ++(225:0.5cm);
   \draw[dotted] (-1cm,-2cm)-- ++(270:0.5cm);
   \draw[dotted] (-2cm,-2cm)-- ++(225:0.5cm);
   \draw[dotted] (-2cm,-1cm)-- ++(180:0.5cm);
   \draw[dotted] (-2cm,0)-- ++(135:0.5cm);
   \draw[dotted] (-2cm,0)-- ++(180:0.5cm);
   \draw[dotted] (-2cm,0)-- ++(225:0.5cm);
   \draw[dotted] (-2cm,1cm)-- ++(180:0.5cm);

   \draw[->,red] (0,0)-- ++(315:1.41cm)--++(180:1cm)--++(270:0.05cm)-- ++(0:1cm);
   \coordinate (o) at (0,0);
   \fill[red](o) circle (1pt);
   \fill[red] (1cm,-1cm) circle (1pt);
   \fill[red] (0,-1cm) circle (1pt);
   
   \draw[left] (0,0) node {\tiny{$\mathfrak{o}$}};
   \draw[left] (1cm,1cm) node {\tiny{$\epsilon_1$}};
   \draw[above] (-1cm,1cm) node {\tiny{$\epsilon_2$}};
   \fill (1cm,1cm) circle (1pt);
   \fill (-1cm,1cm) circle (1pt);
 
   \end{tikzpicture}

  \end{center}

We now calculate the contribution of $T$ (resp. $\delta_T$) to the formulas:\\
\textbf{Gaussent-Littelmann formula}\\
Consider the gallery $(\mathfrak{o}\subset (E_0)_{\mathfrak{o}})$ associated to the last column of $T$ at the origin $\mathfrak{o}$, namely ${\tiny  \Yvcentermath1 \young({{\bar{2}}})}$, in the standard apartment of the residue building at $\mathfrak{o}$:
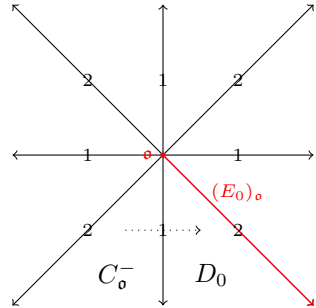
\begin{center}

	\begin{tikzpicture}[x=0.4cm,y=0.4cm]
   
   \draw[->] (0,0)-- (0:2cm);   
   \draw[->] (0,0)-- (45:2.83cm); 
   \draw[->] (0,0)-- (90:2cm);
   \draw[->] (0,0)-- (135:2.83cm);
   \draw[->] (0,0)-- (180:2cm);
   \draw[->] (0,0)-- (225:2.83cm);
   \draw[->] (0,0)-- (270:2cm);
   \draw[->] (0,0)-- (315:2.83cm);
   
   \fill[red] (0,0) circle (1pt);
   \draw[->,red] (0,0)--(2cm,-2cm);
   \draw[red, left] (0,0) node {\tiny{$\mathfrak{o}$}};
   \draw[red] (1cm,-0.5cm) node {\tiny{$(E_0)_{\mathfrak{o}}$}};
   \draw[right] (-1cm,-1.6cm) node {\small{$C_{\mathfrak{o}}^-$}};
   \draw[left] (1cm,-1.6cm) node {\small{$D_0$}};
   
   \draw (0,-1cm) node {\tiny{$1$}};
   \draw (1cm,-1cm) node {\tiny{$2$}};
   \draw (1cm,0) node {\tiny{$1$}};
   \draw (1cm,1cm) node {\tiny{$2$}};
   \draw (0,1cm) node {\tiny{$1$}};
   \draw (-1cm,1cm) node {\tiny{$2$}};
   \draw (-1cm,0) node {\tiny{$1$}};
   \draw (-1cm,-1cm) node {\tiny{$2$}};
   
   \draw[dotted, ->]  (-0.5cm,-1cm)--(0.5cm,-1cm);

  \end{tikzpicture} 

\end{center}
The labeling of the walls is explained in the example of type $A_2$.\\
The Weyl group element $w \in W$ that sends $C_{\mathfrak{o}}^-$ to $D_0$ is $s_1$ indicated above by the dotted line. It follows that
$$c(E_0)=q^1=q.$$ 
Now consider the gallery $((E_0)_{V_1} \supset V_1 \subset (E_1)_{V_1})$ associated to the 2-column Young tableau ${\tiny \Yvcentermath1 \young(2{{\bar{2}}},{{\bar{1}}})}$ at $V_1$, in the standard apartment of the residue building at $V_1$:   
   
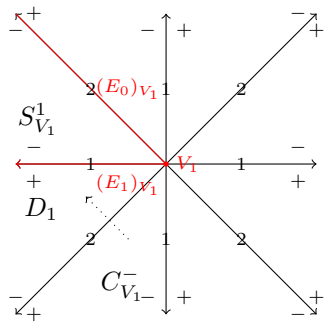
\begin{center}

	\begin{tikzpicture}[x=0.4cm,y=0.4cm]
   
   \draw[->] (0,0)-- (0:2cm);   
   \draw[->] (0,0)-- (45:2.83cm); 
   \draw[->] (0,0)-- (90:2cm);
   \draw[->] (0,0)-- (135:2.83cm);
   \draw[->] (0,0)-- (180:2cm);
   \draw[->] (0,0)-- (225:2.83cm);
   \draw[->] (0,0)-- (270:2cm);
   \draw[->] (0,0)-- (315:2.83cm);
   
   \fill[red] (0,0) circle (1pt);
   \draw[->,red] (-2cm,2cm)--(0,0)--(-2cm,0);
   \draw[red, right] (0,0) node {\tiny{$V_1$}};
   \draw[red,] (-0.5cm,1cm) node {\tiny{$(E_0)_{V_1}$}};
   \draw[red, below] (-0.5cm,0) node {\tiny{$(E_1)_{V_1}$}};
   
   \draw[right] (-1cm,-1.6cm) node {\small{$C_{V_1}^-$}};
   \draw[left] (-1.3cm,-0.6cm) node {\small{$D_1$}};
   \draw[left] (-1.3cm,+0.6cm) node {\small{$S^1_{V_1}$}};
   
   \draw (0,-1cm) node {\tiny{$1$}};
   \draw (1cm,-1cm) node {\tiny{$2$}};
   \draw (1cm,0) node {\tiny{$1$}};
   \draw (1cm,1cm) node {\tiny{$2$}};
   \draw (0,1cm) node {\tiny{$1$}};
   \draw (-1cm,1cm) node {\tiny{$2$}};
   \draw (-1cm,0) node {\tiny{$1$}};
   \draw (-1cm,-1cm) node {\tiny{$2$}};
   
   \draw[dotted, ->]  (-0.5cm,-1cm)-- ++(135:0.8cm);
   
   \draw (0,2cm) node[below left] {\tiny{$-$}};
   \draw (0,2cm) node[below right] {\tiny{$+$}};
   \draw (2cm,2cm) node[left] {\tiny{$-$}};
   \draw (2cm,2cm) node[below] {\tiny{$+$}};
   \draw (2cm,0) node[above left] {\tiny{$-$}};
   \draw (2cm,0) node[below left] {\tiny{$+$}};
   \draw (2cm,-2cm) node[above] {\tiny{$+$}};
   \draw (2cm,-2cm) node[left] {\tiny{$-$}};
   \draw (0,-2cm) node[above left] {\tiny{$-$}};
   \draw (0,-2cm) node[above right] {\tiny{$+$}};
   \draw (-2cm,-2cm) node[above] {\tiny{$-$}};
   \draw (-2cm,-2cm) node[right] {\tiny{$+$}};
   \draw (-2cm,0) node[above right] {\tiny{$-$}};
   \draw (-2cm,0) node[below right] {\tiny{$+$}};
   \draw (-2cm,2cm) node[below] {\tiny{$-$}};
   \draw (-2cm,2cm) node[right] {\tiny{$+$}};

  \end{tikzpicture} 

\end{center}
The labeling of the walls and the signs at the walls are as explained in the example of type $A_2$.\\
The Weyl group element $s_2$ sends the chamber $C_{V_1}^-$ to $D_1$ (indicated by the dotted line). There is a single gallery of chambers $\textbf{c}_1$ in the set $\Gamma_{S^1_{V_1}}^+((2),op)$:
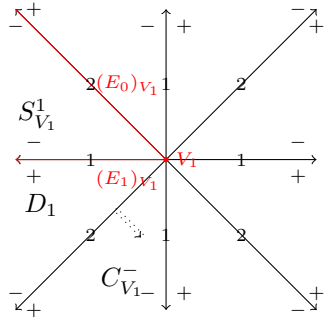
\begin{center}

	\begin{tikzpicture}[x=0.4cm,y=0.4cm]
   
   \draw[->] (0,0)-- (0:2cm);   
   \draw[->] (0,0)-- (45:2.83cm); 
   \draw[->] (0,0)-- (90:2cm);
   \draw[->] (0,0)-- (135:2.83cm);
   \draw[->] (0,0)-- (180:2cm);
   \draw[->] (0,0)-- (225:2.83cm);
   \draw[->] (0,0)-- (270:2cm);
   \draw[->] (0,0)-- (315:2.83cm);
   
   \fill[red] (0,0) circle (1pt);
   \draw[->,red] (-2cm,2cm)--(0,0)--(-2cm,0);
   \draw[red, right] (0,0) node {\tiny{$V_1$}};
   \draw[red,] (-0.5cm,1cm) node {\tiny{$(E_0)_{V_1}$}};
   \draw[red, below] (-0.5cm,0) node {\tiny{$(E_1)_{V_1}$}};
   
   \draw[right] (-1cm,-1.6cm) node {\small{$C_{V_1}^-$}};
   \draw[left] (-1.3cm,-0.6cm) node {\small{$D_1$}};
   \draw[left] (-1.3cm,+0.6cm) node {\small{$S^1_{V_1}$}};
   
   \draw (0,-1cm) node {\tiny{$1$}};
   \draw (1cm,-1cm) node {\tiny{$2$}};
   \draw (1cm,0) node {\tiny{$1$}};
   \draw (1cm,1cm) node {\tiny{$2$}};
   \draw (0,1cm) node {\tiny{$1$}};
   \draw (-1cm,1cm) node {\tiny{$2$}};
   \draw (-1cm,0) node {\tiny{$1$}};
   \draw (-1cm,-1cm) node {\tiny{$2$}};
   
   \draw[dotted, ->]  (-0.3cm,-1cm)-- ++(135:0.5cm)-- ++ (225:0.05cm)- ++(315:0.5cm);
   
   \draw (0,2cm) node[below left] {\tiny{$-$}};
   \draw (0,2cm) node[below right] {\tiny{$+$}};
   \draw (2cm,2cm) node[left] {\tiny{$-$}};
   \draw (2cm,2cm) node[below] {\tiny{$+$}};
   \draw (2cm,0) node[above left] {\tiny{$-$}};
   \draw (2cm,0) node[below left] {\tiny{$+$}};
   \draw (2cm,-2cm) node[above] {\tiny{$+$}};
   \draw (2cm,-2cm) node[left] {\tiny{$-$}};
   \draw (0,-2cm) node[above left] {\tiny{$-$}};
   \draw (0,-2cm) node[above right] {\tiny{$+$}};
   \draw (-2cm,-2cm) node[above] {\tiny{$-$}};
   \draw (-2cm,-2cm) node[right] {\tiny{$+$}};
   \draw (-2cm,0) node[above right] {\tiny{$-$}};
   \draw (-2cm,0) node[below right] {\tiny{$+$}};
   \draw (-2cm,2cm) node[below] {\tiny{$-$}};
   \draw (-2cm,2cm) node[right] {\tiny{$+$}};

  \end{tikzpicture} 

\end{center}
Since $\textbf{c}_1$ has one positive folding we obtain
$$ c(((E_0)_{V_1} \supset V_1 \subset (E_1)_{V_1}))=(q-1).$$
Now consider the gallery $((E_1)_{V_2} \supset V_2 \subset (E_2)_{V_2})$  associated to the 2-column Young tableau $T= \tiny{\Yvcentermath1 \young(12,{{\bar{2}}}{{\bar{1}}})}$ at $V_2$ in the standard apartment of the residue building at $V_2$:
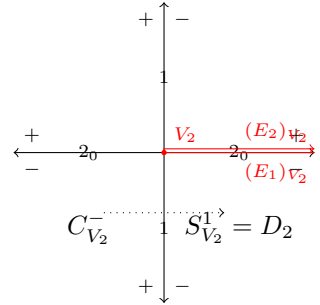
\begin{center}
	
   	\begin{tikzpicture}[x=0.4cm,y=0.4cm]
    
    \draw[->] (0,0)-- ++(0:2cm);
    \draw (1cm,0) node {\tiny{$2_0$}};
    \draw[->] (0,0)-- ++(90:2cm);
    \draw (0,1cm) node {\tiny{$1$}}; 
    \draw[->] (0,0)-- ++(180:2cm);
    \draw (-1cm,0) node {\tiny{$2_0$}};
    \draw[->] (0,0)-- ++(270:2cm);
    \draw (0,-1cm) node {\tiny{$1$}};

    \draw[red,->] (2cm,0)--(0,0)-- ++(90:0.05cm)-- ++(0:2cm);
    
    \fill[red] (0,0) circle (1pt); 
    \draw[red, above right] (0,0) node {\tiny{$V_2$}};
    
    \draw (-1cm,-1cm) node {\small{$C_{V_2}^-$}};
    \draw (1cm,-1cm) node {\small{$S_{V_2}^1=D_2$}};
    
    \draw[above right] (-2cm,0) node {\tiny{$+$}};
    \draw[below right] (-2cm,0) node {\tiny{$-$}};
    \draw[above left] (2cm,0) node {\tiny{$+$}};
    \draw[below left] (2cm,0) node {\tiny{$-$}};
    \draw[above right] (0,-2cm) node {\tiny{$-$}};
    \draw[above left] (0,-2cm) node {\tiny{$+$}};
    \draw[below right] (0,2cm) node {\tiny{$-$}};
    \draw[below left] (0,2cm) node {\tiny{$+$}};
    \draw[red,above] (1.5cm,0) node {\tiny{$(E_2)_{V_2}$}};
    \draw[red,below] (1.5cm,0) node {\tiny{$(E_1)_{V_2}$}};
    
    \draw[dotted,->] (-0.8cm,-0.8cm)--(0.8cm,-0.8cm);

    \end{tikzpicture}

\end{center}   

The labeling of the walls and the signs at the walls are as explained in the example for type $A_2$.\\
The reduced expression for the Weyl group element $w \in W_{V_2}^a$ that sends $C_{V_2}^-$ to $D_2$ is $s_{1}$ indicated in the picture above by the dotted line. There is only one gallery of chambers $\textbf{c}_1$ in $\Gamma_{S_{V_2}^2}^+((1),op)$ which is illustrated below:
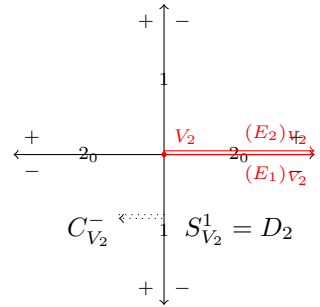
\begin{center}
	
   	\begin{tikzpicture}[x=0.4cm,y=0.4cm]
    
    \draw[->] (0,0)-- ++(0:2cm);
    \draw (1cm,0) node {\tiny{$2_0$}};
    \draw[->] (0,0)-- ++(90:2cm);
    \draw (0,1cm) node {\tiny{$1$}}; 
    \draw[->] (0,0)-- ++(180:2cm);
    \draw (-1cm,0) node {\tiny{$2_0$}};
    \draw[->] (0,0)-- ++(270:2cm);
    \draw (0,-1cm) node {\tiny{$1$}};

    \draw[red,->] (2cm,0)--(0,0)-- ++(90:0.05cm)-- ++(0:2cm);
    
    \fill[red] (0,0) circle (1pt); 
    \draw[red, above right] (0,0) node {\tiny{$V_2$}};
    
    \draw (-1cm,-1cm) node {\small{$C_{V_2}^-$}};
    \draw (1cm,-1cm) node {\small{$S_{V_2}^1=D_2$}};
    
    \draw[above right] (-2cm,0) node {\tiny{$+$}};
    \draw[below right] (-2cm,0) node {\tiny{$-$}};
    \draw[above left] (2cm,0) node {\tiny{$+$}};
    \draw[below left] (2cm,0) node {\tiny{$-$}};
    \draw[above right] (0,-2cm) node {\tiny{$-$}};
    \draw[above left] (0,-2cm) node {\tiny{$+$}};
    \draw[below right] (0,2cm) node {\tiny{$-$}};
    \draw[below left] (0,2cm) node {\tiny{$+$}};
    \draw[red,above] (1.5cm,0) node {\tiny{$(E_2)_{V_2}$}};
    \draw[red,below] (1.5cm,0) node {\tiny{$(E_1)_{V_2}$}};
    
    \draw[dotted,->] (-0.6cm,-0.8cm)--(0,-0.8cm)-- ++ (270:0.05cm)-- ++(180:0.6cm) ;

    \end{tikzpicture}
\end{center}
The gallery of chambers $\textbf{c}_1$ has one positive folding and we obtain:
$$c(((E_1)_{V_2} \supset V_2 \subset (E_2)_{V_2}))=(q-1)$$
and
$$ c(T)= q*(q-1)^2 .$$
\textbf{Combinatorial Gaussent-Littelmann formula}\\
First consider the last column of $T$, namely $C_2= \tiny{\Yvcentermath1 \young({{\bar{2}}})}$. Applying $s_1 \in W$ to $C_2$ increases it and there is no simple reflection that increases $s_1(\tiny{\Yvcentermath1 \young({{\bar{2}}})})=\tiny{\Yvcentermath1 \young({{\bar{1}}})}$ further. We obtain
$$ c(C_2)=q^1=q.$$
Now consider the 2-column Young tableau $(C_1,C_2)=\tiny{\Yvcentermath1 \young(2{{\bar{2}}},{{\bar{1}}})}$. Since the two columns consist of an unequal number of boxes the tableau is at a special vertex. Applying the algorithm yields
\begin{center}
    
    \begin{tikzpicture}[x=0.4cm,y=0.4cm]
    
    \draw (0,0) node  {$ \Yvcentermath1 \young(1{{\bar{2}}},{{\bar{2}}})$};
    \draw[->] (1cm,0)-- ++(0:1cm);
    \draw (1.5cm,0.2cm) node {\tiny{$id_{2}^+$}};
    \draw (3cm,0) node {$ \Yvcentermath1 \young({{\bar{2}}}{{\bar{2}}},{{\bar{1}}})$};

   \end{tikzpicture}.
   
\end{center}
It follows that 
$$ c(C_1,C_2)=(q-1).$$
It remains to calcuate $c(C_0,C_1)$. Consider the 2-column Young tableau $(C_0,C_1)=\tiny{\Yvcentermath1 \young(12,{{\bar{2}}}{{\bar{1}}})}$ at the vertex $V$. The Weyl group $W_V^v$ equals $W_{V_{\omega_2}}^v$. Applying the algorithm yields
\begin{center}
    
    \begin{tikzpicture}[x=0.4cm,y=0.4cm]
    
    \draw (0,0) node  {$ \Yvcentermath1 \young(12,{{\bar{2}}}{{\bar{1}}})$};
    \draw[->] (1cm,0)-- ++(0:1cm);
    \draw (1.5cm,0.2cm) node {\tiny{$id_{1}^+$}};
    \draw (3cm,0) node {$ \Yvcentermath1 \young(22,{{\bar{1}}}{{\bar{1}}})$};

   \end{tikzpicture}.
   
\end{center}
It follows
$$c(C_0,C_1)=(q-1)$$
and 
$$ c(T)= q*(q-1)^2.$$

\end{document}